\documentclass[a4paper,reqno,12pt]{amsart}
\usepackage{geometry}
\usepackage{amsmath,amsthm,amssymb, amscd, mathrsfs,amsfonts}
\usepackage{wasysym}

\usepackage[lite]{amsrefs}
\newcommand*{\MRref}[2]{ \href{http://www.ams.org/mathscinet-getitem?mr=#1}{MR \textbf{#1}}}
\newcommand*{\arxiv}[1]{\href{http://www.arxiv.org/abs/#1}{arXiv: #1}}
\renewcommand{\PrintDOI}[1]{\href{http://dx.doi.org/\detokenize{#1}}{doi: \detokenize{#1}}%
  \IfEmptyBibField{pages}{, (to appear in print)}{}}

\usepackage{eulervm}
\usepackage{times}

\usepackage{latexsym}

\DeclareMathAlphabet{\mathbbm}{U}{bbm}{m}{n}

\usepackage{hyperref}
\usepackage{array}

\usepackage{nicefrac}

\usepackage[all]{xy}

\def\commutatif{\ar@{}[rd]|{\circlearrowleft}}
 
\newcommand{\eq}[1][r]
   {\ar@<-3pt>@{-}[#1]
    \ar@<-1pt>@{}[#1]|<{}="gauche"
    \ar@<+0pt>@{}[#1]|-{}="milieu"
    \ar@<+1pt>@{}[#1]|>{}="droite"
    \ar@/^2pt/@{-}"gauche";"milieu"
    \ar@/_2pt/@{-}"milieu";"droite"}
    
    \def\dar[#1]{\ar@<2pt>[#1]\ar@<-2pt>[#1]}
  \entrymodifiers={!!<0pt,0.7ex>+} 

\newtheorem{thm}{Theorem}[section]
\newtheorem{pro}[thm]{Proposition}
\newtheorem{lem}[thm]{Lemma}
\newtheorem{cor}[thm]{Corollary}
\newtheorem{qst}[thm]{Question}

\theoremstyle{definition}{\normalfont}
\newtheorem{df}[thm]{Definition}
\newtheorem{dfpro}[thm]{Definition and Proposition}
\newtheorem{dflem}[thm]{Definition and Lemma}
\newtheorem{rem}[thm]{Remark}

\newtheorem{ex}[thm]{Example}

\newtheorem{nota}[thm]{Notations}
\allowdisplaybreaks

\newcommand{\cA}{{\mathcal A}}  
\newcommand{\cB}{{\mathcal B}}  
\newcommand{\cC}{{\mathcal C}}  
  
\newcommand{\cE}{{\mathcal E}}  
   
\newcommand{\cG}{{\mathcal G}}
\newcommand{\cH}{{\mathcal H}}  \newcommand{\sH}{{\mathscr H}}
\newcommand{\cI}{{\mathcal I}}

\newcommand{\cK}{{\mathcal K}}
\newcommand{\cL}{{\mathcal L}}  
\newcommand{\cM}{{\mathcal M}}

\newcommand{\cS}{{\mathcal S}}

\newcommand{\cU}{{\mathcal U}}  \newcommand{\sU}{{\mathscr{U}}}

\newcommand{\CC}{{\mathbb C}}
\newcommand{\EE}{{\mathbb E}}

\newcommand{\HH}{{\mathbb H}}
\newcommand{\KK}{{\mathbb K}}
\newcommand{\NN}{{\mathbb N}}

\newcommand{\PP}{{\mathbb P}}
\newcommand{\QQ}{{\mathbb Q}}
\newcommand{\RR}{{\mathbb R}}

\newcommand{\uc}{{\mathbb S}^1}

\newcommand{\ZZ}{{\mathbb Z}}

\newcommand{\bfS}{{\mathbf S}}
\newcommand{\bft}{{\mathbf t}}

\newcommand{\id}{{\mathbf 1}}

\newcommand{\wGa}{{\widetilde \Gamma}}

\newcommand{\wPU}{\what{\operatorname{PU}}}
\newcommand{\U}{\operatorname{U}}
\newcommand{\wU}{\what{\operatorname{U}}}

\newcommand{\wRExt}{{\what{\operatorname{\textrm{Ext}R}}}}

\newcommand{\wRBr}{{\what{\operatorname{\textrm{Br}R}}}}
\newcommand{\wBRO}{\what{\operatorname{\textrm{Br}O}}}

\newcommand{\fwRBr}{{\what{\mathfrak{BrR}}}}

\newcommand{\wK}{\what{\cK}}
\newcommand{\wKK}{\hat{\mathbb{K}}}
\newcommand{\wfK}{\hat{\mathfrak{K}}}

\newcommand{\wBr}{\what{{\text{\textrm{Br}}}}}
\newcommand{\fwBr}{\what{\mathfrak{Br}}}
\newcommand{\Inv}{{\operatorname{\mathsf{Inv}}}}
\newcommand{\range}{\operatorname{Im}}

\newcommand{\Gpdo}{\cG^{(0)}}
\newcommand{\grpd}{\xymatrix{\cG \dar[r] & X}}
\newcommand{\gamgpd}{\xymatrix{\Ga \dar[r] & Y}}

\newcommand{\Gamo}{\Ga^{(0)}}

\newcommand{\what}{\widehat}

\newcommand{\fr}{{\mathfrak r}}
\newcommand{\fs}{{\mathfrak s}}

\newcommand{\RG}{{\mathfrak {RG}}}

\newcommand{\frc}{\mathfrak{c}}

\newcommand{\To}{\longrightarrow}
\newcommand{\sto}{\rightarrow}
\newcommand{\mto}{\longmapsto}

\newcommand{\cstar}{C^{\ast}}
\newcommand{\Ga}{\Gamma}

\newcommand{\del}{\delta}
\newcommand{\al}{\alpha}

\newcommand{\vr}{\varrho}
\newcommand{\ve}{\varepsilon}
\newcommand{\vp}{\varphi}
\newcommand{\g}{\gamma}

\newcommand{\Om}{\Omega}
\newcommand{\om}{\omega}

\newcommand{\Hom}{\operatorname{Hom}}
\newcommand{\Isom}{\operatorname{Isom}}
\newcommand{\Aut}{\operatorname{\textsf{Aut}}}

\newcommand{\Cl}{\CC l}

\newcommand{\Co}{{\mathcal C}_0}

\def\<{\langle}
\def\>{\rangle}
\let\ipscriptstyle=\scriptscriptstyle
\def\lipsqueeze{{\mskip -3.0mu}}
\def\ripsqueeze{{\mskip -3.0mu}}
\def\ipcomma{\nobreak\mathrel{,}\nobreak}
\newbox\ipstrutbox
\setbox\ipstrutbox=\hbox{\vrule height8.5pt
width 0pt}
\def\ipstrut{\copy\ipstrutbox}
\def\lip#1<#2,#3>{\mathopen{\relax_{\ipstrut\ipscriptstyle{
#1}}\lipsqueeze
\langle} #2\ipcomma #3 \rangle}
\def\blip#1<#2,#3>{\mathopen{\relax_{\ipstrut
\ipscriptstyle{ #1}}\lipsqueeze\bigl\langle} #2\ipcomma #3 \bigr\rangle}
\def\rip#1<#2,#3>{\langle #2\ipcomma #3
\rangle_{\ripsqueeze\ipstrut\ipscriptstyle{#1}}}
\def\brip#1<#2,#3>{\bigl\langle #2\ipcomma #3
\bigr\rangle_{\ripsqueeze\ipstrut\ipscriptstyle{#1}}}
\def\angsqueeze{\mskip -6mu}
\def\smangsqueeze{\mskip -3.7mu}
\def\trip#1<#2,#3>{\langle\smangsqueeze\langle #2\ipcomma #3
\rangle\smangsqueeze\rangle_{\ripsqueeze\ipstrut\ipscriptstyle{#1}}}
\def\btrip#1<#2,#3>{\bigl\langle\angsqueeze\bigl\langle #2\ipcomma
#3
\bigr\rangle
\angsqueeze\bigr\rangle_{\ripsqueeze\ipstrut\ipscriptstyle{#1}}}
\def\tlip#1<#2,#3>{\mathopen{\relax_{\ipstrut\ipscriptstyle{
#1}}\lipsqueeze \langle\smangsqueeze\langle} #2\ipcomma #3
\rangle\smangsqueeze\rangle}
\def\btlip#1<#2,#3>{\mathopen{\relax_{\ipstrut\ipscriptstyle{
#1}}\lipsqueeze
\bigl\langle\angsqueeze\bigl\langle} #2\ipcomma #3
\bigr\rangle\angsqueeze\bigr\rangle}

\def\ip(#1|#2){(#1\mid #2)}
\def\bip(#1|#2){\bigl(#1 \mid #2\bigr)}
\def\Bip(#1|#2){\Bigl( #1 \bigm| #2 \Bigr)}
\def\h[#1,#2]{[#1,#2]_{H}}

\newcommand{\Id}{\text{\normalfont Id}}

\newcommand{\mydot}{\mathbin{:}}

\def\ipp(#1|#2){\ip({#1}|{#2})_{\pi}}

\hypersetup{
 colorlinks = true,
 urlcolor = blue, 
 linkcolor = blue,
 citecolor = red,
 pdfauthor = {Moutuou, E.M.},
 pdfkeywords = {Graded Brauer groups, Real groupoids, groupoid extensions},
 pdftitle = {Graded Brauer groupoids of groupoids with involutions},
 pdfpagemode = UseNone
}

\title{Graded Brauer groups of a groupoid with involution}
\author{El-ka\"ioum M. Moutuou}

\address{LMAM, Universit\'e de Lorraine - Metz, CNRS UMR 7122, Ile du Saulcy,
F-57045 Metz Cedex 1, France
}
\curraddr{School of Mathematical Sciences, University of Southampton, Highfield, Southampton SO17 1BJ, UK}
\email{E.MohamedMoutuou@soton.ac.uk}

\begin{document}

\begin{abstract}
We define a group $\wRBr(\cG)$ containing, in a sense, the graded complex and orthogonal Brauer groups of a locally compact groupoid $\cG$ equipped with an involution. When the involution is trivial, we show that the new group naturally provides a generalisation of Donovan-Karoubi's graded orthogonal Brauer group $GBrO$. More generally, it is shown to be a direct summand of the well-known graded complex Brauer goup. In addition, we prove that $\wRBr(\cG)$ identifies with a direct sum of a Real cohomology group and the abelian group $\wRExt(\cG,\uc)$ of Real graded $\uc$-central extensions. A cohomological picture is then given.  
\end{abstract}

\maketitle



\section*{Introduction}

The idea of working with $\ZZ_2$-graded real $\cstar$-algebras~\cite{Li:Real_algebras} as if they were complex ones first emerged in Kasparov's founding paper~\cite{Kasparov:Operator_K} of bivariant $K$-theory. The trick merely consists of "\emph{complexifying}" a given graded real $\cstar$-algebra; that is, considering the complex $\cstar$-algebra $A_\CC:=A\otimes_\RR\CC$ together with the induced $\ZZ_2$-grading. The latter admits the obvious conjugate-linear involution $a\otimes_\RR\lambda \mto a\otimes_\RR\bar{\lambda}$. Conversely, any $\ZZ_2$-graded complex $\cstar$-algebra $B$ admitting a conjugate linear involution $\sigma$ is necessarily the complexification of a graded real $\cstar$-algebra $B_\RR$, which identifies with the fixed points of $\sigma$. It follows that "complexification" defines an equivalence from the category of $\ZZ_2$-graded real $\cstar$-algebras to the category of $\ZZ_2$-graded complex $\cstar$-algebras endowed with conjugate-linear involutions (also called \emph{Real involutions} or Real structures in the literature~\cite{Kasparov:Operator_K}). The inverse functor is "\emph{realification}"; that is, taking the fixed point set of the involution. In fact, working with the complexified algebra instead of the original real one is useful especially when it comes to discuss functional calculus. However, that equivalence of categories no longer holds when the $\cstar$-algebras are acted upon by topological groupoids endowed with Real structures~\cite{Moutuou:Real.Cohomology}. By an action of a Real groupoid $(\cG,\tau)$ on a $\ZZ_2$-graded $\cstar$-algebra $A$ equipped with a Real involution $\sigma$ we mean an action $\al=(\al_g)_{g\in \cG}$ of $\cG$ by ${}^\ast$-automorphisms on $A$~\cite{Kumjian-Muhly-Renault-Williams:Brauer} such that for all $g\in \cG, a\in A_{s(g)}, \al_{\tau(g)}(\sigma(a))=\sigma(\al_g(a))$, and $\al_g:A_{s(g)}\To A_{r(g)}$ is an isomorphism of $\ZZ_2$-graded $\cstar$-algebras. The reason of such a failure is that an action of $\cG$ on $A_\RR$ does not extends to an action of $(\cG,\tau)$ on $A$ satisfying the condition mentioned above, unless the involution $\tau:\cG \To \cG$ is trivial.

In the present paper, we are dealing with stable continuous-trace $\ZZ_2$-graded $\cstar$-algebras $A$~\cite{Parker:Brauer,Rosenberg:Continuous-trace} endowed with Real involution  and acted upon by a Real groupoid $(\cG,\tau)$. Forgetting the involutions, it is known that~\cite{Kumjian-Muhly-Renault-Williams:Brauer,Tu:Twisted_Poincare} giving such $\cstar$-algebras is equivalent to giving $\ZZ_2$-graded Dixmier-Douady bundles $\cA$ over $\cG$; that is, a $\ZZ_2$-graded elementary $\cstar$-bundle $\cA\To \Gpdo$ satisfying Fell's condition, together with a family of $\ZZ_2$-graded ${}^\ast$-isomorphisms $\al_g:\cA_{s(g)}\To \cA_{r(g)}$ such that $\al_{gh}=\al_g\al_h$ whenever the product makes sense and $\al_g^{-1}=\al_{g^{-1}}$. The \emph{graded Brauer group} $\wBr(\cG)$~\cite{Tu:Twisted_Poincare} of $\cG$ is defined from Morita equivalence classes of such bundles, or equivalently from stable continuous-trace $\ZZ_2$-graded $\cstar$-algebras equipped with a $\cG$--action. If $\cG$ is a transformation groupoid $X\rtimes G$, where $G$ is a locally compact group acting on $X$, then $\wBr(\cG)=\wBr_G(X)$ is nothing but the equivariant analogue of Parker's~\cite{Parker:Brauer}, and the graded analogue of Crocker-Kumjian-Raeburn-Williams~\cite{Crocker-Kumjian-Raeburn-Williams:Brauer}. It is shown in~\cite{Tu:Twisted_Poincare} that if $\cG$ is locally compact second-countable and Hausdorff, then 
$$\wBr(\cG)\cong \check{H}^0(\cG_\bullet,\ZZ_2)\times \check{H}^1(\cG_\bullet,\ZZ_2)\times \check{H}^2(\cG_\bullet,\uc).$$  

Note that $\wBr$ generalises Donovan-Karoubi's $GBrU$~\cite{Donovan-Karoubi} and Parker's $GBr^\infty$. Roughly speaking, a graded Dixmier-Douady bundle $\cA\in \wBr(\cG)$ is of \emph{parity $0$} (resp. of parity $1$) if it has typical fiber $\KK(\hat{\cH})$ (resp. $\KK(\cH)\oplus \KK(\cH)$), where $\hat{\cH}$ (resp. $\cH$) is a graded complex separable Hilbert space (resp. is a complex separable Hilbert space)~\cite{Tu:Twisted_Poincare,Freed-Hopkins:Twisted_K2}. Noticing that $\KK(\cH)\oplus \KK(\cH)\cong \KK(\cH)\hat{\otimes}\Cl_1$, the isomorphism above implies that if the base space of $\cG$ is connected, then $\wBr(\cG)$ is a $\ZZ_2$-graded group. 

Instead of simply generalising Donovan-Karoubi's graded orthogonal Brauer group $GBrO$ to groupoids, we are going further. More precisely, we introduce a new group $\wRBr$, which enables us to study graded complex and real Dixmier-Douady bundles simultaneously. We start with a locally compact Hausdorff second-countable Real groupoid $(\cG,\tau)$ with a Haar system~\cite[\S2]{Moutuou:Real.Cohomology}, and define $\wRBr(\cG)$ as the set of Morita equivalence classes of \emph{Real graded Dixmier-Douady bundles} over $(\cG,\tau)$; \emph{i.e.}, graded Dixmier-Douady bundles $\cA$ that come equipped with Real structures satisfying some relevant relations. We shall note that we introduced in~\cite{Moutuou:Twistings} a group $\wRBr_\ast(\cG)$ which actually is but the subgroup of $\wRBr(\cG)$ consisting of Real graded Dixmier-Douady bundles (that we called $B$-fields in {\it loc. cit.}) that locally look like a graded elementary complex $\cstar$-algebra $\wKK$ endowed with a Real involution. We have shown that such bundles are, up to Morita equivalence, of eight types. Thus, for a Real groupoid $(\cG,\tau)$ with connected base space, $\wRBr_\ast(\cG)$ is a $\ZZ_8$-graded group. Roughly speaking, $\wRBr_\ast(\cG)$ is the subgroup of elements of constant types in $\wRBr(\cG)$. In the present paper, we explore both geometric (in terms of groupoid extensions) and cohomological interpretations of $\wRBr(\cG)$. We show that when the Real structure of $\cG$ is trivial, $\wRBr(\cG)$ is a generalisation of Donovan-Karoubi's graded orthogonal Brauer group. For fixed point free involutions, we show that $\wRBr$ is a direct summand of $\wBr$. Our interest in the cohomological classification of Real graded Dixmier-Douady bundles is motivated by the study of twisted $KR$-theory we present in~\cite{Moutuou:Thesis}. \\

{\bf General plan.} In Appendix A, we classify all Real structures on graded elementary complex $\cstar$-algebras. In Section 1, we give general notions of Real graded Banach bundles on a Real space. In Section 2, we define Real graded Dixmier-Douady bundles over a locally compact second-countable Hausdorff Real groupoid. Then we define the group $\wRBr(\cG)$ for locally compact second-countable Hausdorff Real groupoid $\cG$ and present few properties. In Section 3, we investigate connections between $\wRBr$ and the already known Brauer groups of topological groupoids and spaces, mainly with complex and real Brauer groups. In Section 4, we introduce a group $\Inv\wfK$ that is crucial in the cohomological picture of $\wRBr(\cG)$. In Section 5, we define the notion of generalised classifying morphisms for Real graded Dixmier-Douady bundles, and then exhibit their construction in Section 6. In Section 7, we prove the first intermediate isomorphism theorem establishing an isomorphism between the subgroup $\wRBr_0(\cG)$ generated by elements of "type 0" and the group of isomorphism classes of "stable" generalised classifying morphisms. Section 8 is devoted to the case of a locally compact Real group; {\it i.e.}, when the unit space of the groupoid is the one point set. In Section 9, we prove the main isomorphism theorems of the paper. Finally, in Section 10, we discuss the case of "ungraded" Real Dixmier-Douady bundles, from which we obtain the Real analogue of~\cite{Kumjian-Muhly-Renault-Williams:Brauer}.\\

{\bf Preliminaries and conventions}. Recall~\cite{Atiyah:K_Reality} that a Real space is a pair $(X,\tau)$ consisting of a (topological) space $X$ and a homeomorphism $\tau:X\To X$, called {\it (Real) involution} such that $\tau^2=\id$. A {\it Real map} $f:(X,\tau)\To (Y,\sigma)$ between Real spaces consists of a map $f:X\To Y$ such that $\sigma(f(x))=f(\tau(x))$ for all $x\in X$. More generally, a {\it Real groupoid} is a groupoid $\grpd$ endowed with a groupoid isomorphism $\tau:\cG\To \cG$ such that $\tau^2=\id$. {\it Real morphisms} between Real groupoids are defined in the obvious way. The category of all Real groupoids and Real morphisms is denoted $\RG_s$. There is also the notion of Real generalised morphisms between Real groupoids defined in~\cite{Moutuou:Real.Cohomology}, which is an equivariant version of the ordinary notion of generalised morphism (~\cite{Hilsum-Skandalis:Morphismes,Tu-Xu-Laurent-Gengoux:Twisted_K}. The category whose objects are Real groupoids and whose morphisms are Real generalised morphisms is denoted by $\RG$. We refer to~\cite{Moutuou:Real.Cohomology} for details about all the materials used in this paper: Real groupoids and their cohomology theory, Real generalised morphisms and Real graded central extensions. 

Given a Real space $(X,\tau)$ (resp. a Real groupoid $(\cG,\tau)$), we will often omit the involution and simply write $X$ (resp. $\cG$). The image of an element by the involution will often be represented by a "bar" over it; {\it i.e.}, for $g\in \cG$, we write $\bar{g}:=\tau(g)$. By a Real point of a Real groupoid we will mean a one that is invariant under the involution. The {\it Real part} of $\cG$ is the subset of all its Real points. 

Finally, throughout the paper, by a Real groupoid we will mean a locally compact second-countable Hausdorff Real groupoid with Real Haar system (see {\it loc.cit.}), unless otherwise stated.


\section{Real fields of graded Banach bundles}

In this section we give the basics of Real fields of graded Banach spaces, Banach algebras, Hilbert spaces, and of $\cstar$--algebras over Real spaces. We shall first recall from Appendix A that a {\it Real graded $\cstar$-algebra} (abbreviated as {\it Rg $\cstar$--algebra}) is a $\ZZ_2$-graded complex $\cstar$-algebra $A=A^0\oplus A^1$ endowed with a conjugate-linear ${}^\ast$-automorphism $\sigma:A\To A$ such that $\sigma^2=\id$ and $\sigma(A^i)=A^i,i=0,1$. Rg Banach spaces, Banach algebras or Hilbert spaces are defined similarly. The involution $\sigma$ is called a {\it Real structure}, and will often be omitted and represented, as in the case of groupoids, by a "bar"; that is, for $a\in A$ we will write $\bar{a}:=\sigma(a)$.

\begin{df}[Compare~\cite{Doran-Fell:Representations}]\label{R-gdd-bdl}
Let $(X,\tau)$ be a locally compact Hausdorff Real space. A  \emph{continuous (resp. u.s.c., for upper semi-continuous) Real field of graded Banach spaces} $\cA$ over $(X,\tau)$ consists of a family $(\cA_x)_{x\in X}$ of graded Banach spaces together with a topology on $\tilde{\cA}=\coprod_{x\in X}\cA_x$ and an involution $\sigma:\tilde{\cA}\To \tilde{\cA}$ such that :
\begin{itemize}
  \item[(i)] the topology on $\cA_x$ induced from that on $\tilde{\cA}$ is the norm-topology;
  \item[(ii)] the projection $p:\tilde{\cA}\To X$ is Real, continuous, and open;
 \item[(iii)] the map  $a\mto \|a\|$ is continuous (resp. u.s.c) from $\tilde{\cA}$ to $\RR^+$, and $\|\sigma(a)\|=\|a\|, \ \forall a\in A$. 
 \item[(iv)] the map $(a,b)\mto a+b$ is continuous from $\tilde{\cA}\times_{X} \tilde{\cA}$ to $\tilde{\cA}$;
 \item[(v)] the scalar multiplication $(\lambda, a) \mto \lambda a$ is continuous from $\CC \times \tilde{\cA}$ to $\tilde{\cA}$;
  \item[(vi)] the induced bijection  $\sigma_x:\cA_x \To \cA_{\tau(x)}$ is an anti-linear isomorphism of graded Banach spaces for every $x\in X$, {\it i.e.} the diagram 
\begin{equation}~\label{diag-usc}
  \xymatrix{\CC \times \cA_x \ar[r] \ar[d]  & \cA_x \ar[d] \\ \CC \times \cA_{\tau(x)} \ar[r] & \cA_{\tau(x)}
  }
\end{equation}
   commutes, where the horizontal arrows are the action of $\CC$ on the fibres and the vertical ones are the Real involutions ($\CC$ being endowed with the complex conjugation), and $\sigma \circ \epsilon_x= \epsilon_{\tau(x)}\circ \sigma_x$.
 \item[(vii)] if $\lbrace a_i\rbrace$ is a net in $\tilde{\cA}$ such that $\|a_i\| \sto 0$ and $p(a_i)\sto x \in X$, then $a_i\sto 0_x$, where $0_x$ is the zero element in $\cA_x$. 
  \end{itemize}
We also say that $(\cA,\sigma)$ is a \emph{Real graded Banach bundle (resp. u.s.c. bundle) over $(X,\tau)$}.   
\end{df}

\begin{df}
A \emph{Rg Hilbert bundle (resp. u.s.c. bundle)} over $(X,\tau)$ is a Real graded Banach bundle (resp. u.s.c. bundle) $(\cA,\sigma)$ over $(X,\tau)$ each fibre $\cA_x$ o which is a graded Hilbert space such that the fibre-wise scalar product satisfies $$ \< \sigma_x(\xi),\sigma_x(\eta)\> = \overline{\< \xi,\eta\>}$$ for all $\xi, \eta \in \cA_x$.
\end{df}

\begin{df}
A \emph{Rg $\cstar$-bundle} (resp. u.s.c. $\cstar$-bundle) over $(X,\tau)$ is a Rg Banach bundle (resp. u.s.c. bundle) $(\cA,\sigma)$ such that each fibre is a graded $\cstar$-algebra and the following properties hold
\begin{itemize}
 \item[(a)] the map $(a,b)\mto ab$ is continuous from $\tilde{\cA} \times_X \tilde{\cA}$ to $\tilde{\cA}$;
 \item[(b)] $\sigma(ab)=\sigma(a)\sigma(b)$ for all $(a,b)\in \tilde{\cA}\times_X \tilde{\cA}$;
 \item[(c)] for $x\in X$, $\sigma_x(a^\ast)=\sigma_x(a)^\ast$ for all $a\in \cA_x$.
 \end{itemize}		
\end{df}

Homomorphism of Rg u.s.c. Banach bundles and of u.s.c. $\cstar$-bundles are defined in an obvious way. 

\begin{ex}[Trivial bundles]
If $(A,{}^-)$ is any Rg Banach algebra (resp. $\cstar$-algebra), then the first projection $pr_1: (X\times A,\tau \times {}^-)\To (X,\tau)$ defines a Rg Banach bundle (resp. $\cstar$-bundle) with fibre $A$. A Real graded Banach bundle (resp. $\cstar$-bundle) of this form is called \emph{trivial}. 
\end{ex}

\begin{df}~\label{df:loc-triv_gBund}
A u.s.c. field of graded Banach spaces $\tilde{\cA}\To X$ (without Real structure) is said to be \emph{locally trivial} if for every $x\in X$, there exists a neighborhood $U\ni x$ such that $\tilde{\cA}_{|U}$ is isomorphic (under a graded isomorphism) to a trivial field $U\times A$, where $A$ is a graded Banach space. 

Similarly, we talk about locally trivial field of graded Hilbert spaces, graded Banach algebras, and graded $\cstar$-algebras.
\end{df}

Unless otherwise stated, all of the graded Banach bundles and $\cstar$-bundles we are dealing with are assumed to be locally trivial.

We shall however point out that the above notion of local triviality is not sufficient when Real structures are involved. Specifically, suppose $(X,\tau)$ is a Real space and $(\cA,\sigma)\To (X,\tau)$ is a u.s.c. Real field of graded Banach spaces which is locally trivial in the sense of Definition~\ref{df:loc-triv_gBund}. Then it is not true that there exists a Rg Banach space $A$ such that the Real space $\cA$ locally behaves like $A$ in the sense that there would exist for all $x\in X$ an open invariant neighborhood (also called {\it Real neighborhood}) $U$ of $x$ (\emph{i.e.} $\tau(U)=U$) and a Real homeomorphism $h:p^{-1}(U)\To U\times A$; or equivalently, that there would exist a Real open cover $\{U_i\}_{i\in I}$ of $X$ ({\it i.e.} $I$ has an involution $i\mto \bar{i}$ such that $\tau(U_i)=U_{\bar{i}}, \forall i\in I$; see~\cite{Moutuou:Real.Cohomology}) and a trivialisation $h_i:p^{-1}(U_i)\To U_i\times A$ such that the following diagram commutes 

\begin{eqnarray}~\label{eq:diagram-LTCR}
 \xymatrix{p^{-1}(U_i) \ar[r]^{h_i}\ar[d]^{\tau_{|U_i}} & U_i\times A \ar[d]^{\tau \times bar} \\ p^{-1}(U_{\bar{i}}) \ar[r]^{h_{\bar{i}}} & U_{\bar{i}}\times A}	
 \end{eqnarray}

We then give the following

\begin{df}
A Rg Banach bundle (resp. $\cstar$-bundle, Hilbert bundle, etc.) $\cA\To X$ is LTCR (\emph{locally trivial in the category of Real spaces}) if there exists a Rg Banach space (resp. $\cstar$-algebra, Hilbert space, etc.) $A$ and a Real local trivialisation $(U_i,h_i)_{i\in I}$ such that the diagram~\eqref{eq:diagram-LTCR} commutes.	
\end{df}

Here is an example of a not LTCR Real graded $\cstar$-bundle.

\begin{ex}
Let $A$ be a simple separable stably finite unital $\cstar$-algebra that is not the complexification of any real $\cstar$-algebra~\cite[Corollary 4.1]{Phillips:anti-isom}. Define a continuous Real field of (trivially) graded $\cstar$-algebras $\tilde{\cA}$ over the Real space $\bfS^{0,1}=\{+1,-1\}$ by setting $$\cA_{-1}:=A, \quad  {\rm and\ } \cA_{+1}:=\overline{A},$$ where $\overline{A}$ is the complex conjugate of $A$, together with the Real structure $\sigma:\tilde{\cA}\To \tilde{\cA}$ given by the conjugate linear ${}^\ast$-isomorphism $\flat:A\To \overline{A}$ (the identity map). Then $\cA$ is not LTCR since $A \ncong \overline{A}$.	
\end{ex}

\begin{df}[Pull-backs]
If $(\cA,\sigma)$ is a Rg $\cstar$-bundle over $(X,\tau)$ and $\vp:(Y,\vr) \To (X,\tau)$ is a continuous Real map, then the \emph{pull-back of $(\cA,\sigma)$ along $\vp$} is the Rg $\cstar$-bundle $(\vp^\ast \cA,\vp^\ast \sigma)\To (Y,\vr)$, where $\vp^\ast \cA:=Y \times_{\vp,Y,p}\cA$, and $\vp^\ast \sigma (y,a):=(\vr(x),\sigma(a)), \ \forall (y,a)\in \vp^\ast \cA$. Each fibre $(\vp^\ast \cA)_y$ can be identified with $\cA_{\vp(y)}$ and then inherits the grading of the latter.
\end{df}

It can easily be checked that if $(\cA,\sigma)\To (X,\tau)$ is LTCR, then so is the pull-back $(\vp^\ast\cA,\vp^\ast\sigma)$ over $(Y,\vr)$.

From now on, all our Rg bundles are supposed to be LTCR, unless otherwise stated.

\begin{rem}
For any Rg Banach (resp. $\cstar$-) bundle $(\cA,\sigma)\To (X,\tau)$, $\Co(X,\cA)$ is a Rg Banach (resp. $\cstar$-) algebra with respect to the obvious pointwise operations and norm $\|s\|:=\sup_{x\in X} \|s(x)\|$; the grading and the Real structure are given by $\epsilon(s)(x):=\epsilon_x(s(x))$ and $\sigma(s)(x):=\sigma_{\tau(x)}(s(\tau(x)))$.\
It is straightforward that $\sigma_x \circ \sigma_{\tau(x)} = \Id_{\cA_{\tau(x)}}, \  \sigma_{\tau(x)}\circ \sigma_x=\Id_{\cA_x}$. In particular, for a Real point $x\in X$, $\cA_x$ is a Rg Banach (resp. $\cstar$-) algebra. 
\end{rem}

Note that if $p:(\cA,\sigma)\To (X,\tau)$ is a Rg $\cstar$-bundle, then $\Co(X)$ acts by multiplication on $\Co(X,\cA)$. Moreover, this action is Real and graded in the sense that it is compatible with the Real structure and the grading. Indeed, for $f\in \Co(X)$ and $s\in \Co(X,\cA)$, we set $\sigma(f\cdot s)(x):=\overline{f(\tau(x))}\sigma(s(\tau(x)))=\tau(f)(x)\cdot \sigma(s)(x)$. Thus, $(A,\sigma)$, where $A=\Co(X,\cA)$, is a Rg $\Co(X)$-module~\cite{Kasparov:Operator_K}.

If $(\cA,\sigma)$ is a Rg Banach bundle over $(X,\tau)$, then a continuous function $s:X\To \cA$ such that $p\circ s=\Id_X$ is called a \emph{section} of $\cA$. Observe that if $s$ is a section of $\cA$, then for any $x\in X$, $s(\tau(x))$ and $\sigma(s(x))$ are in the same fibre $\cA_{\tau(x)}$. We say that $s$ is Real if $s(\tau(x))=\sigma(s(x))$. The set of all sections $s$ for which $x\mto \|s(x)\|$ is in $\Co(X)$  is denoted by $\Co(X,\cA)$.

\begin{df}
A Rg Banach bundle $p:\cA \To X$ has \emph{enough sections} if given any $x\in X$ and any $a\in \cA_x$, there is a continuous section $s\in \Co(X,\cA)$ such that $s(x)=a$. 
\end{df}

The following result ensures that all Rg Banach bundles considered in the paper have enough sections (see~\cite[Appendix C]{Doran-Fell:Representations} for a detailed proof).

\begin{thm}[Douady--dal Soglio-H\'erault]~\label{enough-sect}
Any Banach bundle over a paracompact or locally compact space has enough sections.
\end{thm}

\begin{cor}
Suppose $(X,\tau)$ is a locally compact Hausdorff Real space. Then, if $p: (\cA,\sigma)\To (X,\tau)$ is a Rg Banach bundle, Real sections always exist.
\end{cor}

\begin{proof}
Let $x\in X, a\in \cA_x$; then by Theorem~\ref{enough-sect} there exists $s\in \Co(X,\cA)$ such that $s(x)=a$. Since for every $x\in X$, $s(x)$ and $\sigma_{\tau(x)}(s(\tau(x)))$ belong to the Banach algebra $\cA_x$, the map $\tilde{s}:=\frac{1}{2}(s+\sigma(s))$ is a well-defined section in $\Co(X,\cA)$ which verifies $\sigma(\tilde{s})=\tilde{s}$.
\end{proof}

\begin{df}[Elementary Rg $\cstar$-bundle]
A Rg $\cstar$-bundle $(\cA,\sigma)\To (X,\tau)$ is called \emph{elementary} if each fibre $\cA_x$ is isomorphic to a graded elementary $\cstar$-algebra. 
\end{df}

\begin{df}
We say that a Rg elementary $\cstar$-bundle $p:(\cA,\sigma)\To(X,\tau)$ satisfies \emph{Fell's condition} if (and only if) $\Co(X,\cA)$ is continuous-trace.
\end{df}

If $(\cA,\sigma)\To (X,\tau)$ is a Rg elementary $\cstar$-bundle, the spectrum of $A$ is naturally identified with the Real space $(X,\tau)$ (see~\cite{Rosenberg:Continuous-trace, Raeburn-Williams:Morita_equivalence, Li:Real_algebras}). 

In the sequel, we will write $A$ for $\Co(X,\cA)$ and if $\vp: (Y,\vr)\To (X,\tau)$ is a continuous Real map, we will write $\vp^\ast A$ for $\Co(Y,\vp^\ast \cA)$.

\begin{df}~\label{Morita-DD-bdl}
Suppose that $p_\cA:(\cA,\sigma_\cA)\To (X,\tau)$ and $p_\cB:(\cA,\sigma_\cB)\To (X,\tau)$ are Rg $\cstar$-bundles. Then a Rg Banach bundle $q: (\cE,\sigma_\cA)\To (X,\tau)$ is called a \emph{Rg $\cA$--$\cB$-imprimitivity bimodule} if each fibre $\cE_x$ is a graded $\cA_x$--$\cB_x$--imprimitivity bimodule such that 
\begin{itemize}
\item[(a)] the natural maps $(\cA\times_X \cE,\sigma_\cA \times \sigma_\cE)\To (\cE,\sigma_\cE), \ (a,\xi)\To a\cdot\xi$ and $(\cE\times_X \cB,\sigma_\cE \times \sigma_\cA)\To (\cE,\sigma_\cE), \ (b,\xi)\mto b\cdot\xi$ are Real and continuous;
\item[(b)] $(\sigma_\cA)_x ({}_{\cA_x}\langle \xi,\eta\rangle)={}_{\cA_{\tau(x)}}\langle (\sigma_\cE)_x(\xi),(\sigma_\cE)_x(\eta)\rangle$  and  $(\sigma_\cB)_x(\langle \xi,\eta \rangle_{\cB_x})=\langle (\sigma_\cB)_x(\xi),(\sigma_\cE)_x(\eta)\rangle_{\cB_{\tau(x)}}$.
\end{itemize}

If such a Rg $\cA$--$\cB$-imprimitivity bimodule exists, we say that $(\cA,\sigma_\cA)$ and $(\cB,\sigma_\cB)$ are \emph{Morita equivalent}.
\end{df}

Let $(\cA,\sigma_\cA)$ and $(\cB,\sigma_\cB)$ be elementary Rg $\cstar$-bundles over $(X,\tau)$. Then, as in the ungraded complex case~\cite[p.18]{Kumjian-Muhly-Renault-Williams:Brauer}, there is a unique elementary Rg $\cstar$-bundle $\cA \hat{\otimes} \cB$ over $X \times X$ with fibre $\cA_x \hat{\otimes} \cB_y$ over $(x,y)$ and such that $(x,y)\mto f(x)\hat{\otimes} g(y)$ is a section for all $f\in A=\Co(X,\cA)$ and $g\in B=\Co(X,\cB)$. The Real structure is given by $(\sigma_\cA)_x \hat{\otimes} (\sigma_\cB)_y$ over $(x,y)$. By this construction, the elementary Rg $\cstar$-bundle $(\cA\hat{\otimes}\cB,\sigma_\cA\hat{\otimes}\sigma_\cB)$ satisfies Fell's condition if $(\cA,\sigma_\cA)$ and $(\cB,\sigma_\cB)$ do so, as does its restriction $(\cA\hat{\otimes}_X\cB,\sigma_\cA\hat{\otimes}_X\sigma_\cB)$ to the diagonal $\Delta=\lbrace (x,x)\in X\times X \rbrace$.

\begin{df}
 Let $(\cA,\sigma_\cA)$ and $(\cB,\sigma_\cB)$ be Rg elementary $\cstar$-bundles over $(X,\tau)$. Then, their \emph{tensor product} is defined to be the Rg elementary $\cstar$-bundle $(\cA \hat{\otimes}_X\cB,\sigma_\cA\hat{\otimes}_X\sigma_\cB)$ over the Real space $(X,\tau)$ which is identified with the diagonal $(\Delta,\tau)$ of $(X\times X,\tau \times \tau)$.
\end{df}

For the rest of the paper, all Real structures will be omitted, except where there likely to be confusion. We will then simply write $\cA$ for the Rg Banach bundle $(\cA,\sigma)$, and $\cG$ for the Real groupoid $(\cG,\tau)$.


\section{The graded Brauer group of a Real groupoid}

The graded Brauer group of a locally compact groupoid $\cG$ was defined in~\cite{Tu:Twisted_Poincare} as the set of Morita equivalence classes of graded complex Dixmier-Douady bundles over $\cG$, generalising~\cite{Kumjian-Muhly-Renault-Williams:Brauer}. Our purpose in this section is to define a new Brauer group which is relevant to the category of (locally compact second-countable Hausdorff) Real groupoids.

\begin{df}\label{df-DD-bdle}
Let $\grpd$ be a Real groupoid with involution $\tau$. Let $p:\cA\To X$ be a (LTCR) Rg u.s.c. Banach bundle. A Real $\cG$-action by isomorphisms $\alpha$ on $\cA$ is a collection $(\al_g)_{g\in \cG}$ of graded isomorphisms (resp. $\ast$-isomorphisms) $\alpha_g : \cA_{s(g)}\To \cA_{r(g)}$ such that
\begin{itemize}
\item[(a)] $g\cdot a:=\alpha_g(a)$ makes $(\cA,\sigma)$ into a (left) Real $\cG$-space with respect to $p$;
\item[(b)] the induced anti-linear graded isomorphisms $\tau_x:\cA_x\To \cA_{\bar{x}}$ verify $\tau_{r(g)}\circ \alpha_g=\alpha_{\bar{g}}\circ \tau_{s(g)}: \cA_{s(g)} \overset{\cong}{\To} \cA_{\overline{r(g)}}$, for every $g\in \cG$;
\item[(c)] $\alpha_{gg'}=\alpha_g \circ \alpha_{g'}$ for all $(g,g')\in \cG^{(2)}$. 
\end{itemize}

We say that $(\cA,\alpha)$ is a Rg u.s.c. Banach $\cG$--bundle. If the field $\tilde{\cA}=\coprod_X\cA_x\To X$ is continuous, then $(\cA,\al)$ is called a {\it Rg Banach $\cG$-bundle}.  
\end{df}

One also defines Rg u.s.c. $\cstar$--$\cG$-bundles, and Rg u.s.c. Hilbert bundles. In the case of $\cstar$--$\cG$-bundles, the isomorphisms $\al_g$ are required to be ${}^\ast$-isomorphisms, while in the case of Hilbert $\cG$-bundles, they are required to be isometries.  

\begin{df}
A morphism of Rg Banach $\cG$-bundles (resp. $\cstar$--$\cG$-bundles) $(\cA,\alpha)\To(\cB,\beta)$ is a morphism $\Phi: \cA\To \cB$ of Rg Banach bundles (resp. $\cstar$-bundles) which is $\cG$-equivariant; \emph{i.e.}, $\Phi_{r(g)}\circ \alpha_g=\beta_g\circ \Phi_{s(g)}$ for all $g\in \cG$.
\end{df}

Notice that if $(\cA,\alpha)$ is Rg Banach $\cG$-bundle, then $\alpha_x=\Id_{\cA_x}$ for all $x\in X$. Indeed, for every $x\in X$, $\alpha_x:\cA_x \To \cA_x$ is a graded automorphism, and $\alpha_x=\alpha_{x\cdot x}=\alpha_x \circ \alpha_x$. In particular, if we put $x=gg^{-1}\in X$ for $g\in \cG$, we obtain $\alpha_{gg^{-1}}=\alpha_g \circ \alpha_{g^{-1}}  =\Id$ and then $\alpha_{g^{-1}}=\alpha_g^{-1}$ for every $g\in \cG$.

\begin{df}
A \emph{Rg D-D (Dixmier-Douady) bundle} over $(\cG,\tau)$ is a Rg $\cstar$--$\cG$-bundle $(\cA,\alpha)$ such that  $\cA\To X$ is a Rg elementary $\cstar$-bundle that satisfies Fell's condition. We denote by $\fwRBr(\cG)$ the collection all Rg D-D bundles over $\cG$.
\end{df}

Suppose $\alpha$ is a Rg $\cG$-action by isomorphisms on the Rg $\cstar$--$\cG$-bundle $\cA$. Consider the Rg $\cstar$-algebra $A=\Co(X;\cA)$. Then $\al$ induces a Rg $\Co(\cG)$--linear isomorphism $\alpha: s^\ast A \To r^\ast A$ defined by $\alpha(f)(g):=\alpha_g(f(g))$ for $f\in s^\ast A$ and $g\in \cG$. 

\begin{ex}~\label{ex:triv_DD-bdle}
Let $X\times \CC$ be endowed with the Real structure $\overline{(x,t)}:=(\tau(x),\bar{t})$, and the $\cG$-action by automorphisms $g\cdot (s(g),t):=(r(g),t)$. Then with respect to the projection $X\times\CC\To X$, $X\times \CC$ is a Rg D-D bundle over $\grpd$.	
\end{ex}

\begin{ex}~\label{ex:construction-K^G}
Let $\mu=\{\mu^x\}_{x\in X}$ be a Real Haar system on $\grpd$. Let the graded Hilbert space $\hat{\cH}=l^2(\NN)\oplus l^2(\NN)$ be equipped with a fixed Real structure of type $J_{\RR,0}$ (see Appendix A). For $x\in X$, we put $\what{\cH}_{p,x}:=L^2(\cG^x)\hat{\otimes}\hat{\cH}$, together with the scalar product $\tlip<\cdot,\cdot>(x)$ given by
\begin{equation}~\label{eq:scalar-prod_on_H_G}
	\tlip<\xi,\eta>(x):= \int_{\cG^x}\<\xi(g),\eta(g)\>_\CC d\mu_\cG^x(g), \ {\rm for\ } \xi,\eta\in L^2(\cG^x;\hat{\cH})\cong L^2(\cG^x)\hat{\otimes}\hat{\cH}.
\end{equation}

 Let $\hat{\cH}_{\cG}:=\coprod_{x\in X}\hat{\cH}_x$ be equipped with the action $g\cdot(s(g),\vp\hat{\otimes}\xi):=(r(g),(\vp\circ g^{-1})\hat{\otimes}\xi)\in \what{\cH}_{r(g)}$. Define the Real structure on $\what{\cH}_{\cG}$ by $(x,\vp \hat{\otimes}\xi)\mto (\tau(x),\tau(\vp)\hat{\otimes}J_{\RR,0}(\xi))$. Then one shows that there exists a unique topology on $\hat{\cH}_\cG$ such that the canonical projection $\what{\cH}_{\cG}\To \Gpdo$ defines a locally trivial Rg Hilbert $\cG$-bundle.

Now, let $\what{\cK}_{x}:=\cK(\what{\cH}_{x})$ be equipped with the operator norm topology, and put $$\what{\cK}_{\cG}:=\coprod_{x\in X}\what{\cK}_{x}$$ together with the Real structure given by $\overline{(x,T)}:=(\bar{x},\bar{T})$, where $\bar{T}\in \what{\cK}_{\bar{x}}$ is defined by $\bar{T}(\vp\hat{\otimes}\xi):=T(\tau(\vp)\hat{\otimes}Ad_{J_\RR,0}(\xi))$ for all $\vp\hat{\otimes}\xi \in \what{\cH}_{\bar{x}}$. Next, define the Real $\cG$-action $\theta$ on $\what{\cK}_{\cG}$ by 
\[ \theta_g(s(g),T):=(r(g),gTg^{-1}). 
\]
We then have a Rg D-D bundle $(\what{\cK}_{\cG},\theta)$ over $\cG$ given by the canonical projection 
\[\what{\cK}_{\cG}  \To X, \ (x,T)\mto x .\]
\end{ex}

\begin{rem}
The construction of the Rg D-D bundle $\what{\cK}_{\cG}$ in Example~\ref{ex:construction-K^G} may be used in a straightforward way to obtain a Rg D-D bundle $\cK(\hat{\sH})$ out of any Rg Hilbert bundle $\cG$-bundle $\hat{\sH}$. 
\end{rem}

Let $\grpd$ and $\gamgpd$ be Real groupoids and let 
 $$\xymatrix{Z \ar[d]^\fr \ar[r]^\fs & X \\ Y}$$ be a generalised Real homomorphism~\cite{Moutuou:Real.Cohomology}. Let $p:\cA\To X$ is a Rg (u.s.c.) Banach bundle. Then $\fs^\ast \cA=Z \times_{\fs,X,p}\cA$ is a principal (right) $\cG$-space with respect to the operation:
\[ (z,a)\cdot g:=(zg,\alpha_g^{-1}(a))
\]
for $r(g)=z$. It is obviously a Real space with respect to the involution $\overline{(z,a)}:=(\bar{z},\bar{a})$. Next, define $\cA^Z$ to be the quotient space $\fs^\ast \cA / \cG$ together with the induced Real structure 
$$\overline{[z,a]}:=[\overline{(z,a)}],$$
 where we use the notation $[z,a]$ for the orbit of $(z,a)$ in $\cA^Z$. Consider the continuous surjective map $\fr \circ pr_1: \fs^\ast \cA \To Y, \ (z,a)\mto \fr(z)$, where, as usual, $pr_1$ denotes the first projection. Since $\fr(zg)=\fr(z)$ for $(z,g)\in Z\times_{\fs,X,r}\cG$ (see for instance~\cite[Definition 1.17]{Moutuou:Real.Cohomology}), we get a well defined continuous surjection $p^Z: \cA^Z \To Y$ given by $p^Z([z,a]):=\fr(z)$. Furthermore, since $\fr$ is Real, one has $$p^Z(\overline{[z,a]})=p^Z([\bar{z},\bar{a}])=\fr(\bar{z})=\overline{\fr(z)}=\overline{p^Z([z,a])}.$$
Thus, $p^Z: \cA^Z\To Y$ is a continuous Real surjection. Moreover, it is not hard to check that $p^Z: \cA^Z\To \Gamo$ is open and the map $a\mto [z,a]$ defines a graded isomorphism from $\cA_{\fs(z)}$ onto $\cA^Z_{\fr(z)}$ (see~\cite[p.14]{Kumjian-Muhly-Renault-Williams:Brauer}).

\begin{pro}.
Let $\grpd$ and $\gamgpd$ be Real groupoids. Suppose that $Z:\Ga\To\cG$ is a generalised Real homomorphism and that $(\cA,\alpha)$ is a Rg Banach $\cG$-bundle. Then, with the constructions above, $p^Z: \cA^Z\To Y$ is a Rg Banach bundle. Furthermore, \begin{eqnarray}
\alpha^Z_\g [z,a]:=[\g\cdot z,a], \ \text{for} \ \fr(z)=s(\g),
\end{eqnarray}
defines a Real left $\Ga$-action on $\cA^Z$ making $(\cA^Z,\alpha^Z)$ into a Rg Banach $\Ga$-bundle called the \emph{pull-back} of $(\cA,\alpha)$ along $Z$.

In particular, if $(\cA,\alpha)\in \fwRBr(\cG)$, then $(\cA^Z,\alpha^Z)\in \fwRBr(\Ga)$.
\end{pro}

The proof of this proposition is almost the same as that of~\cite[Proposition 2.15]{Kumjian-Muhly-Renault-Williams:Brauer}, so we omit it.

\begin{cor}
Let $\grpd$ and $\gamgpd$ be Real groupoids. Suppose that $Z: \cG\To \Ga$ is a Real Morita equivalence. Then the map $\Phi^Z: \fwRBr(\cG) \To \fwRBr(\Ga)$ given by 
\begin{eqnarray}
\Phi^Z(\cA,\alpha):= (\cA^Z,\alpha^Z)
\end{eqnarray}
is a bijection.
\end{cor}

The following construction is slightly modified from~\cite[p.20]{Kumjian-Muhly-Renault-Williams:Brauer}.

\begin{df}~\label{df:conjugate-Ban-G-bdl}
Let $(\cA,\al)$ be a Rg (u.s.c.) Banach $\cG$-bundle. Define the \emph{conjugate bundle} $(\overline{\cA},\bar{\al})$ of $(\cA,\al)$ as follows. Let $\overline{\cA}$ be the topological Real space $\cA$  and let $\flat:\cA\To \overline{\cA}$ be the identity map. Then $\overline{p}:\overline{\cA}\To X$ defined by $\bar{p}(\flat(a))=\flat(p(a))$ is a Rg (u.s.c.) Banach bundle with fibre $\overline{\cA}_x$ identified with the conjugate graded Banach algebra of $\cA_x$ (the grading is $\overline{\cA}_x^i=\overline{\cA^i_x}, i=0,1$). Furthermore, endowed with the Real $\cG$-action by automorphisms $\bar{\al}_g(\flat(a)):=\flat(\al_g(a))$ for $g\in \cG, a\in \cA_{s(g)}$, $\overline{\cA}\To X$ becomes a Rg (u.s.c.) Banach $\cG$-bundle. If $(\cA,\al)\in \fwRBr(\cG)$, then $(\overline{\cA},\bar{\al})\in \fwRBr(\cG)$. 	
\end{df}

\begin{df}
Let $\grpd$ be a Real groupoid. Two elements $(\cA,\alpha)$ and $(\cB,\beta)$ in $\fwRBr(\cG)$ are \emph{Morita equivalent} if there is a Rg $\cA$--$\cB$-imprimitivity bimodule $\cE\To X$ which admits a Real $\cG$-action $V$ by isomorphisms such that
\begin{equation}~\label{eq_df:morita-BrR}
\begin{array}{ll}
{}_{\cA_{s(g)}} \langle V_g(\xi),V_g(\eta)\rangle = & \alpha_g({}_{\cA_{s(g)}} \langle \xi,\eta\rangle); \text{and}\\
\langle V_g(\xi),V_g(\eta)\rangle_{\cB_{s(g)}} = & \beta_g(\langle \xi,\eta\rangle_{\cB_{s(g)}})
\end{array}
\end{equation}
for all $g\in \cG$, and $\xi, \eta \in \cE_{s(g)}$. In this case we write $(\cA,\alpha)\sim_{(\cE,V)}(\cB,\beta)$.
\end{df}

\begin{ex}
Suppose that $\Phi:(\cA,\alpha)\To (\cB,\beta)$ is an isomorphism of Rg D-D bundles over $\grpd$. Then $(\cA,\alpha)\sim_{(\cB,\beta)}(\cB,\beta)$.
\end{ex}

\begin{lem}
Morita equivalence of Rg DD-bundles over $\grpd$ is an equivalence relation in $\fwRBr(\cG)$.
\end{lem}

\begin{proof}
The proof can almost be copied from that of Lemma 3.2 in~\cite{Kumjian-Muhly-Renault-Williams:Brauer}.	
\end{proof}

We are now ready to define the Real graded Brauer group.

\begin{df}
Let $\grpd$ be a Real groupoid. \emph{The Real graded Brauer group of $\cG$} $\wRBr(\cG)$ is defined as the set of Morita equivalence classes of Rg D-D bundles over $\cG$. The class of $(\cA,\al)$ in $\wRBr(\cG)$ is denoted by $[\cA,\al]$.
\end{df}

\begin{ex}
Let $\cG$ be the one point set $\lbrace \ast \rbrace$ together with the trivial involution. Then, every Rg DD-bundle over $\lbrace \ast \rbrace$ is trivial; \emph{i.e}. it is given by a Rg elementary $\cstar$-algebra $\wK_p$. We thus recover the Real graded Brauer group of the point $\wRBr(\ast)\cong \ZZ_8$ described in Appendix~\ref{subsect:BrR(pt)}.
\end{ex}

Let $(\cA,\alpha)$ and $(\cB,\beta)$ be Rg DD-bundles. We have defined, in the previous section, the tensor product $\cA\hat{\otimes}_X \cB$ which is a Rg $\cstar$-bundle over $X$. We  want to equip this bundle with a Real $\cG$-action $\alpha \hat{\otimes}\beta$ such that $(\cA\hat{\otimes}_X \cB,\alpha \hat{\otimes}\beta) \in \fwRBr(\cG)$. We define $\alpha \hat{\otimes}\beta$ as follows. For all $g\in \cG$, we put $\alpha_g \hat{\otimes}\beta_g: \cA_{s(g)}\hat{\otimes}\cB_{s(g)}\To \cA_{r(g)}\hat{\otimes}\cB_{r(g)}, \ a\hat{\otimes}b \mto \alpha_g(a)\hat{\otimes}\beta_g(b)$. Note that from the definition of a Real $\cG$-action on a Rg $\cstar$-bundle, $\alpha_g\hat{\otimes}\beta_g$ is a graded $\ast$-isomorphism that clearly verifies conditions (b) and (c) of Definition~\ref{df-DD-bdle}. Therefore, the same arguments used in~\cite[pp.18-19]{Kumjian-Muhly-Renault-Williams:Brauer} can be used here to show that $\alpha\hat{\otimes}\beta$ is continuous; thus, its restriction $\alpha \hat{\otimes} \beta$ on the closed subset $\cA\hat{\otimes}_X \cB$ of $\cA\hat{\otimes}\cB$ defines a Real $\cG$-action. Furthermore, this operation is easily seen to be Morita equivalence preserving.

\begin{pro}
Let $\grpd$ be a Real groupoid. Then $\wRBr(\cG)$ is an abelian group with respect to the operations
 \begin{equation}
 	[\cA,\al]+[\cB,\beta]:=[\cA\hat{\otimes}_X\cB,\al\hat{\otimes}\beta].	
 	\end{equation}	
The identity of $\wRBr(\cG)$ is given by the class $0:=[X\times \CC,\tau\times bar]$ of the Rg D-D bundle defined in Example~\ref{ex:triv_DD-bdle}. The inverse of $[\cA,\al]$ is $[\overline{\cA},\bar{\al}]$.
\end{pro}

\begin{proof}
See Proposition 3.6 and Theorem 3.7 in~\cite{Kumjian-Muhly-Renault-Williams:Brauer}.	
\end{proof}

For the sake of simplicity we will often use the following notations.

\begin{nota}~\label{nota:notations_in_BrR}
We will abusively write $\cA$ for the class $[\cA,\al]$ in $\wRBr(\cG)$; we will also leave out the actions when we are working in the group $\wRBr(\cG)$: for instance we will write $\cA+\cB$ instead of $[\cA,\al]+[\cB,\beta]$. 	
\end{nota}

\begin{lem}~\label{lem:A+K_G=A}
Let $(\cA,\al)\in \fwRBr(\cG)$ and let $(\wK_\cG,\theta)$ be the Rg D-D bundle defined in Example~\ref{ex:construction-K^G}. Then $\cA+\wK_\cG=\cA$ in $\wRBr(\cG)$.
\end{lem}

\begin{proof}
Recall that the Real $\cG$-action $\theta$ is given by $Ad_{\Theta}$, where $\Theta$ is the Real $\cG$-action on the Rg Hilbert $\cG$-bundle $\hat{\cH}_\cG\To X$ (see Example~\ref{ex:construction-K^G}); \emph{i.e}. $\theta_g(T)=\Theta_g T\Theta_g^{-1}$. 
The Rg Banach $\cG$-bundle $(\cA\hat{\otimes}_X\hat{\cA},\al\hat{\otimes}\Theta)$ provides a Morita equivalence $$(\cA\hat{\otimes}_X\wK_\cG,\al\hat{\otimes}\theta)\sim (\cA,\al)$$ in $\fwRBr(\cG)$ with respect to the point-wise actions and inner-product given by:
\begin{eqnarray}
(a\hat{\odot}T)\cdot (b\hat{\otimes}\xi) &:=& (-1)^{\partial T\cdot \partial b}ab\hat{\odot}T\xi, \ {\rm and\ } \nonumber \\
{}_{\cA_x\hat{\otimes}\wK_x}\<b\hat{\odot}\xi,d\hat{\odot}\eta\> &:=&(-1)^{\partial\xi\partial d}bd^\ast\hat{\odot}T_{\xi,\eta};\nonumber \\
(b\hat{\odot}\xi)\cdot c &:=& bc\hat{\odot}\xi,  \ {\rm and\ } \nonumber \\
 \<b\hat{\odot}\xi,d\hat{\odot}\eta\>_{\cA_x}&:= &\tlip<\xi,\bar{\eta}>(x)\cdot b^\ast d,\nonumber	
\end{eqnarray}
for $x\in X$, $a\hat{\odot}T\in \cA_x\hat{\otimes}\wK_x, b\hat{\odot}\xi,d\hat{\odot}\eta\in \cA_x\hat{\otimes}\hat{\cH}_x$, and $c\in \cA_x$.
\end{proof}

\begin{lem}
Let $\grpd$ be a Real groupoid and $(\cA,\al)\in \fwRBr(\cG)$. Then $\cA=0$ in $\wRBr(\cG)$ if and only if there exists a Rg Hilbert $\cG$-bundle $(\hat{\mathscr{H}},U)$ such that $(\cA,\al)\cong (\cK(\hat{\mathscr{H}}), Ad_U)$ in $\fwRBr(\cG)$. 	
\end{lem}

\begin{proof}
If $(\cE,V)$ is a Morita equivalence between $(\cA,\al)$ and the trivial bundle $X\times\CC$, then each fibre $\cE_x$ is a graded Hilbert space; and since $\cE_x$ is a full graded Hilbert $\cA_x$-module and since $\cE$ is a Rg Morita equivalence, there is an isomorphism of graded $\cstar$-algebras $\vp_x:\cA_x\To \cK(\cE_x)$ such that $\vp_x({}_{\cA_x}\<\xi,\eta\>)=T_{\xi,\eta}$, for all $\xi,\eta\in \cE_x$, and $\vp_{\bar{x}}(a)=\overline{\vp_x(\bar{a})}$ for all $a\in \cA_x$. Moreover, in view of relations~\eqref{eq_df:morita-BrR}, we have \[\vp_{r(g)}(\al_g({}_{\cA_{s(g)}}\<\xi,\eta\>))=\vp_{r(g)}({}_{\cA_{r(g)}}\<V_g(\xi),V_g(\eta)\>)=T_{V_g\xi,V_g\eta}=Ad_{V_g}(T_{\xi,\eta}),\]
for every $\g\in \cG$ and $\xi,\eta\in \cA_{s(g)}$. It follows that the family $(\vp_x)_{x\in X}$ is an isomorphism of Real graded D-D bundles $\vp: (\cA,\al)\To (\cK(\cE),Ad_V)$.

Conversely, using the same operations as in the proof of Lemma~\ref{lem:A+K_G=A}, the Rg Hilbert $\cG$-bundle $(\hat{\mathscr{H}},U)$ defines a Morita equivalence of Rg D-D bundles between $(\cK(\hat{\mathscr{H}}),Ad_U)$ and the trivial one $X\times \CC\To X$.
\end{proof}

From this lemma we deduce the following characterisation of Morita equivalent Rg D-D bundles.

\begin{cor}~\label{cor:characterization_Morita_BrR}
Let $(\cA,\al)$ and $(\cB,\beta)\in \fwRBr(\cG)$. Then $\cA=\cB$ in $\wRBr(\cG)$ if and only if there exists a Rg Hilbert $\cG$-bundle $(\hat{\mathscr{H}},U)$ such that $(\cA\hat{\otimes}_X\overline{\cB},\al\hat{\otimes}\bar{\beta})\cong (\cK(\hat{\mathscr{H}}),Ad_U)$ in $\fwRBr(\cG)$.	
\end{cor}


\section{Complex and orthogonal Brauer groups}

The aim of this section is to compare the group $\wRBr(\cG)$ of a Real groupoid $\grpd$ we defined in the previous section with the well-known graded complex and Brauer group $\wBr(\cG)$ of the groupoid $\cG$ (see~\cite{Parker:Brauer, Donovan-Karoubi,Freed-Hopkins:Twisted_K2, Tu:Twisted_Poincare}), as well as with a generalisation of Donovan-Karoubi's \emph{graded orthogonal Brauer group} $\wBRO(X)$~\cite{Donovan-Karoubi} mentioned in the introduction.

Recall again that the \emph{graded complex Brauer group} $\wBr(\cG)$ is defined as the set of Morita equivalence classes of \emph{graded complex D-D bundles}~\footnote{Elements of $\wBr(\cG)$ are defined in the same way as that of $\wRBr(\cG)$ except that no Real structures are involved} over the groupoid $\cG$. Moreover, there is an interpretation of $\wBr(\cG)$ in terms of \v{C}ech cohomology classes. More precisely, there is an isomorphism 
\begin{equation}
\wBr(\cG)\cong \check{H}^0(\cG_\bullet,\ZZ_2)\oplus\left( \check{H}^1(\cG_\bullet,\ZZ_2)\ltimes \check{H}^2(\cG_\bullet,\cS^1)\right).
\end{equation} 
For topological spaces, this group was denoted by $\textsf{GBr}^\infty(X)$ in Parker's paper~\cite{Parker:Brauer} and classifies countinuous-trace separable $\cstar$-algebras with spectrum $X$.  

In order to define \emph{twisted $K$-theory}, Donovan and Karoubi defined in their founding paper~\cite{Donovan-Karoubi} two groups $\textsf{GBrU}(X)$ and $\textsf{GBrO}(X)$ respectively called \emph{the graded unitary Brauer group} and \emph{the graded orthogonal Brauer group} of the space $X$. The former is just the finite-dimensional version of $\textsf{GBr}^\infty(X)$, while the latter is the set of equivalence classes of \emph{graded real simple algebra bundles}. They proved that 
\begin{equation}~\label{eq:isom_GBrO-cohom}
\textsf{GBrO}{X}\cong \check{H}^0(X,\ZZ_8)\oplus \left(\check{H}^1(X,\ZZ_2)\ltimes \check{H}^2(X,\ZZ_2)\right).
\end{equation}
We will define an infinite-dimensional analogue of $\textsf{GBrO}(X)$ for groupoids, and show later an isomorphism analogous to~\eqref{eq:isom_GBrO-cohom}.

We first give the following few results.

\begin{pro}~\label{pro:RBr_vs_Br-special-case}
Suppose that $\grpd$ is a Real groupoid which can be written as the disjoint union of two groupoids $\xymatrix{\cG_1\dar[r]& X_1}$ and $\xymatrix{\cG_2\dar[r]&X_2}$ such that $\tau(g_1)\in \cG_2, \tau(g_2)\in \cG_1, \forall g_1\in \cG_1,g_2\in \cG_2$. Then 
\[\wRBr(\cG)\cong \wBr(\cG_1)\cong \wBr(\cG_2).\]	
\end{pro}

\begin{proof}
Observe first that $\tau$ induces an isomorphism $\cG_1\cong \cG_2$, so that $\wBr(\cG_1)\cong \wBr(\cG_2)$.

 Let $(\cA,\al)\in \fwRBr(\cG)$. Then $\cA=\cA_1\oplus \cA_2$, where $\cA_1\To X_1$ and $\cA_2\To X_2$ are graded complex elementary $\cstar$-bundles. It is clear that the graded $\cG$-action $\al$ on $\cA$ induces a $\cG_i$-action $\al_i$ on $\cA_i, i=1,2$, making $(\cA_i,\al_i)$ into a graded complex D-D bundle over $\cG_i$. However, since the projection $p:\cA\to X_1\sqcup X_2$ intertwines the Real structure of $\cA$ and that of $X$, we have $\bar{a}_1\in \cA_2$ and $\bar{a}_2\in \cA_1$ for all $a_1\in \cA_1, a_2\in \cA_2$. Indeed, over all $x\in X_1$, the involution induces the conjugate linear isomorphism $$\tau_{x}:\cA_x=(\cA_1)_x\To \cA_{\bar{x}}=(\cA_2)_{\bar{x}}.$$ It turns out that the Real structure of $\cA$ induces an isomorphism of graded complex D-D bundles $$\tau:(\cA_2,\al_2)\stackrel{\cong}{\To}(\overline{\cA_1},\overline{\al_1})$$ over the groupoid $\cG_2$. In fact $(\overline{\cA},\bar{\al})$ is isomorphic to the Rg D-D bundle $(\cA',\al')=(\tau^\ast\cA,\tau^\ast\al)$; {\it i.e}. $(\cA',\al')$ is such that 

\begin{equation}~\label{eq:A-bar_vs_A-prime}
\left\{ \begin{array}{ll}
\cA'_x = (\cA_2)_{\bar{x}}, & {\rm if\ } x\in X_1;\\ 
\cA'_x=(\cA_1)_{\bar{x}}, & {\rm if\ } x\in X_2;\\ 
\al'_{g_1}=\al_{\bar{g}_1}: (\cA_2)_{s(\bar{g}_1)}\To (\cA_2)_{r(\bar{g}_1)}, & g_1\in \cG_1; \\
\al'_{g_2}=\al_{\bar{g}_2}:(\cA_1)_{s(\bar{g}_2)}\To (\cA_1)_{r(\bar{g}_2)}, & g_2\in \cG_2. \end{array}\right.
\end{equation} 

In fact, the same is true for every Rg Banach bundle over $\cG$. Now define the map
\begin{equation*}
\begin{array}{llll}
\Phi_{12}: & \wRBr(\cG) & \To & \wBr(\cG_1) \\
 & [\cA,\al] & \mto & [\cA_1,\al_1],  
\end{array}
\end{equation*}

$\Phi_{12}$ is well-defined since if $(\cA,\al)\sim_{(\cE,V)}(\cB,\beta)$ in $\fwRBr(\cG)$, then the restriction $(\cE_1,V_1)$ of $(\cE,V)$ over $\cG_1$ induces a Morita equivalence of graded complex D-D bundles $(\cA_1,\al_1)\sim (\cB_1,\beta_1)$ over $\cG_1$. Moreover from the identifications~\eqref{eq:A-bar_vs_A-prime} we see that $\Phi_{12}(-[\cA,\al])=-\Phi([\cA,\al])$. Furthermore, we have clearly $(\cA\hat{\otimes}_X\cB)_i=\cA_i\hat{\otimes}_{X_i}\cB_i$ for $i=1,2$, and that the involution induces an isomorphism of graded complex D-D bundles $\cA_2\hat{\otimes}_{X_2}\cB_2\stackrel{\cong}{\To} \overline{\cA_1\hat{\otimes}_{X_1}\cB_1}=\overline{\cA_1}\hat{\otimes}_{X_2}\overline{\cB_1}$ over $\xymatrix{\cG_2\dar[r]& X_2}$, which shows $\Phi_{12}$ is a group homomorphism. \\

Conversely, if $(\cA_1,\al_1)$ is a graded complex D-D bundle over $\cG_1$, we define the Real graded D-D bundle $(\cA,\al)$ over $\cG$ by setting $\cA:=\cA_1\oplus \overline{\tau^\ast_{|X_2}\cA_1}$, and $\al:=\al_1\oplus \overline{\tau^\ast_{|\cG_2}\al_1}$; then we define 
\begin{equation*}
\begin{array}{llll}
\Phi_{12}' : & \wBr(\cG_1) & \To &\wRBr(\cG)\\
 & [\cA_1,\al_1] & \mto &[\cA_1\oplus\overline{\tau_{|X_2}^\ast \cA_1},\al_1\oplus\overline{\tau^\ast_{|\cG_2}\al_1}].
\end{array}
\end{equation*}
It is straightforward that $\Phi_{12}$ and $\Phi_{12}'$ are inverse of each other.
\end{proof}

\begin{cor}~\label{cor:BrR(G-S)_vs_Br(G)}
Let $\grpd$ be a groupoid. Let the product groupoid $\xymatrix{\cG\times\bfS^{0,1}\dar[r]& X\times\bfS^{0,1}}$ be equipped with the Real structure $(g,\pm1)\mto (g,\mp1)$. Then \[\wRBr(\cG\times\bfS^{0,1})\cong \wBr(\cG).\]	
\end{cor}

\begin{proof}
Apply Proposition~\ref{pro:RBr_vs_Br-special-case} to $\cG=(\cG\times\{+1\})\sqcup (\cG\times \{-1\})$.
\end{proof}

\begin{ex}
The groupoid $\xymatrix{\bfS^{0,1}\dar[r]& \bfS^{0,1}}$ identifies with $\{pt\}\times \{\pm1\}$. Thus from Corollary~\ref{cor:BrR(G-S)_vs_Br(G)} we get 
\[\wRBr(\bfS^{0,1})\cong \wBr(\{pt\})\cong \ZZ_2.\]	
\end{ex}

\begin{df}
Let $\grpd$ be a groupoid. A \emph{graded real~\footnote{Here "real" with a lower-case "r" is to emphasise that the fibres of $\cA$ are $\RR$-$\cstar$-algebras.} D-D bundle} $(\cA,\al)$ over $\cG$ consists of a locally trivial $\cstar$-bundle $p:\cA\To X$, a family of isomorphisms of graded $\RR$-$\cstar$-algebras $\al_g:\cA_{s(g)}\To \cA_{r(g)}$, such that 
\begin{itemize}
\item[(a)] the operation $g\cdot a:=\al_g(a)$ makes $\cA$ into a $\cG$-space with respect to the projection $p$;
\item[(b)] $\al_{gh}=\al_g\circ \al_h, \forall (g,h)\in \cG^{(2)}$;
\item[(c)] the \emph{complexification $(\cA_\CC,\al_\CC)$} of $(\cA,\al)$ defines an element of the collection $\fwBr(\cG)$ of graded complex D-D bundles over $\cG$, where $\cA_\CC:=\cA\otimes_\RR\CC \To X$ is the bundle with fibre $(\cA_\CC)_x:=\cA_x\otimes_\RR\CC$, and for $g\in \cG$, $(\al_\CC)_g:=\al\otimes \Id_\CC$.	
\end{itemize} 	
\end{df}

\begin{df}
Let $\grpd$ be a groupoid. The \emph{graded orthogonal Brauer group $\wBRO(\cG)$} of $\cG$ is defined to be the set or Morita equivalence classes of graded real D-D bundles over $\cG$, where two such bundles $(\cA,\al)$ and $(\cB,\beta)$ are said to be Morita equivalent if and only if their complexifications $(\cA_\CC,\al_\CC)$ and $(\cB_\CC,\beta_\CC)$ are Morita equivalent in $\fwBr(\cG)$.	
\end{df}

We will use the same notations in $\wBr(\cG)$ and $\wBRO(\cG)$ as in Notations~\ref{nota:notations_in_BrR}.

\begin{thm}~\label{thm:Br_vs_BrR_vs_BrO}
Let $\grpd$ be a Real groupoid. 
\begin{enumerate}
\item If the Real structure $\tau$ is fixed point free, then we have an isomorphism 
\begin{equation}~\label{eq:decomp_Br}
	\wBr(\cG)\otimes\ZZ[\nicefrac{1}{2}] \cong (\wRBr(\cG)\oplus \wBr(\cG/_{\tau}))\otimes \ZZ[\nicefrac{1}{2}],		
		\end{equation}
where $\cG/_{\tau}$ is the groupoid $\xymatrix{\cG/_\tau \dar[r] & X/_{\tau}}$ obtained from $\grpd$ by identifying every point $g\in \cG$ with its image by $\tau$.
\item	It $\tau$ is trivial, then every element $\cA\in \wRBr(\cG)$ is a $2$-torsion; {\it i.e}. \[2\cA=0.\]	
	\end{enumerate}	
Furthermore, $\wRBr(\cG)\cong \wBRO(\cG)$. In particular, $\wBRO(\cG)$ is an abelian group under the obvious operations, the zero element being given by the trivial bundle $X\times \RR\To X$ with the $\cG$-action $g\cdot(s(g),t):=(r(g),t)$.  
\end{thm}

We shall mention that property 2 was already proved by D. Saltman in the special case of Azumaya algebras with involution (see~\cite[Theorem 4.4 (a)]{Saltman:Azumaya}). Our result extends his work to infinite-dimensional Real bundles of algebras.

To prove Theorem~\ref{thm:Br_vs_BrR_vs_BrO} we need the 

\begin{lem}~\label{lem:involution_on_Br}
Let $(\cG,\tau)$ be a Real groupoid. Then the assignment $(\cA,\al)\mto (\overline{\tau^\ast\cA},\overline{\tau^\ast\al})$ defines a group involution 
\[
\begin{array}{llll}
\hat{\tau}: &\wBr(\cG) & \To & \wBr(\cG) \\
& \cA & \mto & -\tau^\ast \cA
\end{array}
\]
such that the Real part $\wBr(\cG)_\tau$ is isomorphic to $\wRBr(\cG)$ after tensoring with $\QQ$; more precisely, 
\[
\wBr(\cG)_\tau\otimes \ZZ[{1\over 2}]\cong \wRBr(\cG)\otimes \ZZ[{1\over 2}].
\]	
\end{lem}

\begin{proof}
That $\hat{\tau}$ is a group homomorphism follows from the functorial property of the abelian group $\wBr(\cG)$ with respect to $\cG$ shown in~\cite{Kumjian-Muhly-Renault-Williams:Brauer}. Now let $$\Phi_\CC: \wRBr(\cG)\To \wBr(\cG), \cA\mto \cA$$ be the map consisting of "forgetting the Real structures", and let \[\begin{array}{llll}\Phi_R:& \wBr(\cG) & \To & \wRBr(\cG) \\ & \cA & \mto & \cA+\hat{\tau}(\cA)\end{array}\]
That $\Phi_\CC$ is a well-defined group homomorphism is clear. 

To prove that $\Phi_R$ is well defined, we shall first verify that $(\cA\hat{\otimes}_X\overline{\tau^\ast\cA},\al\hat{\otimes}\overline{\tau^\ast\al})\in \fwRBr(\cG)$ for all $(\cA,\al)\in \fwBr(\cG)$. Let $\sigma=(\sigma_x)_x$ be the family of conjugate-linear isomorphisms of graded complex $\cstar$-algebras $\sigma_x:  \cA_x\hat{\otimes}\overline{\cA_{\bar{x}}}  \To \cA_{\bar{x}}\hat{\otimes} \overline{\cA_x}$ given on homogeneous tensors by 
\begin{equation}
\sigma_x(a\hat{\odot}\flat(b)):= (-1)^{\partial a\cdot\partial b}(b\hat{\odot}\flat(a)).
\end{equation}
Then $\sigma$ is a Real structure on the bundle $\cA\hat{\otimes}_X\overline{\tau^\ast\cA}\To X$, and it is a matter of simple verifications to see that conditions (a)-(c) in Definition~\ref{df-DD-bdle} are satisfied when $(\cA\hat{\otimes}_X\overline{\tau^\ast\cA},\al\hat{\otimes}\overline{\tau^\ast\al})$ is equipped with the involution $\sigma$. 

Suppose further that $(\cA,\al)\sim_{(\cE,V)}(\cB,\beta)$ in $\fwBr(\cG)$. By using the same reasoning we used above for graded complex D-D bundles, one verifies that the graded complex Banach $\cG$-bundle $(\cE\hat{\otimes}_X\overline{\tau^\ast\cE},V\hat{\otimes}\overline{\tau^\ast V})$ admits a Real structure $\sigma^{\cE}$ making it into a Rg Banach $\cG$-bundle. Furthermore, this bundle implements a Morita equivalence $(\cA\hat{\otimes}_X\overline{\tau^\ast\cA},\al\hat{\otimes}\overline{\tau^\ast\al})\sim (\cB\hat{\otimes}_X\overline{\tau^\ast\cB},\beta\hat{\otimes}\overline{\tau^\ast\beta})$ in $\fwBr(\cG)$. Moreover, since by definition 
\begin{align*}
{}_{\cA_x\hat{\otimes}\overline{\cA_{\bar{x}}}}\<\xi\hat{\odot}\flat(\eta),\xi'\hat{\odot}\flat(\eta')\>= {}_{\cA_x}\<\xi,\xi'\>\hat{\odot}{}_{\overline{\cA_{\bar{x}}}}\<\flat(\eta),\flat(\eta')\>, \ {and\ } 
\<\xi\hat{\odot}\flat(\eta),\xi'\hat{\odot}\flat(\eta')\>_{\cB_x\hat{\otimes}\overline{\cB_{\bar{x}}}}
\end{align*}
for every $x\in X, \xi,\xi',\eta,\eta'\in \cE_x$, we see that the inner products $_{\cA\hat{\otimes}_X\overline{\tau^\ast\cA}}\<\cdot,\cdot\>$ and $\<\cdot,\cdot\>_{\cB\hat{\otimes}_X\overline{\tau^\ast\cB}}$ of $\cE\hat{\otimes}_X\overline{\tau^\ast\cE}$ intertwine the Real structures; hence we have a Morita equivalence $$(\cA\hat{\otimes}_X\overline{\tau^\ast\cA},\al\hat{\otimes}\overline{\tau^\ast\al})\sim_{(\cE\hat{\otimes}_X\overline{\tau^\ast\cE},V\hat{\otimes}\overline{\tau^\ast V})}(\cB\hat{\otimes}_X\overline{\tau^\ast\cB},\beta\hat{\otimes}\overline{\tau^\ast\beta})$$ in $\fwRBr(\cG)$, so that $\Phi_R$ is well defined.

$\Phi_R$ is a group homomorphism since $\wBr(\cG)$ is an abelian group and since $\hat{\tau}$ is linear; {\it i.e}. for every $\cA,\cB\in \wBr(\cG)$, $$(\cA+\hat{\tau}(\cA))+(\cB+\hat{\tau}(\cB))=(\cA+\cB)+\hat{\tau}(\cA+\cB).$$ 

Let us verify that up to inverting $2$, $\Phi_R'$ and $\Phi_\CC$ are inverse of each other, where $\Phi_R'$ is the restriction of $\Phi_R$ on the fixed points $\wBr(\cG)_\RR$ of $\hat{\tau}$. First observe that if $(\cA,\al)\in \fwRBr(\cG)$, then the Real structure of $\cA$ induces an isomorphism $(\cA,\al)\cong (\overline{\tau^\ast\cA},\overline{\tau^\ast \al})$ in $\fwBr(\cG)$. Thus, for $\cA\in \wRBr(\cG)$, we get $(\Phi_R\circ \Phi_\CC)(\cA)=2\cA$. Suppose now that $\cA\in \wBr(\cG)_\RR$. Then $(\Phi_\CC\circ\Phi_R')(\cA)=\Phi_\CC(2\cA)=2\cA$, which completes the proof.
\end{proof}

\begin{rem}~\label{rem:characterization_Morita_Br}
It is straightforward, by using Lemma~\ref{lem:involution_on_Br}, that  one has a similar characterisation for graded complex D-D bundles as that of Corollary~\ref{cor:characterization_Morita_BrR}. 	
\end{rem}

\begin{proof}[Proof of Theorem~\ref{thm:Br_vs_BrR_vs_BrO}]
1. It suffices to show that the imaginary part ${}^\cI \wBr(\cG)$ with respect to the involution $\hat{\tau}:\wBr(\cG)\To \wBr(\cG)$ of Lemma~\ref{lem:involution_on_Br} ({\it i.e.}, the set of all $\cA\in \wBr(\cG)$ such that $\hat{\tau}(\cA)=-\cA$) is isomorphic to $\wBr(\cG/_\tau)$ (after inverting $2$), and then we will apply~\cite[Lemma 1.4]{Moutuou:Real.Cohomology}.

Assume $(\cA,\al)\in \fwRBr(\cG)$ is such that $\hat{\tau}(\cA)=-\cA$. Then thanks to Corollary~\ref{cor:characterization_Morita_BrR}, there exists a Rg Hilbert $\cG$-bundle $(\hat{\mathscr{H}},U)$ and an isomorphism of Rg D-D bundles 
\begin{equation}~\label{eq:isom_A-tens-K(H)-imaginary}
(\tau^\ast\cA\hat{\otimes}_X\cK(\hat{\mathscr{H}}),\overline{\tau^\ast\al}\hat{\otimes}Ad_U)\stackrel{\cong}{\To} (\cA\hat{\otimes}_X\cK(\hat{\mathscr{H}}),\bar{\al}\hat{\otimes}Ad_U).
\end{equation}
We then obtain a Rg D-D bundle $(\cA/_\tau,\al^\tau)$ over $\xymatrix{\cG/_\tau \dar[r]& X/_\tau}$ by setting 
\begin{equation}
	\cA/_\tau:=\cA\hat{\otimes}_X\cK(\hat{\mathscr{H}})\hat{\otimes}_X\cK(\tau^\ast\hat{\mathscr{H}}), \ {\rm and\ } \al^\tau:=\al\hat{\otimes}Ad_U\hat{\otimes}Ad_{\tau^\ast U}),
\end{equation}
with projection $p_\tau:\cA/_\tau\To X/_\tau$ given by $$p_\tau(a\hat{\odot}T\hat{\odot}T')=p(a), \ {\rm for\ } a\hat{\odot}T\hat{\odot}T'\in \cA_x\hat{\otimes}\cK(\hat{\mathscr{H}}_x)\hat{\otimes}\cK(\hat{\mathscr{H}}_{\bar{x}}).$$
Next define the map 
\begin{equation*}
\begin{array}{llll}
\Psi_\tau: & {}^\cI \wRBr(\cG) & \To & \wBr(\cG/_\tau)\\
& \cA & \mto & \cA/_\tau.
\end{array}
\end{equation*}
This definition does not depend on the choice of $(\hat{\mathscr{H}},U)$, for if $(\hat{\mathscr{H}'},U')$ is another Rg Hilbert $\cG$-bundle such that $(\tau^\ast\cA\hat{\otimes}_X\cK(\hat{\mathscr{H}}'),\tau^\ast\al\hat{\otimes}Ad_{U'})\cong (\cA\hat{\otimes}_X\cK(\hat{\mathscr{H}}'),\al\hat{\otimes}Ad_{U'})$, then putting $$\cA'/_\tau:=\cA\hat{\otimes}_X\cK(\hat{\mathscr{H}}')\hat{\otimes}_X\cK(\tau^\ast\hat{\mathscr{H}}'),$$ we get 
\[
\begin{array}{ll}
\cA/_\tau\hat{\otimes}_{X/_\tau}\overline{\cA'/_\tau} & \cong \cA\hat{\otimes}_X\overline{\cA}\hat{\otimes}_X\hat{\otimes} \cK(\hat{\mathscr{H}}\hat{\otimes}_X\tau^\ast\hat{\mathscr{H}}\hat{\otimes}_X\overline{\hat{\mathscr{H}}'}\hat{\otimes}_X\overline{\tau^\ast\hat{\mathscr{H}}'}) \\
& \cong \cK(\hat{\mathscr{H}}\hat{\otimes}_X\hat{\mathscr{H}}\hat{\otimes}_X\tau^\ast\hat{\mathscr{H}}\hat{\otimes}_X\overline{\hat{\mathscr{H}}'}\hat{\otimes}_X\overline{\tau^\ast\hat{\mathscr{H}}'}).
\end{array}
\]
Moreover, $\cK(\hat{\mathscr{H}}\hat{\otimes}_X\hat{\mathscr{H}}\hat{\otimes}_X\tau^\ast\hat{\mathscr{H}}\hat{\otimes}_X\overline{\hat{\mathscr{H}}'}\hat{\otimes}_X\overline{\tau^\ast\hat{\mathscr{H}}'}$ defines a graded Hilbert $\cG/_\tau$-bundle. Hence, by Corollary~\ref{cor:characterization_Morita_BrR} and Remark~\ref{rem:characterization_Morita_Br} we see that $\cA/_\tau=\cA'/_\tau$ in $\wBr(\cG/_\tau)$. $\Psi_\tau$ is a group homomorphism by commutativity of the graded tensor product.

Conversely, denote by $\pi_\tau:\cG\To \cG/_\tau$ the canonical projection. Then the pull-back of a graded complex D-D bundle $(\cA,\al)\in \fwBr(\cG/_\tau)$ is a graded complex D-D bundle $(\cA',\al'):=(\pi_\tau^\ast\cA,\pi_\tau^\ast\al)\in \fwBr(\cG)$ which clearly verifies $(\tau^\ast\cA',\tau^\ast\al')\cong (\cA',\al')$ in $\fwBr(\cG)$ (this is because for all $x\in X$ we have $\cA'_x=\cA'_{\bar{x}}$); so $\hat{\tau}(\cA')=-\cA'$ and $\cA'\in {}^\cI\wBr(\cG)$. Thus the pull-back map $\pi_\tau^\ast$ induces a group homomorphism
\[ 
\begin{array}{llll}
\pi_\tau^\ast: & \wBr(\cG/_\tau) & \To & {}^\cI\wBr(\cG) \\
 & \cA & \mto & \cA':= \pi_\tau^\ast\cA.
\end{array} 
\]
Now, for all $\cA\in \wBr(\cG/_\tau)$ we have $(\pi_\tau^\ast\cA)/_\tau=\cA$ since $\tau^\ast\pi_\tau^\ast\cA=\pi_\tau^\ast\cA$ and so that a graded Hilbert $\cG$-bundle $\hat{\mathscr{H}}$ such that relation~\eqref{eq:isom_A-tens-K(H)-imaginary} holds for the graded complex D-D bundle $(\pi_\tau^\ast\cA,\pi_\tau^\ast\al)$ is the trivial one $X\times \CC\To X$. This shows that $\Psi_\tau\circ \pi_\tau^\ast=\Id$. Also, one clearly has $\pi_\tau^\ast \circ \Psi_\tau=\Id$, which gives the isomorphism ${}^\cI\wBr(\cG)\cong \wBr(\cG/_\tau)$. From Lemma~\ref{lem:involution_on_Br}, we obtain the desired isomorphism~\eqref{eq:decomp_Br}. 

2.  We always have $\cA+\overline{\cA}=0$ in $\wRBr(\cG)$ for all $(\cA,\al)\in \fwRBr(\cG)$. Moreover, we have already seen in the end of the proof of Lemma~\ref{lem:involution_on_Br} that the Real structure of $\cA$ induces an isomorphism of Rg D-D bundles $(\cA,\al)\cong (\overline{\tau^\ast\cA},\overline{\tau^\ast\al})$. In particular, if $\tau:\cG\To\cG$ is trivial, we have $(\cA,\al)\cong (\overline{\cA},\bar{\al})$; hence $\cA=-\cA$ in $\wRBr(\cG)$.

Furthermore, $\tau$ being trivial, each fibre of $\cA$ is in fact a Rg elementary $\cstar$-algebra, and then the complexification of a graded real elementary $\cstar$-algebra. $(\cA,\al)$ is then the complexification of a graded real D-D bundle over $\cG$. Conversely, every complexification $(\cA_\CC,\al_\CC)$ of a graded real D-D bundle $(\cA,\al)$ over $\cG$ is a Rg D-D bundle whose Real structure is carried out by $\CC$; {\it i.e}. $\overline{a\otimes_\RR \lambda}:=a\otimes_\RR \bar{\lambda}$ for $a\otimes_\RR\lambda\in \cA_x\otimes_\RR\CC$. This process is easily seen to provide an isomorphism $\wRBr(\cG)\cong \wBRO(\cG)$.
\end{proof}

Observe that any Rg D-D bundle $(\cA,\al)$ can also be considered as a graded real D-D bundle $(\cA_{real},\al_{real})$ by forgetting the complex structure of the fibres. Moreover, the conjugate bundle of real $\cstar$-algebras $(\cA_{real},\al_{real})$ identifies to itself. Hence, if the involution $\tau$ of $\cG$ is fixed point free, we have $\overline{\tau^\ast\cA_{real}}=\tau^\ast\cA_{real}\cong \cA_{real}$, which means that $(\cA_{real},\al_{real})$ is a bundle of graded real elementary $\cstar$-algebras over the quotient groupoid $\xymatrix{\cG/_\tau\dar[r]&X/_\tau}$. We therefore have the

\begin{pro}
Suppose $\grpd$ is endowed with a fixed point free involution $\tau$. Then there is a group homomorphism 
\[\Psi_{real}:\wRBr(\cG)\To \wBRO(\cG/_\tau)\]
obtained by "forgetting the complex structures" of Rg graded D-D bundles over $\cG$.	
\end{pro}

\begin{rem}
Beware that $\Psi_{real}$ is not injective; indeed $\Psi_{real}(\overline{\cA})=\Psi_{real}(\cA)$ for all $\cA\in \wRBr(\cG)$, while in general $\overline{\cA}\neq \cA$ in $\wRBr(\cG)$.	
\end{rem}

\section{Elementary involutive triples and types of Rg D-D bundles}

In this section we define the \emph{type} of a Rg D-D bundle over a Real groupoid. We start by introducing few notions.

\begin{df}
An \emph{elementary involutive triple} $(\wKK,\wKK^-,\bft)$ consists of a graded elementary $\cstar$-algebra $\wKK$, a graded $\cstar$-algebra $\wKK^-$ isomorphic to the conjugate $\cstar$-algebra of $\wKK$, and a conjugate linear isomorphism $\bft:\wKK\To \wKK^-$ of graded $\cstar$-algebra. Such triple will be represented by the map $\bft$. Denote by $\wfK$ the collection of all elementary involutive triples.

A morphism from $\bft$ to $\bft'$ is the data of homomorphisms of graded $\cstar$-algebras $\vp:\wKK\To \wKK'$, and $\vp^-:\wKK^-\To \wKK'^-$ such that the following diagram commutes 
\begin{eqnarray}
\xymatrix{\wKK\ar[d]^\bft \ar[r]^\vp & \wKK' \ar[d]^{\bft'}\\ \wKK^- \ar[r]^{\vp^-} & \wKK'^-}
\end{eqnarray}
Finally, we define the sum in $\wfK$ by:
\[\bft+\bft':=(\wKK\hat{\otimes}\wKK',\wKK^-\hat{\otimes}\wKK'^-,\bft\hat{\otimes}\bft').\]	
\end{df}

\begin{ex}
The Real structure "$bar$" of $\wK_0$ induces an isomorphism of $\wK_0$ into its conjugate algebra. We then have an elementary involutive triple $\bft_0=(\wK_0,\wK_0,bar)$.	
\end{ex}

\begin{dflem}
Two elements $\bft,\bft'\in \wfK$ are said to be \emph{stably isomorphic} if and only if $\bft+\bft_0$ is isomorphic to $\bft'+\bft_0$; in this case, we write $\bft\cong_s \bft'$. The set of stable isomorphism classes of elements of $\wfK$ forms an abelian group $\Inv \wfK$ under the sum defined above. The inverse of $\bft$ in $\Inv \wfK$ is the stable isomorphisms class of $$-\bft:= (\wKK^-,\wKK,\bft^{-1}).$$
The class of $\bft$ in $\Inv \wfK$ will also be denoted by $\bft$.	
\end{dflem}

\begin{proof}
It is straightforward that $\bft+\bft'=\bft'+\bft$ in $\Inv \wfK$. Moreover, we have $$\bft-\bft=(\wKK\hat{\otimes}\wKK^-,\wKK^-\hat{\otimes}\wKK,\bft\hat{\otimes}\bft^{-1})\cong (\wKK\hat{\otimes}\wKK^-,\wKK\hat{\otimes}\wKK^-,\bft'),$$ via the isomorphism $(\Id_{\wKK\hat{\otimes}\wKK^-}, \vp')$, where $\vp':\wKK^-\hat{\otimes}\wKK\To \wKK\hat{\otimes}\wKK^-$ is the canonical isomorphism $\vp'(T\hat{\otimes}T'):=(-1)^{\partial T\partial T'}T'\hat{\otimes}T$, and $\bft':=\vp'\circ (\bft\hat{\otimes}\bft^{-1})$. Thus, $\bft-\bft\cong_s \bft_0$. 	
\end{proof}

We can recover the group $\wRBr(\ast)$ from $\Inv \wfK$. More precisely, suppose $\bft=-\bft$, and $$(\vp,\vp'):(\wKK\hat{\otimes}\wK_0,\wKK^-\hat{\otimes}\wK_0,\bft\hat{\otimes}bar)\To (\wKK^-\hat{\otimes}\wK_0,\wKK\hat{\otimes}\wK_0,\bft^{-1}\hat{\otimes}bar)$$ is an isomorphism. Then, $\vp'\circ (\bft\hat{\otimes}bar)=(\bft^{-1}\hat{\otimes}bar)\circ \vp$ is a Real structure on the graded elementary $\cstar$-algebra $\wKK\hat{\otimes}\cK(\hat{\cH})$. Moreover, if $(\vp_0,\vp_0')$ is another isomorphism, it is easy to check that $\vp'\circ (\bft\hat{\otimes}bar)$ and $\vp_0'\circ (\bft\hat{\otimes}bar)$ are conjugate, hence define the same element of $\wRBr(\ast)$. Conversely, any Real graded elementary $\cstar$-algebra is obviously a $2$-torsion of $\Inv \wfK$. We then have proved the following 

\begin{lem}
The group $\wRBr(\ast)$ is isomorphic to the subgroup of $\Inv\wfK$ of elements of order $2$.
\end{lem}

Now let us return to the study of Rg D-D bundles over Real groupoids.

\begin{pro}~\label{pro:type-map}
Let $(\cA,\al)\in \fwRBr(\cG)$. Then each fiber $\cA_x$ gives rise to an element $\bft_x^\cA\in \Inv\wfK$, and the family $\bft^\cA:=(\bft_x^\cA)_{x\in X}$ defines a cohomology class in $\check{H}R^0(\cG_\bullet,\Inv\wfK)$. This process defines a group homomorphism 
\[
\bft: \wRBr(\cG) \To \check{H}R^0(\cG_\bullet,\Inv\wfK),
\]
which is surjective. 	
\end{pro}

\begin{proof}
Denote by $\sigma$ the Real structure of $\cA$. Over all $x\in X$, there is a conjugate linear isomorphism of graded $\cstar$-algebras $\sigma_x:\cA_x\To \cA_{\bar{x}}$. Then the graded elementary (complex) $\cstar$-algebras $\cA_x$ and $\cA_{\bar{x}}$ are of the same parity. Let $(\cU,\vp)$ be a local trivialisation of the graded elementary complex $\cstar$-bundle $\cA$ such that $\cU=(U_i)$ is a Real open cover of $X$. Then the isomorphisms $\vp_i:U_i\times \wKK_i\To \cA_{|U_i}$ induces a family of graded isomorphisms $\vp_x:\wKK_x\To \cA_x$. Then  $\bft_x:= (\wKK_x,\wKK_{\bar{x}},t_x)$, where $t_x:=\vp_{\bar{x}}\circ \tau_x\circ \vp_x$, is an element of $\Inv\wfK$, and the assignment $X\ni x\mto \bft_x \in \Inv\wfK$ is a locally constant $\cG$-invariant Real function. Indeed, the $\cG$-invariance (i.e. $\bft_{r(g)}=\bft_{s(g)}$ in $\Inv\wfK$ for all $g\in \cG$) comes from the commutative diagram 
\[
\xymatrix{\wKK_{s(g)} \ar[r]^{\vp_{s(g)}}\ar[d]^{\bft_{s(g)}} & \cA_{s(g)} \ar[r]^{\al_g}\ar[d]^{\sigma_{s(g)}} & \cA_{r(g)} \ar[r]^{\vp_{r(g)}^{-1}} \ar[d]^{\sigma_{r(g)}} & \wKK_{r(g)} \ar[d]_{\bft_{r(g)}} \\ \wKK_{s(\bar{g})} \ar[r]^{\vp_{s(\bar{g})}} & \cA_{s(\bar{g})} \ar[r]^{\al_{\bar{g}}} & \cA_{\bar{g}} \ar[r]^{\vp_{r(\bar{g})}} & \wKK_{r(\bar{g})}} 
\]
Moreover, since $\sigma$ is a continuous function, $\bft^\cA:X\ni x\mto \bft_x \in \Inv\wfK$ is locally constant. Hence $\bft^\cA\in \check{H}R^0(\cG_\bullet,\Inv\wfK)$. 

That $\bft^{\bar{\cA}}=-\bft^\cA$ and $\bft^{\cA+\cB}=\bft^\cA+\bft^\cB$ is clear from the definition of the sum and the inverse in $\Inv\wfK$, and from the definition of the conjugate bundle and the tensor product of Rg D-D bundles.  Observe that from the construction of $\wK_\cG$, $\bft^{\wK_\cG}=\bft_0=0$ since $\wK_\cG\cong \coprod_{x\in X}\cK(L^2(\cG^x))\otimes \wK_0$ with involution given by $\wK_x\ni \vp\otimes T\mto \tau(\vp)\otimes Ad_{J_{0,\RR}}(T) \in \wK_{\bar{x}}$. Thus, if $\cA=\cB$ in $\wRBr(\cG)$, we have (thanks to Lemma~\ref{lem:A+K_G=A} and Lemma~\ref{cor:characterization_Morita_BrR}) $\cA+\bar{\cB}+\wK_\cG=\cK(\hat{\cH}_\cG\hat{\otimes}_X\hat{\mathscr{H}})=0$; hence $\bft^{\cA-\cB}=\bft^\cA-\bft^\cB=0$, which shows that $\bft:\wRBr(\cG)\To \check{H}R^0(\cG_\bullet,\Inv\wfK)$ is a group homomorphism. It is surjective since for all $\bft\in \check{H}R^0(\cG_\bullet,\Inv\wfK)$,
\begin{equation}~\label{eq:construction-K_(G,t)}
\wK_{\cG,\bft}:=\coprod_{x\in X}\wK_x\hat{\otimes}\wKK_{\bft_x},
\end{equation} 
equipped with the obvious involution and $\cG$-action, defines a Rg D-D bundle over $\cG$. 
\end{proof}

\begin{df}
For $(\cA,\al)\in \fwRBr(\cG)$, the element $\bft^\cA$ of Proposition~\ref{pro:type-map} is called \emph{the type of $(\cA,\al)$}. The homomorphism $\bft:\wRBr(\cG)
\To \check{H}R^0(\cG_\bullet,\Inv\wfK)$ is called \emph{the type map}. 
\end{df}

\begin{df}
A Rg D-D bundle $(\cA,\al)$ is said to be \emph{of type $i \mod 8$} if $\bft^\cA$ is the constant function $\bft^\cA=i \in \ZZ_8\subset \Inv\wfK$. By $\wRBr_i(\cG)$ we denote the set of Morita equivalence classes of Rg D-D bundles of type $i \mod 8$ over $\grpd$. Next, we define 
\[\wRBr_\ast(\cG):= \bigoplus_{i=0}^7 \wRBr_i(\cG).\]
\end{df}

\begin{ex}
Let $\hat{\cH}$ be, as usual, equipped with the Real structure $J_{,\RR}$. Then $\wK_0\To \cdot$ is a Rg D-D bundle of type $0$ over $\xymatrix{\wPU(\hat{\cH})\dar[r]& \cdot}$, where the Real $\wPU(\hat{\cH})$-action is given by $Ad$; \emph{i.e.} $[u]\cdot T:=Ad_u(T)$, for $[u]\in \wPU(\hat{\cH})\cong \Aut^{(0)}(\wK_0), T\in \wK_0$.	
\end{ex}

We have the following easy result which shows that the study of $\wRBr(\cG)$ reduces to that of Rg D-D bundles of type $0$.

\begin{pro}~\label{pro:simplification-BrR}
Let $\grpd$ be a Real groupoid. Then $\wRBr_0(\cG)$ is a subgroup of $\wRBr(\cG)$. Furthermore, the group sequences 
\begin{eqnarray}
	0 \To \wRBr_0(\cG) \stackrel{\iota_0}{\To} \wRBr(\cG) \stackrel{\bft}{\To} \check{H}R^0(\cG_\bullet,\Inv\wfK) \To 0\\
0 \To \wRBr_0(\cG) \stackrel{\iota_0}{\To} \wRBr_\ast(\cG) \stackrel{\bft}{\To} \check{H}R^0(\cG_\bullet,\ZZ_8) \To 0,	
	\end{eqnarray}	
where $\iota_0$ is the inclusion homomorphism, are split-exact. Therefore, we have two isomorphisms of abelian groups 
\[\wRBr(\cG)\cong \check{H}R^0(\cG_\bullet,\Inv\wfK)\oplus \wRBr_0(\cG), \ {\rm and } \quad \wRBr_\ast(\cG)\cong \check{H}R^0(\cG_\bullet,\ZZ_8)\oplus \wRBr_0(\cG).\] 
\end{pro}

\begin{proof}
We only prove the first sequence, from which we deduce the second one. It is clear that $\bft\circ \iota_0 =0$ and $\iota_0$ is an injective homomorphism. We also proved in Proposition~\ref{pro:type-map} that $\bft$ was surjective. To show the sequence splits, we only have to verify that the correspondence $\bft\mto \wK_{\cG,\bft}$, where $\wK_{\cG,\bft} \To X$ is the Rg D-D bundle given by~\eqref{eq:construction-K_(G,t)} defines a group homomorphism $\check{H}R^0(\cG_\bullet,\Inv\wfK)\To \wRBr(\cG)$. This is immediate from construction: we have $\wK_{\cG,\bft+\bft'}\cong \wK_{\cG,\bft}\hat{\otimes}_X\wK_{\cG,\bft'}$, and a routine verification shows that any isomorphism $\bft+\bft_0 \cong \bft'+\bft_0$ induces an isomorphism of Rg D-D bundles $\wK_{\cG,\bft+\bft_0}\cong \wK_{\cG,\bft'+\bft_0}$ so that $\wK_{\cG,\bft}=\wK_{\cG,\bft'}$ in $\wRBr(\cG)$ if $\bft\sim_s \bft'$. Also from the definition of $-\bft$, we have $\wK_{\cG,-\bft}=\overline{\wK_{\cG,\bft}}=-\wK_{\cG,\bft}$. Finally it is obvious that $\wK_{\cG,\bft}$ is of type $\bft$.  
\end{proof}

\section{Generalised classifying morphisms}~\label{par:classifying-morph}

Let $\hat{\cH}$ be the usual Rg Hilbert space with the involution induced by the degree $0$ Real structure $J_{0,\RR}$ on $\hat{\cH}$ (see Appendix A). Let the algebra $\cL(\hat{\cH})$ of bounded linear operators on $\hat{\cH}$ be endowed with the obvious grading and the Real structure $\bar{T}(h):=\overline{T(\bar{h})}, T\in \cL(\hat{\cH}), h\in \hat{\cH}$. Denote by $\wU(\hat{\cH})$ the group of all homogeneous unitaries in $\cL(\hat{\cH})$ ({\it i.e}. unitaries of degree $0$ or $1$)~\cite{Parker:Brauer} equipped with the Real structure inherited from that of $\cL(\hat{\cH})$. Next, we define the Real group $\wPU(\hat{\cH})$ to be the quotient $\wU(\hat{\cH})/\uc$, where $\uc$ is given the the Real structure defined by complex conjugation (it is obvious that the action of $\uc$ on $\wU(\hat{\cH})$ is compatible with the involutions).\\ 

In this section we are restricting attention to the class of the Rg D-D bundle $(\wK_0,Ad)$ of type $0$ over the Real groupoid $\xymatrix{\wPU(\hat{\cH})\dar[r] & \cdot}$, where $\wPU(\hat{\cH})$ is equipped with the compact-open topology and  

\begin{df}
Let $(\cA,\al)\in \fwRBr(\cG)$ of type $0$. A \emph{generalised classifying morphism} for $(\cA,\al)$ is a generalised Real homomorphism $P:\cG\To \wPU(\hat{\cH})$ such that $(\cA,\al)\cong (\wK_0^P,Ad^P)$ as Rg D-D bundles.	
\end{df}

\begin{rem}
Note that $\wK_0^P=P\times_{\wPU(\hat{\cH})}\wK_0:= P\times \wK_0/_{\sim}$, where the equivalence relation is $(\vp,T)\sim (\vp\cdot [u],[u^{-1}]\cdot T)$ for $[u]\in \wPU(\hat{\cH})$. The Real $\cG$-action by automorphisms is given by the Real (left) $\cG$-action on $P$. 
\end{rem}
 
Before going on the study of generalised classifying morphisms, we shall say something about generalised Real homomorphisms $\cG\To \wPU(\hat{\cH})$. First of all, recall from~\cite[Remark 2.50]{Moutuou:Real.Cohomology} that although the Real group $\wPU(\hat{\cH})$ is not abelian, it still is possible to define Real $\wPU(\hat{\cH})$-valued \v{C}ech $1$-cocycles over any Real groupoid $\cG$, hence to form the set $\check{H}R^1(\cG_\bullet,\wPU(\hat{\cH}))$, and we have a set-theoretic bijection between $\check{H}R^1(\cG_\bullet,\wPU(\hat{\cH}))$ and $\Hom_{\RG}(\cG,\wPU(\hat{\cH}))$. 

What's more, when identified with $\Hom_{\RG_\Omega}(\cG,\wPU(\hat{\cH}))$, the set $\Hom_{\RG}(\cG,\wPU(\hat{\cH}))$ admits the structure of abelian monoid defined as follows. Fix an isomorphism $$\hat{\cH}\hat{\otimes}\hat{\cH}\stackrel{\cong}{\To}\hat{\cH}$$ of Rg Hilbert spaces. Then the map $$\wPU(\hat{\cH})\times \wPU(\hat{\cH})\ni ([u_1],[u_2])\mto [u_1\hat{\otimes}u_1]\in \wPU(\hat{\cH}\hat{\otimes}\hat{\cH}) \cong \wPU(\hat{\cH})$$ is a Real a homeomorphism, where the unitary $u_1\hat{\otimes}u_2$ is given on $\hat{\cH}\hat{\otimes}\hat{\cH}$ by $$(u_1\hat{\otimes}u_2)(\xi_1\hat{\otimes}\xi_2):=(-1)^{\partial u_2\cdot \partial u_1}u_1(\xi_1)\hat{\otimes}u_2(\xi_2).$$ Given $p_1,p_2\in \Hom_{\RG_\Omega}(\cG,\wPU(\hat{\cH}))$, we may, without loss of generality, assume that they are represented on the same open Real cover $\cU$ of $X$; \emph{i.e.} $p_1,p_2:\cG[\cU]\To \wPU(\hat{\cH})$ are strict Real morphisms. Henceforth, the map 
\begin{equation}
\begin{array}{lll}
p_1\hat{\otimes}p_2: &\cG[\cU]  \To & \wPU(\hat{\cH}) \\
                     & \g  \mto & p_1(\g)\hat{\otimes}p_2(\g)
\end{array}
\end{equation}
becomes a well defined strict Real homomorphism. Therefore we have:

\begin{dflem}
For $[P_1],[P_2]\in \Hom_{\RG}(\cG,\wPU(\hat{\cH}))$. We define the sum $$[P_1]+[P_2]:=[P_1\otimes P_2],$$
where $P_1\otimes P_2$ is the generalised homomorphism from $\grpd$ to $\xymatrix{\wPU(\hat{\cH})\dar[r]& \cdot}$ obtained by composing the corresponding morphism $p_1\hat{\otimes}p_2\in \Hom_{\RG_\Omega}(\cG,\wPU(\hat{\cH}))$ with the generalised Real morphism induced by a canonical Morita equivalence $\cG\sim \cG[\cU]$. Then $\Hom_{\RG}(\cG,\wPU(\hat{\cH}))$ is an abelian monoid with respect to this operation.	
\end{dflem}

\begin{rem}
The same reasoning applies to the Real group $\U^0(\hat{\cH})$: the operation of tensor product of cocycles makes the set $\Hom_{\RG}(\cG,\U^0(\hat{\cH}))\cong \check{H}R^1(\cG_\bullet,\U^0(\hat{\cH}))$ into an abelian monoid. Similarly the corresponding operation in $\Hom_{\RG}(\cG,\U^0(\hat{\cH}))$ is denoted additively.
\end{rem}

We list few properties for generalised classifying morphisms.

\begin{pro}~\label{pro:sum-of-classifying}
If $P_1$ and $P_2$ are generalised classifying morphisms for $(\cA_1,\al_1)$ and $(\cA_2,\al_2)$, respectively, then $P_1\otimes P_2$ is a generalised classifying morphism for the Rg tensor product $(\cA_1\hat{\otimes}_X\cA_2,\al_1\hat{\otimes}\al_2)$. 	
\end{pro}

\begin{proof}
Up to considering the pull-back of $\cA_i, i=1,2$ along the canonical Real inclusion $\cG[\cU]\hookrightarrow \cG$, we can suppose that the $(\cA_i,\al_i)$ are Rg D-D bundles over the Real cover groupoid $\cG[\cU]$, where $\cU$ is an open Real cover on which the morphisms $p_i\in \Hom_{\RG_\Omega}(\cG,\wPU(\hat{\cH}))$ corresponding to $P_i$ are represented. The isomorphisms $(\cA_i,\al_i)\cong (\wK_0^{P_i},Ad^{P_i})$ mean that the pull-back $(p_i^\ast \wK_0,p_i^\ast Ad)$ is isomorphic to $(\cA_i,\al_i), i=1,2$. We thus have reduced the proposition to showing that $$(\cA_1\hat{\otimes}_Y\cA_2,\al_1\hat{\otimes}\al_2)\cong ((p_1\hat{\otimes}p_2)^\ast \wK_0,(p_1\hat{\otimes}p_2)^\ast Ad),$$ where $Y=\coprod_{i\in I}U_i$. But this is clear by using the functorial property of $\fwRBr(\cdot)$ in the category $\RG_s$ of Real groupoids and strict Real homomorphisms and the isomorphism of Rg D-D bundles $(\wK_0\hat{\otimes}_{\cdot}\wK_0,Ad\hat{\otimes}Ad)\cong (\wK_0,Ad)$ over $\wPU(\hat{\cH})$.
\end{proof}

\begin{pro}
Suppose $Z:\Ga\To \cG$ is a generalised Real homomorphism. Let $(\cA,\al)\in \fwRBr(\cG)$ be of type $0$. If $P$ is a generalised classifying morphism for $(\cA,\al)$, then $P\circ Z$ is a generalised classifying morphism for $(\cA^Z,\al^Z)\in \fwRBr(\Ga)$.	
\end{pro}

\begin{proof}
This is a consequence of the cofunctorial property of $\fwRBr(\cdot)$ in the category $\RG$: \emph{i.e.} $\cA\cong \wK_0^P$ implies $\cA^Z\cong (\wK_0^P)^Z\cong \wK_0^{P\circ Z}$.	
\end{proof}

\begin{pro}~\label{pro:uniqueness-of-P}
Let $(\cA_1,\al_1), (\cA_2,\al_2)$ be isomorphic Rg D-D bundles of type $0$ over $\grpd$. If $P_1$ and $P_2$ are generalised classifying morphisms for $(\cA_1,\al_1)$ and $(\cA_2,\al_2)$, respectively, there exists a $\cG$--$\wPU(\hat{\cH})$-equivariant Real isomorphism $P_1\cong P_2$.	
\end{pro}

\begin{proof}
Let $f_i:\cA_i\To P_i\times_{\wPU(\hat{\cH})}\wK_0, i=1,2$, and $\phi:\cA_1\To \cA_2$ be isomorphisms of Rg D-D bundles. Then $h:=f_2\circ \phi\circ f_1^{-1}:P_1\times_{\wPU(\hat{\cH})}\wK_0\To P_2\times_{\wPU(\hat{\cH})}\wK_0$ is an isomorphism of Rg D-D bundles over $\cG$. $P_1\To X$ and $P_2\To X$ being $\wPU(\hat{\cH})$-principal bundles, it follows from the theory of principal bundles (see for instance Husem\"oller~\cite[\S 4.6]{Husemoller:Fibre}) that there exists an isomorphism of $\wPU(\hat{\cH})$-principal bundles $f:P_1\To P_2$ over $X$ such that $h([\vp,T])=[f(\vp),T]$ for all $[\vp,T]\in P_1\times_{\wPU(\hat{\cH})}\wK_0$. Moreover, $f$ must be Real since $h$ is. Also, since $h$ is $\cG$-equivariant, we have $h([g\cdot \vp,T])=g\cdot h([\vp,T])=g\cdot[f(\vp),T]=[g\cdot f(\vp),T]$, so that $[f(g\cdot \vp),T]=[g\cdot f(\vp),T], \forall [\vp,T]\in P_1\times_{\wPU(\hat{\cH})}\wK_0$. $f$ is thus an isomorphism of generalised Real homomorphisms $P_1\cong P_2:\cG\To \wPU(\hat{\cH})$. 	
\end{proof}

From Proposition~\ref{pro:uniqueness-of-P} we deduce the following

\begin{cor}~\label{cor:uniqueness-of-P}
If there exists a generalised classifying morphism for $(\cA,\al)\in \fwRBr(\cG)$, then it is unique up to isomorphisms of generalised Real homomorphisms.	
\end{cor}

The existence of generalised classifying morphisms is the content of the next section.


\section{Construction of $P$}

It is known~\cite{Tu:Twisted_Poincare} that graded complex D-D bundles are, in some sense, classified by the groupoid $\xymatrix{\wPU(\hat{\cH})\dar[r] & \cdot}$; \emph{i.e}. giving a graded complex D-D bundle $\cA$ over $\cG$ is equivalent to giving a generalised morphism $P\in \Hom_{\mathfrak{G}}(\cG,\wPU(\hat{\cH}))$, where $\mathfrak{G}$ is the category of topological groupoids and isomorphism classes of generalised morphisms. In view of the isomorphism established in Lemma~\ref{lem:involution_on_Br}, it is natural to expect a similar correspondence in the category of Real spaces. We show that the $\cG$-equivariant $\wPU(\hat{\cH})$-principal bundle associated to a Rg D-D bundle admits a natural involution turning it into an element in $\Hom_{\RG}(\cG,\wPU(\hat{\cH}))$, where the  Real groupoid $\xymatrix{\wPU(\hat{\cH})\dar[r] & \cdot}$ is given the compact-open topology (which is equivalent to the ${}^\ast$-strong operator topology) and the usual involution $Ad_{J_{0,\RR}}$.
 
A Rg D-D bundle $(\cA,\al)\in \fwRBr(\cG)$ being of type $0$ means that the fibres $\cA_x$ are isomorphic to the graded complex elementary $\cstar$-algebra $\wK_{ev}=\cK(\hat{\cH})$, and there is a Real local trivialisation $(U_j,\vp_j)_{j\in J}$ with commutative diagrams 
\begin{equation}~\label{eq:diagram-h_i-loc-triv-A}
\xymatrix{\cA_{|U_j} \ar[d]_{\tau_{U_j}}\ar[r]^{h_j} & U_j\times \wK_0  \ar[d]^{\tau\times bar} &  \\ \cA_{|U_{\bar{j}}} \ar[r]^{h_{\bar{j}}}& U_{\bar{j}}\times\wK_0 }
\end{equation}
and a Real family of continuous function $a_{ij}:U_{ij}\To \Aut^{(0)}(\wK_0)=\wPU(\hat{\cH})$, over every non-empty intersection $U_{ij}=U_i\cap U_j$, such that the homeomorphism $$h_i\circ h_j^{-1}:U_{ij}\times\wK_0\To U_{ij}\times\wK_0$$ sends $(x,T)$ to $(x,a_{ij}(x)T)$. Observe that $(a_{ij})\in ZR^1(\cU,\wPU(\hat{\cH}))$. We then obtain the "Real" analogue of the well known \emph{Dixmier-Douady class} (see~\cite{Dixmier-Douady:Champs},~\cite{Raeburn-Williams:Morita_equivalence}). Furthermore, we get a Real $\wPU(\hat{\cH})$-valued \v{C}ech $1$-cocyle $\mu$ over $\cG$ as follows. From the Real open cover $\cU=(U_j)_{j\in J}$ of $X$, form the Real open cover $\cU_1=(U^1_{(j_0,j_1)})$ of $\cG$ by setting $U^1_{(j_0,j_1)}=\{g\in \cG\mid r(g)\in U_{j_0}, s(g)\in U_{j_1}\}$ (~\cite[\S2]{Moutuou:Real.Cohomology}). Using the isomorphism of Rg $\cstar$-bundles $s^\ast\cA \To r^\ast\cA$ over $\cG$ induced by the Real $\cG$-action $\al$ and the commutative diagram~\eqref{eq:diagram-h_i-loc-triv-A}, there is a Real family of continuous functions $\mu_{(j_0,j_1)}:U^1_{(j_0,j_1)}\To \Aut^{(0)}(\wK_0)=\wPU(\hat{\cH})$ such that 
\begin{equation}~\label{eq:def-of-mu-A}
  \mu_{(j_0,j_1)}(g)=h_{j_0|_{r(g)}}\circ \al_g \circ h_{j_1|_{s(g)}}^{-1}, \forall g\in U^1_{(j_0,j_1)},	
  \end{equation}  
where $h_{j_0|_{r(g)}}:\cA_{r(g)}\To \{r(g)\}\times \wK_0$, and $h_{j_1|_{s(g)}}:\{s(g)\}\times \wK_0 \To \cA_{s(g)}$ are the restrictions of the isomorphisms $h_{j_0}:r^\ast\cA_{|U^1_{(j_0,j_1)}}\To U_{j_0}\times \wK_0$ and $h^{-1}_{j_1}:U_{j_1}\times \wK_0\To s^\ast \cA_{|U^1_{(j_0,j_1)}}$. It is easy to verify that $\mu^\cA=(\mu_{(j_0,j_1)})$ is a Real $1$-coboundary. We are going to show that the generalised Real homomorphism corresponding to the class of $\mu^\cA$ in $\check{H}R^1(\cG_\bullet,\wPU(\hat{\cH}))$ is actually a classifying generalised morphism for $(\cA,\al)$.

We first need some further constructions. For $x\in X$, let $P_x:=\Isom^{(0)}(\wK_0,\cA_x)$. Put 
\begin{equation}~\label{eq:def-of-P}
P:= \coprod_{x\in X} P_x.
\end{equation}
For $g\in \cG$ and $p=(s(g),\vp)\in P_{s(g)}$, the $\cG$-action $\al$ of $\cA$ provides the element $g\cdot p\in P_{r(g)}$ given by 
\begin{equation}~\label{eq:G-action_on_P}
	g\cdot p:= (r(g),\al_g\circ \vp).
\end{equation}

We wish to define a topology on $P$ such that not only the canonical projection $P\ni (x,\vp)\mto x\in X$ is continuous but also the formula~\eqref{eq:G-action_on_P} defines a continuous action of $\cG$ on $P$ with respect to the projection just given. To do so, we first consider the pull-backs $s^\ast P\To \cG$ and $r^\ast P\To \cG$ of $P\To X$ along the range and source maps. Then we look at the fibered-product $s^\ast P\times_\cG r^\ast P\To \cG$.

\begin{lem}~\label{lem:embedding-sPxrP}
 The $\cG$-action $\al$ of $\cA$ induces a (set-theoretical) embedding 
 \[s^\ast P\times_\cG r^\ast P\hookrightarrow \cG\times\wPU(\hat{\cH}).\] 	
\end{lem} 

\begin{proof}
If $((g,\vp),(g,\psi))\in s^\ast P\times_\cG r^\ast P$, then $\psi^{-1}\circ \al_g\circ \vp\in \Aut^{(0)}(\wK_0)=\wPU(\hat{\cH})$. It is straightforward to see that the correspondence 
\begin{eqnarray*}
s^\ast P\times_\cG r^\ast P & \To & \cG \times\wPU(\hat{\cH}) \\
((g,\vp),(g,\psi)) & \mto & (g,\psi^{-1} \circ \al_g\circ \vp)	
\end{eqnarray*}
is a well-defined injection map.	
\end{proof}

\begin{df}
Let $(\cA,\al)\in \fwRBr(\cG)$ of type $0$, and let $P$ be given by~\eqref{eq:def-of-P}. Let the space $s^\ast P\times_\cG r^\ast P$ be given the topology induced from the product topology of $\cG\times \wPU(\hat{\cH})$ via the embedding of Lemma~\ref{lem:embedding-sPxrP}. Then we endow $P$ with the topology induced from the embedding 
\begin{eqnarray*}
P & \hookrightarrow & s^\ast P\times_\cG r^\ast P \\
(x,\vp) & \mto & ((x,\vp),(x,\vp))
\end{eqnarray*}
In this way, $P$ is looked at as a subspace of $\cG\times \wPU(\hat{\cH})$.
\end{df}

From this definition, it is obvious that the projection $P\To X$ is an open continuous map with respect to which the formula~\eqref{eq:G-action_on_P} defines a continuous $\cG$-action on $P$. Moreover, $P$ is a Real $\cG$-space with respect to the involution $P\ni (x,\vp)\mto (\bar{x},\bar{\vp})$, where for $\vp\in \Isom^{(0)}(\wK_0,\cA_x)$, the isomorphism $\bar{\vp}$ is defined by $\bar{\vp}(T):=\overline{\vp(\bar{T})}$ for all $T\in \wK_0$. 

\begin{pro}
Let $u\in \wU(\hat{\cH})$ and $[u]$ its class in the group $\wPU(\hat{\cH})$. For $\vp\in P_x$ we put $\vp\cdot [u]:= \vp\circ Ad_u\in P_x$. Then the map $P\ni (x,\vp)\mto (x,\vp\cdot[u])\in P$ defines a principal Real $\wPU(\hat{\cH})$-action on $P$ compatible with the $\cG$-action with respect to the projection $P\To X$. In other words, we have a generalised Real homomorphism 
\[\xymatrix{ P\ar[r] \ar[d] & \cdot \\ X & }\] from $\cG$ to $\wPU(\hat{\cH})$.	
\end{pro}

\begin{proof}
The continuity of the map $\wPU(\hat{\cH})\times P\ni ([u],(x,\vp))\mto (x,\vp\cdot [u])\in P$ is a direct consequence of the construction of the topology of $P$. It respects the Real structures since for all $T\in \wK_0$,
\[\overline{\vp\cdot[u]}(T)=\overline{\vp(u\bar{T}u^{-1})}=\overline{\vp\circ Ad_u}(\bar{R})=\bar{\vp}\circ Ad_{\bar{u}}(T),=(\vp\cdot [\bar{u}])(T).\]
It remains only to check that the $\wPU(\hat{\cH})$-action is principal (cf.~\cite[\S1.2]{Moutuou:Real.Cohomology},~\cite{Tu-Xu-Laurent-Gengoux:Twisted_K}), which is straightforward since the map $P\times \wPU(\hat{\cH}) \To P\times_X P$ given by $$((x,\vp),[u])\mto ((x,\vp),(x,\vp\cdot [u])),$$  has inverse given by $((x,\vp),(x,\psi))\mto ((x,\vp),\vp^{-1}\circ \psi)$.  	
\end{proof}

\begin{pro}~\label{pro:class-mu^A}
The class of $[P]\in \Hom_{\RG}(\cG,\wPU(\hat{\cH}))$ in $\check{H}R^1(\cG_\bullet,\wPU(\hat{\cH}))$ is $\mu^\cA$, the latter being given by~\eqref{eq:def-of-mu-A}.	
\end{pro}

\begin{proof}
The Real local trivialisation $(U_j,h_j)$ of~\eqref{eq:diagram-h_i-loc-triv-A} gives rise to a Real family of local sections $\mathsf{s}_j:U_j\To P$ such that $\mathsf{s}_j(x)=h_{j|_x}^{-1}\in \Isom^{(0)}(\wK_0,\cA_x)$. For $g\in U^1_{(j_0,j_1)}$, we have $g\cdot h^{-1}_{j_1|_{s(g)}}=\al_g\circ h^{-1}_{j_1|_{s(g)}}$, hence $g\cdot \mathsf{s}_{j_1}(s(g))=\mathsf{s}_{j_0}(r(g))\cdot \mu_{(j_0,j_1)}$, which proves the result.	
\end{proof}

\begin{pro}
Every Rg D-D bundle $(\cA,\al)$ of type $0$ over $\grpd$ admits a generalised classifying morphism $\mathbb{P}(\cA)$. Furthermore, the assignment $$[P]\mto \mathbb{A}([P]):=P\times_{\wPU(\hat{\cH})}\wK_0$$ induces a well defined surjective homomorphism of abelian monoids 
\begin{eqnarray}~\label{eq:map-A-classifying}
\mathbb{A}: \Hom_{\RG}(\cG_\bullet,\wPU(\hat{\cH})) \To & \wRBr_0(\cG).
\end{eqnarray}
\end{pro}

\begin{proof}
Let $P:\cG\To \wPU(\hat{\cH})$ be the generalised Real homomorphism defined by~\eqref{eq:def-of-P}. Then the family of fibre-wise maps $$P\times_{\wPU(\hat{\cH})}\wK_0\ni ((x,\vp),T)\mto \vp(T) \in \cA_x$$ is an isomorphism of Rg D-D bundles over $\cG$. Therefore, $P:\cG\To \wPU(\hat{\cH})$ is a generalised classifying morphism for $(\cA,\al)$. The uniqueness of $P$ is guaranteed by Corollary~\ref{cor:uniqueness-of-P}.

The map $\mathbb{A}$ is well defined since an isomorphism of generalised Real homomorphisms $P\cong P'$ obviously induces an isomorphism between the associated Rg D-D bundles. It is a homomorphism of abelian monoids, for if $[P],[P']\in \Hom_{\RG}(\cG,\wPU(\hat{\cH}))$ then, thanks to Proposition~\ref{pro:sum-of-classifying} and the uniqueness of the generalised classifying morphism,  $P\otimes P'$ is a generalised classifying morphism for $\mathbb{A}([P])\hat{\otimes}_X\mathbb{A}([P'])$ and for $\mathbb{A}([P]+[P'])$ at the same time; so that $\mathbb{A}([P\otimes P'])\cong \wK_0^{P\otimes P'}\cong \mathbb{A}([P])\hat{\otimes}_X\mathbb{A}([P'])$, which implies $\mathbb{A}([P] +[P'])=\mathbb{A}([P])+\mathbb{A}([P'])$ in $\wRBr_0(\cG)$. The surjectivity of $\mathbb{A}$ is a consequence of the existence of the generalised classifying morphism we just proved. 	
\end{proof}

\begin{rem}
Let $X$ be a locally compact Hausdorff space. Recall that Atiyah and Segal defined in~\cite[pp.11-12]{Atiyah-Segal:Twisted_K} the monoid $\operatorname{Proj}^\pm(X)$ to be the set of infinite dimensional projective graded complex Hilbert space bundles on $X$ subjected to the operation of graded tensor products, and showed that as a set, $\operatorname{Proj}^\pm(X)\cong \check{H}^1(X,\ZZ_2)\times \check{H}^2(X,\ZZ)$.  
If  $X$ is endowed with a Real structure $\tau$, then $\Hom_{\RG}(X,\wPU(\hat{\cH}))$ is nothing but the Real analogue of $\operatorname{Proj}^\pm(X)$. We thus may expect to have a similar result as in the complex case; this will be discussed in the next sections.	
\end{rem}


\section{Intermediate isomorphism theorem}

Consider once again the abelian monoids $\Hom_{\RG}(\cG,\wU^0(\hat{\cH}))$ and $\Hom_{\RG}(\cG,\wPU(\hat{\cH}))$. There is a canonical monomorphism 
\[pr:\Hom_{\RG}(\cG,\U^0(\hat{\cH}))\To \Hom_{\RG}(\cG,\wPU(\hat{\cH}))\] 
induced by the canonical Real projection $\wU^0(\hat{\cH})\To \wPU(\hat{\cH})$; \emph{i.e.}, if $\mathbf{U}:\cG\To \U^0(\hat{\cH})$ is a generalised Real homomorphism, then we obtain a generalised Real homomorphism 
\[pr\circ \mathbf{U}:=\mathbf{U}\times_{\U^0(\hat{\cH})}\wPU(\hat{\cH}):\cG\To \wPU(\hat{\cH}).\]

\begin{df}~\label{df:stable-P}
An element $[P]\in \Hom_{\RG}(\cG,\wPU(\hat{\cH}))$ is called \emph{trivial} if $[P]=[pr\circ \mathbf{U}]$ for some $\mathbf{U}:\cG\To \U^0(\hat{\cH})$. 

Define an equivalence relation in $\Hom_{\RG}(\cG,\wPU(\hat{\cH}))$ by saying that $P_1,P_2:\cG\To \wPU(\hat{\cH})$ are \emph{stably isomorphic} if there exists a trivial generalised Real homomorphism $Q$ such that $$[P_1]+ [Q]= [P_2]+[Q].$$ In that case we write $[P_1]
\sim_{st} \ [ P_2]$. We define $$\Hom_{\RG}(\cG,\wPU(\hat{\cH}))_{st}:= \Hom_{\RG}(\cG,\wPU(\hat{\cH}))/_{\sim_{st}}.$$
The class of $[P]$ with respect to "$\sim_{st}$" is denoted by $[P]_{st}$
\end{df}

\begin{lem}~\label{lem:condition-for-P-triv}
$[P]$ is trivial if and only if $P$ is the generalised classifying morphism of a Rg D-D bundle of the form $(\cK(\hat{\sH}),Ad_U)$ where $(\hat{\sH},u)$ is a Rg Hilbert $\cG$-bundle.  	
\end{lem}

\begin{proof}
Assume $P\cong pr\circ \mathbf{U}$ is trivial. Let $[\om]\in \check{H}R^1(X,\wU^0(\hat{\cH}))$ be the class of the Real $\U^0(\hat{\cH})$-principal bundle $\mathbf{U}\To X$, and let $[c]\in \check{H}R^1(\cG_\bullet,\U^0(\hat{\cH}))$ be the class of $\mathbf{U}$ as Real $\U^0(\hat{\cH})$-principal $\cG$-bundle. Suppose, without loss of generality, that $\cU=(U_j)_{j\in J}$ is a Real open cover of $X$ on which $\om$ is represented, and such that $c$ is represented on the Real open cover $\cU_1=(U^1_{(j_0,j_1)})_{j_0,j_1\in J}$. Then we get a Rg Hilbert $\cG$-bundle $(\hat{\sH},u)$ by setting: 
\begin{equation}
	\hat{\sH}:= \coprod_{j\in J} U_j\times \hat{\cH}_{/{\sim}}	
	\end{equation}
where $U_i\times\hat{\cH}\ni (x,\xi) \sim (x,\om_{ij}(x)\xi)\in U_i\times \hat{\cH}$; if $[x,\xi]_j$ denotes the class in $\hat{\sH}$ of $(x,\xi)\in U_j\times \hat{\cH}$, we define the Real structure $[x,\xi]_j\mto [\bar{x},\bar{\xi}]_{\bar{j}}$ (where as usual the "bar" in $\hat{\cH}$ is the Real structure $J_{0,\RR}$), and the projection $\pi:\hat{\sH}\To X$ by $\pi([x,\xi]_j)=x$; the Real $\cG$-action $U$ is 
\begin{equation}
U_g([s(g),\xi]_{j_1}):=[r(g),c_{(j_0,j_1)}(g)\xi]_{j_0}.
\end{equation}
By construction, we see that $pr\circ \mathbf{U}=\mathbf{U}\times_{\U^0(\hat{\cH})}\wK_0$ is a generalised classifying morphism for $(\cK(\hat{\sH}),Ad_U)$ and that the class $\mu^{\cK(\hat{\sH})}$ (recall~\eqref{eq:def-of-mu-A}) in $\check{H}R^1(\cG,\wPU(\hat{\cH}))$ is $Ad_c$, where $(Ad_c)_{(j_0,j_1)}:=Ad_{c_{(j_0,j_1)}}$. 

Conversely, a Rg Hilbert $\cG$-bundle $(\hat{\sH},U)$ gives rise to a class $[\om]\in\check{H}R^1(\cG_\bullet,\U^0(\hat{\cH}))$, hence to a generalised Real homomorphism $\mathbf{U}:\cG\To \U^0(\hat{\cH})$. It follows from Proposition~\ref{pro:class-mu^A} that $pr\circ \mathbf{U}:\cG\To \wPU(\hat{\cH})$ is a generalised classifying morphism for $(\cK(\hat{\sH}),Ad_U)$; therefore $P\cong pr\circ \mathbf{U}$ by Corollary~\ref{cor:uniqueness-of-P}.		
\end{proof}

\begin{lem}
$\Hom_{\RG}(\cG,\wPU(\hat{\cH}))_{st}$ is an abelian group with respect to the sum; the inverse of $[P]_{st}$ is $[P^\ast]_{st}$ where $P^\ast$ is the generalised classifying morphism for the conjugate bundle of $\mathbb{A}([P])$.
\end{lem}

\begin{proof}
We only need to verify the existence of the inverse. $P\otimes P^\ast$ is a generalised classifying morphism for the Rg D-D bundle $\mathbb{A}([P])\hat{\otimes}_X\overline{\mathbb{A}([P])}$. From Corollary~\ref{cor:characterization_Morita_BrR}, $P\otimes P^\ast$ is then a generalised classifying morphism for $(\cK(\hat{\sH}),Ad_U)$ where $(\hat{\sH},U)$ is a Rg Hilbert $\cG$-bundle. Therefore, $[P\otimes P^\ast]$ is trivial, by Lemma~\ref{lem:condition-for-P-triv}.  	
\end{proof}

The main result of this section is the following

\begin{thm}~\label{thm:first-isom-BrR}
Let $\grpd$ be a Real groupoid. Then $\wRBr_0(\cG) \cong \Hom_{\RG}(\cG,\wPU(\hat{\cH}))_{st}$.	
\end{thm}

The proof is based on the following lemma (compare with~\cite[p.3]{Tu:Twisted_Poincare}).

\begin{lem}
The sequence of abelian monoids
\begin{equation}~\label{eq:exact-sequence-BrR}
	0 \To \Hom_{\RG}(\cG,\U^0(\hat{\cH})) \stackrel{pr}{\To} \Hom_{\RG}(\cG,\wPU(\hat{\cH})) \stackrel{\mathbb{A}}{\To} \wRBr_0(\cG) \To 0	
	\end{equation}
is exact.		
\end{lem}

\begin{proof}
We have already seen that $pr$ was a monomorphism of abelian monoids, and $\mathbb{A}$ was an epimorphism of abelian monoids. It then remains to show that $\ker \mathbb{A}=\range(pr)$.

$ \range(pr) \subset \ker \mathbb{A}$: indeed, from Lemma~\ref{lem:condition-for-P-triv} and Corollary~\ref{cor:uniqueness-of-P}, for all $[\mathbf{U}]\in\Hom_{\RG}(\cG,\U^0(\hat{\cH}))$, the Rg D-D bundle $\mathbb{A}([pr\circ \mathbf{U}])$ is of the form $(\cK(\hat{\sH}),Ad_U)$, hence $\mathbb{A}([pr\circ \mathbf{U}])=0$ in $\wRBr_0(\cG)$ by Corollary~\ref{cor:characterization_Morita_BrR}.  

$\ker \mathbb{A}\subset \range(pr)$: if $\mathbb{A}([P])=0$ then $P$ is the generalised classifying morphism for some $(\cK(\hat{\sH}),Ad_U)$. So, by Lemma~\ref{lem:condition-for-P-triv}, $[P]$ is trivial; in other words, $[P]=[pr\circ \mathbf{U}]\in \range(pr)$.
\end{proof}

\begin{proof}[Proof of Theorem~\ref{thm:first-isom-BrR}]
First of all, observe that there is a canonical isomorphism of abelian monoids $$\frac{\Hom_{\RG}(\cG,\wPU(\hat{\cH}))}{\range(pr)} \cong\Hom_{\RG}(\cG,\wPU(\hat{\cH}))_{st},$$   the quotient monoid is then an abelian group. Moreover, from the exact sequence~\eqref{eq:exact-sequence-BrR} we deduce an isomorphism of abelian monoids $$\mathscr{P}:\wRBr_0(\cG)\stackrel{\cong}{\To} \Hom_{\RG}(\cG,\wPU(\hat{\cH}))_{st},$$
such that $\mathscr{P}(\cA)$ is the class in $\Hom_{\RG}(\cG,\wPU(\hat{\cH}))_{st}$ of the generalised classifying morphism $P$ of $(\cA,\al)$. Furthermore, by definition of the inverse in $\Hom_{\RG}(\cG,\wPU(\hat{\cH}))_{st}$, we see that this isomorphism respects the inversion; it therefore is an isomorphism of abelian groups, which completes the proof.	
\end{proof}


\section{Example: case of a Real group}

Here we apply the observations of the previous sections to compute the Rg Brauer group of a locally compact Real group $\xymatrix{G\dar[r] & \cdot}$.

If $\cG=\xymatrix{G\dar[r]&\cdot}$ is a Real group, then for any Real group $S$, $\Hom_{\RG}(G,S)$ identifies with the set $\Hom(G,S)_R$ of continuous Real group homomorphisms from $G$ to $S$. In particular 
\[\Hom_{\RG}(G,\wPU(\hat{\cH}))_{st}\cong \frac{\Hom(G,\wPU(\hat{\cH}))_R}{\Hom(G,\U^0(\hat{\cH}))_R}.\]

For instance, if $G$ is equipped with the trivial Real structure, then 
\[\Hom_{\RG}(G,\wPU(\hat{\cH}))_{st}\cong \frac{\Hom(G,\wPU(\hat{\cH}_\RR))}{\Hom(G,\U^0(\hat{\cH}_\RR))}.\]

Moreover, a Rg D-D bundle over $\xymatrix{G\dar[r]&\cdot}$ is obviously of constant type since it is given by a Real bundle over the point together with a Real action of $G$; so $\wRBr(G)=\wRBr_\ast(G)$. It is  convenient to write $\wRBr_G(\ast)$ instead of $\wRBr(G)$ since it is exactly the Rg Brauer group of the point with the trivial Real $G$-action. Similarly, we write $\wBRO_G(\ast)$ and $\wBr_G(\ast)$.

Now, applying Proposition~\ref{pro:simplification-BrR} and Theorem~\ref{thm:first-isom-BrR} to $G$, we get 

\begin{pro}
 Let $G$ be a locally compact Real group. Then
 \begin{equation*}
 \wRBr_G(\ast) \cong \ZZ_8\oplus \frac{\Hom(G,\wPU(\hat{\cH}))_R}{\Hom(G,\U^0(\hat{\cH}))_R}, \quad 
 \wBRO_G(\ast)  \cong  \ZZ_8\oplus \frac{\Hom(G,\wPU(\hat{\cH}_\RR))}{\Hom(G,\U^0(\hat{\cH}_\RR))}.
 \end{equation*}	
\end{pro}


\section{The main isomorphisms}

The purpose of this section is to prove the following theorem.

\begin{thm}~\label{thm:main-RBr}
Let $\grpd$ be a Real groupoid. Then
\begin{eqnarray*}
 		\wRBr(\cG) &\cong &\check{H}R^0(\cG_\bullet,\Inv\wfK)\oplus \left(\check{H}R^1(\cG_\bullet,\ZZ_2)\ltimes \check{H}R^2(\cG_\bullet,\uc)\right); \\
 		\wRBr_\ast(\cG) &\cong & \check{H}R^0(\cG_\bullet,\ZZ_8)\oplus \left(\check{H}R^1(\cG_\bullet,\ZZ_2)\ltimes \check{H}R^2(\cG_\bullet,\uc)\right). 		
 	\end{eqnarray*} 	
\end{thm}

We first deduce a generalization of Donovan-Karoubi's isomorphism~\eqref{eq:isom_GBrO-cohom}:

\begin{cor}
Let $\grpd$ be a groupoid. Then
\[
\wBRO(\cG) \cong  \check{H}^0(\cG_\bullet,\ZZ_8)\oplus \left(\check{H}^1(\cG_\bullet,\ZZ_2)\ltimes \check{H}^2(\cG_\bullet,\ZZ_2)\right).
\]	
\end{cor}

\begin{proof}
This follows from Theorem~\ref{thm:Br_vs_BrR_vs_BrO} 2), Theorem~\ref{thm:main-RBr}, and the fact that when $\cG$ is given the trivial involution, $\check{H}R^n(\cG_\bullet,\uc)\cong \check{H}^n(\cG_\bullet,\ZZ_2)$, thanks to~\cite[Proposition 2.43]{Moutuou:Real.Cohomology}.
\end{proof}

The proof of Theorem~\ref{thm:main-RBr} is divided into several steps that mainly consist of constructing an isomorphism $\Hom_{\RG}(\cG,\wPU(\hat{\cH}))_{st}\cong \wRExt(\cG,\uc)$. 

Let us consider the following "generic" Rg $\uc$-central extension $\EE_{\wK_0}$ of the Real groupoid $\xymatrix{\wPU(\hat{\cH})\dar[r]& \cdot}$
\begin{eqnarray}
	\xymatrix{\uc \ar[r] & \wU(\hat{\cH}) \ar[r]^{pr} & \wPU(\hat{\cH}) \ar[d]^\partial \\ & & \ZZ_2}
\end{eqnarray}
where $\partial([u])$ is the degree of the homogeneous unitary $u$.

Let $\grpd$ be a Real groupoid and let $P:\cG\To \wPU(\hat{\cH})$ be a generalised Real homomorphism. Then we get a Rg $\uc$-central extension $P^\ast\EE_{\wK_0}$ of $\cG$ by pulling back $\EE_{\wK_0}$ via $P$ (see~\cite[\S1]{Moutuou:Real.Cohomology}).

\begin{lem}
The assignment $P\mto P^\ast\EE_{\wK_0}$ induces a well defined homomorphism of Abelian monoids
\[
\begin{array}{lll}
	\mathscr{T}: & \Hom_{\RG}(\cG,\wPU(\hat{\cH})) & \To  \wRExt(\cG,\uc) \\ &  \quad \quad \quad \quad [P] & \mto [P^\ast \EE_{\wK_0}]	
	\end{array}
	\]	
\end{lem}

\begin{proof}
Assume $P\cong P'$ are isomorphic generalised Real homomorphisms from $\cG$ to $\wPU(\hat{\cH})$. As usual, we may assume that $P$ and $P'$ are represented in $\Hom_{\RG_\Om}(\cG,\wPU(\hat{\cH}))$ (see~\cite[\S1]{Moutuou:Real.Cohomology}) on the same Real open cover $\cU$ of $X$ by two Real strict homomorphisms $f:\cG[\cU]\To \wPU(\hat{\cH})$ and $f':\cG[\cU]\To \wPU(\hat{\cH})$, respectively. The pull-backs $P^\ast\EE_{\wK_0}$ and $(P')^\ast\EE_{\wK_0}$ are then Morita equivalent to 
\[\xymatrix{\uc \ar[r] & \wU(\hat{\cH})\times_{pr,\wPU(\hat{\cH}),f}\cG[\cU] \ar[r] & \cG[\cU] \ar[d]^{\partial\circ f} \\ & & \ZZ_2 } 
\]
and 
\[\xymatrix{\uc \ar[r] & \wU(\hat{\cH})\times_{pr,\wPU(\hat{\cH}),f'}\cG[\cU]\ar[r] & \cG[\cU] \ar[d]^{\partial\circ f'} \\ & & \ZZ_2} 
\] respectively, where in both cases, the projection is the canonical one onto the second factor. Since $P\cong P'$, there is an isomorphism of Real groupoids $\phi: \cG[\cU]\To \cG[\cU]$ such that $f'=f\circ \phi$. Therefore, the map 
\[
\wU(\hat{\cH})\times_{pr,\wPU(\hat{\cH}),f}\cG[\cU]  \To  \wU(\hat{\cH})\times_{pr,\wPU(\hat{\cH}),f'}\cG[\cU] \]
defined by $(u,\g) \mto (u,\phi(\g))$ induces an isomorphism of Rg $\uc$-central extension $P^\ast\EE_{\wK_0}\cong (P')^\ast\EE_{\wK_0}$. Hence $\mathscr{T}$ is well defined. 

Let us check that $\mathscr{T}$ is a homomorphism. Let $[P_1],[P_2]\in \Hom_{\RG}(\cG,\wPU(\hat{\cH}))$ and let $p_1,p_2:\cG[\cU]\To \wPU(\hat{\cH})$ be Real strict homomorphisms representing $[P_1]$ and $[P_2]$, respectively in the category $\RG_{\Om}$. Then the map
\[
\left(\wU(\hat{\cH})\times_{pr,\wPU(\hat{\cH}),p_1}\cG[\cU]\right)\times_{\cG[\cU]}\left(\wU(\hat{\cH})\times_{pr,\wPU(\hat{\cH}),p_2}\cG[\cU])\right)\to \wU(\hat{\cH}\hat{\otimes}\hat{\cH})\times_{pr,\wPU(\hat{\cH}\hat{\otimes}\hat{\cH}),p_1\hat{\otimes}p_2}\cG[\cU])
\]
 defined by $$\left((u_1,\g),(u_2,\g)\right)\mto (u_1\hat{\otimes}u_2,\g),$$ 
is easily checked to define an isomorphism of Rg $\uc$-central extensions $$(P_1^\ast\EE_{\wK_0})\hat{\otimes}(P_2^\ast\EE_{\wK_0})\cong (P_1\otimes P_2)^\ast\EE_{\wK_0}.$$ 
Thus, $\mathscr{T}([P_1]+[P_2])=\mathscr{T}([P_1])+\mathscr{T}([P_2])$, and we are done.
\end{proof}

\begin{lem}
If $[P]$ is trivial, then $\mathscr{T}([P])=0$ in $\wRExt(\cG,\uc)$. Therefore, $\mathscr{T}$ induces a homomorphism of abelian groups 
\[\Hom_{\RG}(\cG,\wPU(\hat{\cH}))_{st} \To \wRExt(\cG,\uc),\] also denoted by $\mathscr{T}$.	
\end{lem}

\begin{proof}
Since $\mathscr{T}([P])$ depends only on the isomorphism class of $P$, it suffices to suppose $P=pr\circ \mathbf{U}$ for some generalised Real homomorphism $\mathbf{U}:\cG\To \U^0(\hat{\cH})$. Let $u:\cG[\cU]\to \U^0(\hat{\cH})$ be a Real strict homomorphism representing $\mathbf{U}$ in the category $\RG_\Om$. Then the Real groupoid morphism $p:\cG[\cU]\To \wPU(\hat{\cH})$ given by $ p(\g)= [u_\g]$ represents $P$. It follows that the map 
\[
\begin{array}{lll}
\cG[\cU] & \To & \wU(\hat{\cH})\times_{pr,\wPU(\hat{\cH}),p}\cG[\cU] \\
\g & \mto & \quad \quad \quad \quad (u_\g,\g)
\end{array}
\]
is a well defined section of the projection of $(p^\ast \wPU(\hat{\cH}),\cG[\cU],\partial\circ p)$; the latter is then a strictly trivial Rg $\uc$-twist~\cite{Moutuou:Real.Cohomology}. Therefore, $P^\ast\EE_{\wK_0}$ is a trivial Rg $\uc$-central extension of $\cG$.
\end{proof}

At this point, we are following closely~\cite[\S2.6]{Tu-Xu-Laurent-Gengoux:Twisted_K} to construct a homomorphism $\mathscr{P}'$ in the other direction; and then we will show that $\mathscr{T}$ and $\mathscr{P}'$ are inverses of each other.

Let $\EE=(\wGa,\Ga,\del,Z)$ be a Rg $\uc$-central extension of $\grpd$.

\begin{df}(~\cite[Definition 2.37]{Tu-Xu-Laurent-Gengoux:Twisted_K}).
A function $\xi\in \cC_c(\wGa)$ is said $\uc$-equivariant if $\xi(\lambda \tilde{\g})=\lambda^{-1}\xi(\tilde{\g})$ for any $\lambda\in \uc$ and any $\tilde{\g}\in \wGa$. 	
\end{df}

 Let $\mu=\{\mu^y\}_{y\in Y}$ be a Real Haar system of the Real groupoid $\xymatrix{\wGa\dar[r]& Y}$. For $y\in Y$, consider the graded Hilbert space $\cL_y^2:=L^2(\wGa^y)^{\uc}$ consisting of $\uc$-equivariant functions on $\wGa^y$ which are $L^2$ with respect to $\mu^y$. Note that the $\ZZ_2$-grading of $\cL_y^2$ is the one induced by $\del$; \emph{i.e.}, for $\xi\in \cC_c(\wGa^y)^{\uc}$, define $\del\xi$ by 
\begin{equation}~\label{eq:df-sH_Ga-grading}
	(\del\xi)(\tilde{\g}) := (-1)^{\del(\g)}\xi(\tilde{\g}),
\end{equation}
where $\g\in \Ga$ is such that $\pi(\tilde{\g})=\g$. Let 
\begin{equation}~\label{eq:df-sH_Ga}
	\sH_{\wGa,y} = \cL_y^2\otimes \cH, \quad {\rm and\ }\tilde{\sH}_{\wGa,y}\mydot=\coprod_y\sH_{\wGa,y},
\end{equation}
where, as usual $\cH=l^2(\NN)$ is the generic separable infinite dimensional Hilbert space, endowed with the Real structure $J_\RR$ given by the complex conjugation with respect to the canonical basis. Then the countably generated continuous field of infinite dimensional graded Hilbert spaces $\tilde{\sH}_{\wGa}\To Y$, is a locally trivial graded Hilbert bundle, and hence trivial thanks to~\cite[Th\'eor\`eme 5]{Dixmier-Douady:Champs}. By identifying $\sH_{\wGa,y}$ with the space 
\[L^2(\wGa^y;\cH)^{\uc}:=\{\xi\in \cC_c(\wGa^y;\cH) \mid \xi(\lambda \tilde{\g})=\lambda^{-1}\xi(\tilde{\g}), \forall \lambda\in \uc,\tilde{\g}\in \wGa^y\},\]
we define the Real structure on $\tilde{\sH}_{\wGa}$ by 
\begin{equation}~\label{eq:df-sH_Ga-involution}
	\sH_{\wGa,y} \ni \xi\mto \bar{\xi}\in \sH_{\wGa,\bar{y}},
\end{equation}
 where $\bar{\xi}(\tilde{\g}):=\overline{\xi(\overline{\tilde{\g}})}$ for all $\tilde{\g}\in \wGa^{\bar{y}}$. Together with this involution, $\tilde{\sH}_{\wGa}$ is clearly a Rg Hilbert bundle over $Y$. 
 
 For $y\in Y$, let $\hat{\sU}_{\wGa,y}=\U^0(\hat{\cH},\sH_{\wGa,y})\bigcup \U^1(\hat{\cH},\sH_{\wGa,y})$ be the space of homogeneous unitary operators from $\hat{\cH}$ to $\sH_{\wGa,y}$. Put 
 $$\hat{\sU}_{\wGa}=\coprod_{y\in Y}\sU_{\wGa,y}.$$
  The field $\hat{\sU}_{\wGa}$ is endowed with the topology induced from $\tilde{\sH}$: a section $Y\ni y \mto u_y\in \hat{\sU}_{\wGa,y}$ is continuous if and only if for every $\xi\in \hat{\cH}$, the map $y\mto u_y\xi$ is a continuous section of $\tilde{\sH}_{\wGa}\To Y$. The bundle $\hat{\sU}_{\wGa}\To Y$ is, in an obvious way, a Real $\wU(\hat{\cH})$-principal bundle, with the Real structure $\hat{\sU}_{\wGa,y}\ni u\mto \bar{u}\in \hat{\sU}_{\wGa,\bar{y}}$, where for $\xi\in\hat{\cH}$, $\bar{u}(\xi):=\overline{u(\bar{\xi})}\in \sH_{\wGa,\bar{y}}$. Notice that scalar multiplication with elements of the fibres $\hat{\sU}_{\wGa,y}$ induces a Real $\uc$-action on $\hat{\sU}_{\wGa}$. It follows that its quotient 
\begin{equation}~\label{eq:df-PU_wGa}
 \PP\hat{\sU}_{\wGa}:=\hat{\sU}_{\wGa}/\uc	
 \end{equation} 
is a Real $\wPU(\hat{\cH})$-principal bundle over $Y$. We write $[(x,u)]$ for the class of $(x,u)\in \hat{\sU}_{\wGa,y}$ in the quotient $\PP\hat{\sU}_{\wGa}$.

One defines a Real left $\Ga$-action on $\PP\hat{\sU}_{\wGa}\To Y$ in the following way: let $\g\in \Ga$ and $[(s(\g),u)]\in P\hat{\sU}_{\wGa}$, then $\g\cdot [(s(g),u)]$ is the class $[(r(\g),\g\cdot u)]$ of the element $(r(\g),\g\cdot u)\in \hat{\sU}_{\wGa,r(\g)}$, where for each $\xi\in \hat{\cH}$, the function $(\g\cdot u)\xi\in L^2(\wGa^{r(\g)};\hat{\cH})^{\uc}$ is given by
\[(\g\cdot u)\xi: \wGa^{r(\g)}\ni h \mto (u\xi)(\tilde{\g}^{-1}h),\]
where $\tilde{\g}\in \wGa_\g$ is any lift of $\g$ with respect to the projection $\wGa\To \Ga$. It is easy to verify that with respect to this well defined Real action, $\PP\hat{\sU}_{\wGa}$ is a Real $\wPU(\hat{\cH})$-principal bundle over the Real groupoid $\gamgpd$, in other words, it is a generalised Real homomorphism from $\Ga$ to $\xymatrix{\wPU(\hat{\cH})\dar[r]& \cdot}$. 

Now the composite $\cG\stackrel{Z^{-1}}{\To} \Ga \stackrel{P\hat{\sU}_{\wGa}}{\To} \wPU(\hat{\cH})$ provides us with the generalised Real homomorphism 
\begin{equation}
P\EE:= \PP\hat{\sU}_{\wGa}\circ Z^{-1}:\cG\To \wPU(\hat{\cH}).
\end{equation}

\begin{rem}
Observe that the Real $\Ga$-action on $\PP\hat{\sU}_{\wGa}$ is induced by the Real $\wGa$-action on $\hat{\sU}_{\wGa}$ defined by $\tilde{\g}\cdot (s(\tilde{\g}),u):=(r(\tilde{\g}),\tilde{\g}\cdot u)$, where for $\xi\in \hat{\cH}$, $((\tilde{\g}\cdot u)\xi)(h):=(u\xi)(\tilde{\g}^{-1}h)$ for all $h\in \wGa^{r(\tilde{\g})}$. In fact, $\hat{\sU}_{\wGa}$ is a Real $\wU(\hat{\cH})$-principal bundle over the Real groupoid $\xymatrix{\wGa\dar[r]& Y}$.	
\end{rem}

\begin{lem}~\label{lem:lem1-def-pp}
Let $\EE_i=(\wGa_i,\Ga_i,\del_i,Z_i), i=1,2$ be Morita equivalent Rg $\uc$-central extensions of $\grpd$, with equivalence implemented by an $\uc$-equivarient Morita equivalence $Z:\wGa_1\To \wGa_2$. Denote by $Z'=Z/\uc:\Ga_1\To \Ga_2$ the induced Morita equivalence. Then $\PP\hat{\sU}_{\wGa_2}\circ Z'$ and $\PP\hat{\sU}_{\wGa_1}$ are isomorphic.
\end{lem}

\begin{proof}
This is a consequence of~\cite[Lemma 2.39]{Tu-Xu-Laurent-Gengoux:Twisted_K}.	
\end{proof}

\begin{lem}
Assume $\EE_1$ and $\EE_2$ are Morita equivalent Rg  $\uc$-central extensions of $\cG$. Then $[P\EE_1]= [P\EE_2]$. 
\end{lem}

\begin{proof}
A Morita equivalence $\Ga\sim_Z \cG$ induces an isomorphism of abelian monoids $$Z_\ast: \hom_{\RG}(\Ga,\wPU(\hat{\cH}))\To \Hom_{\RG}(\cG,\wPU(\hat{\cH}))$$ given by $Z_\ast [P] :=[P\circ Z^{-1}]$ (~\cite[Proposition 2.35]{Tu-Xu-Laurent-Gengoux:Twisted_K}). \\
Now if $Z$ is a Morita equivalence from $\EE_1$ and $\EE_2$, then under the notations of Lemma~\ref{lem:lem1-def-pp}, the commutative diagram of generalised Real homomorphisms 
\[\xymatrix{ \Ga_1\ar[rr]^{Z'} \ar[dr]_{Z_1} && \Ga_2 \ar[dl]^{Z_2} \\ & \cG & }\]
induces a commutative diagram of abelian monoids 

\[\xymatrix{ \Hom_{\RG}(\Ga_1,\wPU(\hat{\cH}))\ar[rr]^{Z'_\ast} \ar[dr]_{Z_{1\ast}} && \Hom_{\RG}(\Ga_2,\wPU(\hat{\cH})) \ar[dl]^{Z_{2\ast}} \\ & \Hom_{\RG}(\cG,\wPU(\hat{\cH})) & }
\]
Consequently, $\PP\hat{\sU}_{\wGa_2}\circ (Z')^{-1}\circ Z_2^{-1}\cong \PP\hat{\sU}_{\wGa_1}\circ Z_1^{-1}=P\EE_1$. But from Lemma~\ref{lem:lem1-def-pp}, $\PP\hat{\sU}_{\wGa_1}\circ (Z')^{-1}\cong \PP\hat{\sU}_{\wGa_2}$. Therefore, $P\EE_2\cong P\EE_1$.	
\end{proof}

\begin{lem}~\label{lem:e-rond-p=1}
The assignment $\EE\mto P\EE$ induces group homomorphism 
\[\mathscr{P}': \wRExt(\cG,\uc)\mto \Hom_{\RG}(\cG,\wPU(\hat{\cH}))_{st}, [\EE]\mto [P\EE]_{st}.\]
Moreover $\mathscr{T}\circ \mathscr{P}'=\Id$.
\end{lem}

\begin{proof}
The second statement follows from~\cite[Proposition 2.38]{Tu-Xu-Laurent-Gengoux:Twisted_K}.

For the first statement, we shall first check that $[P\EE]$ is trivial if $\EE$ is. But, thanks to the previous lemma, it suffices to show that $P\EE_0$ comes from a generalised Real homomorphism $\mathbf{U}:\cG\To \U^0(\hat{\cH})$, where $\EE_0$ is the trivial extension $(\cG\times\uc,\cG,0)$. This is obvious since $L^2(\cG^x\times\uc)^{\uc}\cong L^2(\cG^x)$, which implies in particular that there is a canonical Real graded $\cG$-action on the Real (trivially) graded Hilbert bundle $\sH_{\cG\times\uc}$.

Let $[\EE_1],[\EE_2]\in \wRExt(\cG,\uc)$. For the sake of simplicity, we shall assume $\EE_1$ and $\EE_2$ are represented by Rg $\uc$-twists $(\wGa_1,\Ga,\del_1)$ and $(\wGa_2,\Ga,\del_2)$ over the same Real groupoid $\gamgpd$ Morita equivalent to $\cG$. Let $\mu_i$ be a Real Haar system of $\wGa_i,i=1,2$. Then $\wGa_1\hat{\otimes}\wGa_2$ is equipped with the Real Haar system $\mu_1\times_{\uc}\mu_2$, and $\Ga$ is equipped with the image of the latter (here $\mu_1\times_{\uc}\mu_2$ is meant to say the product measure is invariant under the diagonal action by $\uc$). For $y\in Y$, we denote by  $\cC_c(\wGa^y_1;\cH)^{\uc}\hat{\otimes}_{\cC_c(\Ga^y)}\cC_c(\wGa_2;\cH)^{\uc}$ the completion, with respect to the inductive limit topology in $\cC_c(\wGa^y_1\hat{\otimes}\wGa_2^y)^{\uc}$, of the $\cC_c(\Ga^y)$-linear span of

\[ \left\{\xi_1\hat{\odot}\xi_2 \mid \xi_1\in \cC_c(\wGa^y_1)^{\uc},\xi_2\in \cC_c(\wGa_2^y)^{\uc}\right\},\]
where $\xi_1\hat{\odot}\xi_2\in \cC_c(\wGa_1^y\hat{\otimes}\wGa_2)^{\uc}$ is defined by $[(\tilde{\g}_1,\tilde{\g}_2)]\mto \xi_1(\tilde{\g}_1)\xi_2(\tilde{\g}_2)$, and where $\cC_c(\Ga^y)$ acts on $\cC_c(\wGa_i^y)^{\uc}$ by the formulas: $(\xi_1\cdot \phi)(\tilde{\g}_1)=\xi_1(\tilde{\g}_1)\phi(\pi_1(\tilde{\g}_1))$, and $(\phi\cdot \xi_2)(\tilde{\g}_2)=\phi(\pi_2(\tilde{\g}_2))\xi_2(\tilde{\g}_2)$ for $\xi_1\in \cC_c(\wGa_1^y)^{\uc},\phi\in \cC_c(\Ga^y),\xi_2\in \cC_c(\wGa^y_2)$ and $\tilde{\g}_i\in \wGa_i^y,i=1,2$. Then, passing to the $L^2$-norms, the graded Hilbert spaces $L^2(\wGa_1^y\hat{\otimes}\wGa_2^y)^{\uc}$ and $L^2(\wGa_1^y)^{\uc}\hat{\otimes}_{L^2(\Ga^y)}L^2(\wGa_2^y)^{\uc}$ are isomorphic.

Define a generalised Real homomorphism $\sU(\Ga):\Ga\To \U^0(\hat{\cH})$ by the Real field $$\sU(\Ga):=\coprod_{y\in Y}\U^0(\hat{\cH},L^2(\Ga^y)\otimes\cH),$$
where the Real $\Ga$-action is induced by the Real $\Ga$-action on the Rg Hilbert $\Ga$-bundle $\tilde{\sH}(\Ga)=\coprod_{y\in Y}L^2(\Ga^y)\otimes\cH\To Y$ defined by $\g\cdot(s(\g),\xi):= (r(\g),\g\cdot \xi)$, with $(\g\cdot\xi)(h):=\xi(\g^{-1}h)$ for all $h\in \Ga^{r(\g)}$ (note that the grading of $\tilde{\sH}(\Ga)$ is carried by $L^2(\Ga^y)$ and is given by $(\del_1\otimes\del_2)\xi(\g):=(-1)^{\del_1(\g)+\del_2(\g)}\xi(\g)$).
Then, remarking that $pr\circ \sU(\Ga)=\sU(\Ga)/\uc$, we define an isomorphism of Real $\wPU(\hat{\cH})$-principal bundles over $\gamgpd$
\begin{equation}~\label{eq:isom:P_Ga_1xP_Ga_2}
\begin{array}{lll}
\PP\hat{\sU}_{\wGa_1} \otimes \PP\hat{\sU}_{\wGa_2}\otimes (pr\circ \sU(\Ga)) & \stackrel{\cong}{\To} & \PP\hat{\sU}_{\wGa_1\hat{\otimes}\wGa_2} \\

[(y,u_1)]\otimes [(y,u_2)] \otimes [(y,v)] & \mto & [(y,v\cdot (u_1\hat{\odot}u_2))]
\end{array}
\end{equation}
as follows: given $\xi\in \hat{\cH}$, we write $u_i(\xi)=\sum_{j}\phi_i^j\otimes \eta_i^j$, $i=1,2$, where $\phi_i^j\in L^2(\wGa^y_i)^{\uc}, \eta_i^j\in \cH$, and similarly, $v(\xi)=\sum_j\psi^j\otimes \zeta^j \in L^2(\Ga^y)\otimes \cH$; then the unitary $v\cdot (u_1\hat{\odot}u_2)$ is defined by 
\begin{equation}~\label{eq:df:v.(u_1-otimes-u_2)}
\begin{array}{ll}
(v\cdot (u_1\hat{\odot}u_2))(\xi)& :=\sum_j\left(\psi^j\cdot \phi_1^j\hat{\odot}\phi_2^j\right)\otimes \eta_1^j\otimes \eta_2^j\otimes\zeta^j \\
 & \in L^2(\wGa_1^y\otimes \wGa^y_2)^{\uc}\otimes\cH\otimes\cH\otimes\cH\cong L^2(\wGa_1^y\otimes \wGa^y_2)^{\uc}\otimes\cH.
\end{array}
\end{equation} 
Thus $[\PP\hat{\sU}_{\wGa_1\hat{\otimes}\wGa_2}]\sim_{st} \ [\PP\hat{\sU}_{\wGa_1}\otimes \PP\hat{\sU}_{\wGa_2}]$ in $\Hom_{\RG}(\Ga,\wPU(\hat{\cH}))$. Therefore, by functoriality, we have $[P(\EE_1\hat{\otimes}\EE_1)]\sim_{st} \ [(P\EE_2)\otimes (P\EE_2)]$, which means that $\mathscr{P}'([\EE_1]+[\EE_2])=\mathscr{P}'([\EE_1])+\mathscr{P}'([\EE_2])$.	
\end{proof}

\begin{lem}~\label{lem:Hom_RG-vs-ExtR}
We have $\mathscr{P}'\circ \mathscr{T}=\Id$; consequently, we have a group isomorphism 
\[\Hom_{\RG}(\cG,\wPU(\hat{\cH}))_{st}\cong \wRExt(\cG,\uc).\]	
\end{lem}

\begin{proof}
In view of Lemma~\ref{lem:e-rond-p=1}, we only have to verify that $\mathscr{P}'\circ \mathscr{T}=\Id$. Let $[P]\in \Hom_{\RG}(\cG,\wPU(\hat{\cH}))$ be represented by a pair $[(\cU,p)]\in \Hom_{\RG_\Om}(\cG,\wPU(\hat{\cH}))$. Recall (cf.~\cite[Proposition 1.33]{Moutuou:Real.Cohomology}) that $P\cong P'\circ Z_{\iota_\cU}^{-1}$, where $\iota_\cU:\cG[\cU]\hookrightarrow \cG$ is the canonical Real inclusion and $P'$ is (isomorphic to) the generalised Real homomorphism induced from the strict Real homomorphism $p:\cG[\cU]\To \wPU(\hat{\cH})$. Notice that $P'=\coprod_jU_j \times \wPU(\hat{\cH})$ together with the Real $\cG[\cU]$-action $g_{j_0j_1}\cdot (s(g)_{j_1},[u]):=(r(g)_{j_0},[u])$. On the other hand, there is canonical Real $\cG[\cU]$-action on the Rg Hilbert bundle $\tilde{\sH}_{p^\ast \wU(\hat{\cH})}$, hence $P\hat{\sU}_{p^\ast\wU(\hat{\cH})}=\wPU(\hat{\cH},\tilde{\sH}_{p^\ast\wU(\hat{\cH})})\cong \coprod_jU_j\times\wPU(\hat{\cH})=P'$. It follows that $P\hat{\sU}_{p^\ast\wU(\hat{\cH})}\circ Z_{\iota_\cU}^{-1}\cong P'\circ Z_{\iota_\cU}^{-1}=P$, and hence $\mathscr{P}'(\mathscr{T}([P]))=[P]$.	
\end{proof}

\begin{proof}[Proof of Theorem~\ref{thm:main-RBr}]
By~\cite[Theorem 2.60]{Moutuou:Real.Cohomology}, there is an isomorphism 
\[dd: \wRExt(\cG,\uc)\stackrel{\cong}{\To} \check{H}R^1(\cG_\bullet,\ZZ_2)\ltimes \check{H}R^2(\cG_\bullet,\uc),\]
where the operation in the semi-direct product is $(\del,\om)+(\del',\om'):=(\del+\del',(\del\smile \del')\cdot \om\cdot\om')$.
Therefore, by Theorem~\ref{thm:first-isom-BrR} and Lemma~\ref{lem:Hom_RG-vs-ExtR}: we have an isomorphism 
\[ dd\circ \mathscr{T}\circ \mathscr{P}: \wRBr_0(\cG) \To \check{H}R^1(\cG_\bullet,\ZZ_2)\ltimes\check{H}^2(\cG_\bullet,\uc).\] 
The result then follows from~\ref{pro:simplification-BrR}.
\end{proof}


\section{Oriented Rg D-D bundles}

For $\cA \in \wRBr(\cG)$, its image in $\check{H}R^0(\cG_\bullet,\Inv\wfK)\oplus \left(\check{H}R^1(\cG_\bullet,\ZZ_2)\ltimes\check{H}R^2(\cG_\bullet,\uc)\right)$ is called \emph{the D-D class of $\cA$} and is denoted by $DD(\cA)$. 

\begin{df}
A Rg D-D bundle $(\cA,\al)$ is called \emph{oriented} if its D-D class is of the form $(0,0,\frc)$; $(\cA,\al)$ is then of type $0$. By $\wRBr^+(\cG)$ we denote the subset of $\wRBr(\cG)$ consisting of oriented Rg D-D bundles.
\end{df}

Given a Rg D-D bundle $(\cA,\al)$ of type $0$, the Rg $\uc$-central extension obtained by the composite 
\[\wRBr_0(\cG) \stackrel{\mathscr{P}}{\To} \Hom_{\RG}(\cG,\wPU(\hat{\cH}))\stackrel{\mathscr{T}}{\To} \wRExt(\cG,\uc)\]
is called \emph{the associated Rg extension} and is denoted by $\EE_\cA$.	\\

It should be noted that the associated Rg extension $\EE_\cA$ of an oriented Rg D-D bundle is \emph{even} in the sense that the grading $\del$ of $\EE_\cA$ is the zero function. Of course Morita equivalence and tensor product of Rg D-D bundles preserve orientation. Thus $\wRBr^+(\cG)$ is a subgroup of $\wRBr_0(\cG)$. Indeed, using similar arguments as in~\cite[\S3.4]{Karoubi:Twisted}, we obtain the "Real analogue" of Kumjian-Muhly-Renault-Williams~\cite{Kumjian-Muhly-Renault-Williams:Brauer}.

\begin{thm}
Let $\grpd$ be a Real groupoid. Then
 \[ \wRBr^+(\cG)\cong \Hom_{\RG}(\cG,\wPU^0(\hat{\cH}))_{st}\cong \check{H}R^2(\cG_\bullet,\uc).\]	
\end{thm}

\begin{rem}
We shall note that the above result generalises J. Rosenberg classification of \emph{real continuous-trace $\cstar$-algebras} given in~\cite{Rosenberg:Continuous-trace}. Indeed, let $(X,\tau)$ be a compact Real space. Then $\wRBr^+(X)\cong \check{H}R^2(X,\uc)\cong \check{H}R^3(X,\ZZ^{0,1})$. Thus, if $\cA\in \wRBr^+(X)$ with Dixmier-Douady class $DD(\cA)=\al \in \check{H}R^3(X,\ZZ^{0,1})$, we have $\tau^\ast \al =-\al$, which coincides with~\cite[Proposition 3.1]{Rosenberg:Continuous-trace}. 	
\end{rem}



\appendix


\section{Rg $\cstar$-algebras}

Recall that \emph{a complexification} of a real Banach space  $(E,\|.\|)$ is a complex Banach space $(E_\CC,\|.\|_c)$ such that $E_\CC = E+iE$ as a complex linear space, the norm $\|.\|_c$ restricts to $\|.\|$ on $E$, and $\| \eta +i\xi \| = \| \eta -i\xi \| $ for all $\eta , \xi \in E$ (i.e., $E_\CC =E\otimes_\RR \CC$). Moreover, for any real Banach space $E$ , there is a unique (up to equivalence) complexification of it. We refer the reader to~\cite{Li:Real_algebras} for a general theory of real Banach spaces and real Banach ($^\ast$-)algebras and to~\cite[chap.1]{Schroder:KR-theory} for an extensive exposition of real $\cstar$-algebra. \\
In this way, associated to any real Banach ($^\ast$-)algebra $A$, there is a complex Banach ($^\ast$-)algebra $A_\CC=A\otimes_\RR \CC$. In particular, if $A$ is a real $\cstar$-algebra, then $A_\CC$ admits a structure of a complex $\cstar$-algebra. It is however natural to ask the following question

\begin{qst}~\label{qst:Reality}
Let $B$ be a complex $\cstar$-algebra. Does there exist a closed real $\cstar$-subalgebra $B_r$ of $B$ such that $B\cong B_r\otimes_\RR \CC$?
\end{qst}

Although it was mentioned in~\cite{Li:Real_algebras} that this question was remained open, the answer is in fact "no". Indeed, as we will see later, the existence of $B_r$ is equivalent to the existence of a conjugate-linear involution on $B$, which is also equivalent to $B$ being isomorphic to its conjugate algebra via a $2$-periodic isomorphism). But such an involution induces an involutory anti-automorphism $\vp:B\To B$ ({\it i.e}. $\vp$ verifies $\vp(ab)=\vp(b)\vp(a), \forall a,b\in B$ and $\vp^2=\id$). On the other hand, A. Connes~\cite{Connes:anti-isom} and T. Giordano~\cite{Giordano:Antiautomorphismes} have constructed examples of von Neumann algebras that are not anti-isomorphic to themselves. Very recently, other explicit examples of $\cstar$-algebras not isomorphic to their conjugate algebras have been constructed by N.C. Phillips~\cite{Phillips:anti-isom} and N.C. Phillips and M.G. Viola~\cite{Phillips-Viola:anti-isom}).  

We shall however point out that being anti-isomorphic to itself is not sufficient for a $\cstar$-algebra $B$ to admit a conjugate-linear involution, as it was proved by V. Jones in~\cite{Jones:anti-isom}. 

\subsection{Generalities}
In this subsection we are concerned with those ($\cstar$-)algebras for which Question~\ref{qst:Reality} has a positive answer.

\begin{df}
A \emph{Real ($\ZZ_2$)-graded} $\cstar$-algebra consists of a $\cstar$-algebra $A$ together with 
\begin{itemize}
\item[(i)] an involutive $^\ast$-homomorphism $\alpha: A\To A$ with $\alpha^2=\id$; $\alpha$ is called the \emph{grading};
\item[(ii)] an involutive $^\ast$-automorphism $\sigma_A: A\To A$ which is antilinear, such that $\sigma_A^2=\id$, and $\sigma_A \circ \alpha = \alpha \circ \sigma_A$. $\sigma_A$ is called the \emph{Real structure} of $A$. 
\end{itemize}
We will say that $A$ is a Rg $\cstar$-algebra, for short.
\end{df}

We will often write $(A,\sigma_A)$ for such a Rg $\cstar$-algebra and we decompose $A$ into the direct sum $A=A^{0}\oplus A^1$ where $A^0=\text{Ker}(\frac{\id-\alpha}{2})$ and $A^1=\text{Ker}(\frac{\id+\alpha}{2})$.

An element $a\in A$ is called homogeneous of degree $i$, for $i=0,1 \mod 2$  , if $a\in A^i$. $a$ is said to be \emph{invariant} if it is of degree $0$ and $\sigma_A(a)=a$. We write $|a|$ for the degree of an element $a\in A$. Moreover, it is easy to see that $A^0$ is a $\cstar$-subalgebra of $A$ while $A^1$ is not.\

\begin{ex}
Let $A=A^0\oplus A^1$ be a graded real $\cstar$-algebra. Then its complexification $A_\CC$ is also graded. Indeed, we have $A_\CC=A_\CC^0 \oplus A_\CC^1$. Now the \emph{bar operation} $^-: A_\CC \To A_\CC$ given by $\overline{a+ib}:=a-ib$ defines a Real structure on $A_\CC$.
For instance, any real $\cstar$-algebra $A$ gives rise to a Rg $\cstar$-algebra by taking $A^1=0$.
\end{ex}

\begin{ex}
Given a real $\cstar$-algebra $A$, the direct sum $A\oplus A$ admits a canonical grading given by $(a,b)\mto (b,a)$; then $(A\oplus A)^0 = \lbrace (a,a) \ | \ a\in A \rbrace$ and $(A\oplus A)^1 = \lbrace (a,-a) \ | \ a\in A \rbrace$. This induces a grading on the complex $\cstar$-algebra $A_\CC \oplus A_\CC$ which becomes a Real graded $\cstar$-algebra. This grading is called \emph{the standard odd grading}. In particular, the complex Clifford algebra $\CC l_1 = \CC \oplus \CC$ is a Rg $\cstar$-algebra with its canonical Real structure given by the complex conjugation (see~\cite{Schroder:KR-theory}, for instance).
\end{ex}

\begin{df}
Let $(A,\sigma_A)$ and $(B,\sigma_B)$ be Rg $\cstar$-algebras. A \emph{Real graded homomorphism} between $A$ and $B$ is a homomorphism of $\cstar$-algebras $\vp:A\To B$ that intertwines the Real structures and the gradings. 
\end{df}

In particular, we say that $(A,\sigma_A)$ and $(B,\sigma_B)$ are isomorphic as Rg $\cstar$-algebras, and we write $(A,\sigma_A)\cong (B,\sigma_B)$, if there exists a Rg isomorphism between them.

If $(A,\sigma_A)$ is a Rg $\cstar$-algebra, then the multiplier algebra $\cM(A)$ has also a structure of Real graded $\cstar$-algebra. Indeed, if $\epsilon$ is the grading on $A$ and $(T_1,T_2) \in \cM(A)$, we put $Ad_\epsilon(T_1,T_2):=(\epsilon T_1 \epsilon, \epsilon T_2 \epsilon)$ and  it is easy to see that this defines a grading on $\cM(A)$ with $\cM(A)^{(i)}=\lbrace    (T_1,T_2)\in \cM(A) \ | \ \epsilon T_k \epsilon =(-1)^i T_k, \ k=1,2 \rbrace$; moreover the Real structure is given by $$\sigma_A(T_1,T_2):= (\sigma_A T_1  \sigma_A, \sigma_A  T_2  \sigma_2).$$ A subspace $B$ of $A$ is \emph{Real graded} if it is invariant under $\sigma_A$ and if it is the direct sum of the intersections $B \cap A^i$ (or equivalently, if it is invariant under the grading of $A$). For instance, it is easy to check that the centre of any Rg $\cstar$-algebra is Rg.

Le $I$ be a Real graded ideal in $(A,\sigma)$. Let $[a]$ denote the class of $a$ in $A/I$, then we can show that the maps $\sigma ([a]):= [\sigma(a)]$ and $\epsilon([a]):=[\epsilon(a)]$, are well defined from $A/I$ to $A/I$, giving us a grading and a Real structure on the quotient $\cstar$-algebra $A/I$.

Now let us give the following simple characterisation of Rg $\cstar$-algebras.

\begin{lem}~\label{bar-isom}
Let $(A,\sigma_A)$ be a Rg $\cstar$-algebra. Then there exists a real $\ZZ_2$-graded $\cstar$-algebra $A_\RR$ such that $(A,\sigma_A) \cong (A_\RR \otimes_\RR \CC, ^-)$, where ($^-$) is the bar operation.
\end{lem}

Mainly speaking, a Rg $\cstar$-algebra is just a graded complex $\cstar$-algebra which is the complexification of a graded real  $\cstar$-algebra, together with the bar operation. This justifies the terminology "\emph{Real}" used.

\begin{proof}
Put $A_\RR:=\lbrace a\in A \ | \ \sigma_A(a)=a \rbrace$. Then $A_\RR$ is a graded real $\cstar$-algebra. Moreover, it is very easy to check that the map $A \To A_\RR + iA_\RR, \ a\mto \frac{a+\sigma_A(a)}{2} + i(\frac{a-\sigma_A(a)}{2i})$ extends to an isomorphism of complex $\cstar$-algebras intertwining the Real structures and the gradings.
\end{proof}

\begin{rem}
Similarly, we will call \emph{Rg Banach space} any complex graded Banach space which is the complexification of a Banach space over $\RR$.
\end{rem}

\begin{ex}
Let $(X,\tau)$ be a (Hausdorff and locally compact) Real space. Then $\tau$ induces a Real structure, also denoted by $\tau$, on the $\cstar$-algebra $\Co(X)$ of complex valued functions on $X$ vanishing at infinity, given by $\tau(f)(x)=\overline{f(\tau(x))}$, for $f\in \Co(X), \ x\in X$. Therefore, from Lemma~\ref{bar-isom} we have $(\Co(X),\tau)\cong (\Co(X,\tau) \otimes_\RR \CC , ^-)$ where $\Co(X,\tau):= \lbrace f\in \Co(X) \ | f(\tau(x))=\overline{f(x)} , \ \forall x\in X \rbrace$ is the real $\cstar$-algebra of invariant elements of $(\Co(X),\tau)$.
\end{ex}

We must also say something about the tensor product of two Real graded $\cstar$-algebras. This paragraph is a direct adaptation of~\cite{Parker:Brauer} to the Real case. Let $(A,\sigma)$ be a Real graded $\cstar$-algebra. A \emph{Real graded linear functional} on $A$ is a linear functional $f:A\To \CC$ such that $f_{|_{A^1}}=0$ and $f(\sigma(a))=\overline{f(a)}$ for all $a\in A$. A Real graded state on $A$ is a positive linear functional $s$ on $A$ such that $\| s\| =1$. Suppose that $(A,\sigma)$ and $(B,\varsigma)$ are separable, Real graded  $\cstar$-algebras, then $(A\hat{\odot}B,\sigma\hat{\odot}\varsigma)$ denotes the algebraic Real graded tensor product of $A$ and $B$, where elements are graded be $|a\hat{\odot}b| = |a| +|b|$, and the Real structure is given by $\sigma\hat{\odot}\varsigma(a\hat{\odot}b):=\sigma(a)\hat{\odot}\varsigma(b)$. The product and involutions are defined by 
$$(a\hat{\odot}b)(a'\hat{\odot}b'):=(-1)^{|b||a'|}(aa'\hat{\odot}bb'),$$ $$(a\hat{\odot}b)^\ast:=(-1)^{|a||b|}(a^\ast \hat{\odot}b^\ast).$$

Now if $s$ and $t$ are Real graded states on $A$ and $B$ respectively, let $$(s\hat{\odot}t)(c^\ast c):= \sum_{i,j=1}^n s(a_i^\ast a_j)t(b_i^\ast b_j),$$ for $c=\sum_{i=1}^n a_i\hat{\odot}b_i \in A\hat{\odot}B$. Then $s\hat{\odot}t$ is a Real graded state on $A\hat{\odot}B$. We define a $\cstar$-norm on $A\hat{\odot}B$ by $$\Vert c \Vert:= \underset{s,t,d}{sup}\frac{(s\hat{\odot}t)(d^\ast c^\ast cd)}{(s\hat{\odot}t)(d^\ast d)},$$ where the supremum is taken over all Real graded states $s$ on $A$, $t$ on $B$, and over all $d\in A\hat{\odot}B$ with $(s\hat{\odot}t)(d^\ast d) \neq 0$. The completion of $A\hat{\odot}B$ with respect to this norm is a graded $\cstar$-algebra denoted by $A\hat{\otimes}B$; moreover, $\sigma\hat{\odot}\varsigma$ extends to a Real involution on $A\hat{\otimes}B$ which gives a Real graded $\cstar$-algebra $(A\hat{\otimes}B,\sigma\hat{\otimes}\varsigma)$ called the {\it (Real graded) tensor product} of $(A,\sigma)$ and $(B,\varsigma)$.


\subsection{Elementary graded complex $\cstar$-algebras}

A complex graded $\cstar$-algebra $A$ is called \emph{elementary of parity} $0$ (resp. \emph{of parity} $1$) if it isomorphic as a graded $\cstar$-algebra to $\cK(\hat{\cH})$ (resp. to $\cK(\cH)\oplus \cK(\cH)$), where $\hat{\cH}$ (resp. $\cH$) is a complex graded Hilbert space (resp. a complex Hilbert space),  and $\cK(\cH)\oplus \cK(\cH)$ is equipped with the standard odd grading.

\begin{ex}[The complex Clifford $\cstar$-algebras]~\label{ex:Clifford}
The complex Clifford $\cstar$-algebras $\Cl_p$ can be defined as graded $\cstar$-algebras of compact operators in the following way. If $p=2m$, $\Cl_p$ is $\Cl_{2m}:=\cK(\CC^{2^{m-1}}\oplus \CC^{2^{m-1}})$ equipped with the standard even grading $\text{Ad}_\epsilon$, where $\epsilon=\left(\begin{array}{ll}0&1\\1&0\end{array}\right)$; if $p=2m+1$ is odd, then $\Cl_{2m+1}:=\cK(\CC^{2^m})\oplus \cK(\CC^{2^m})$ with the standard odd grading. We then see that the $\Cl_{2m}$'s are graded elementary $\cstar$-algebras of parity $0$, while the $\Cl_{2m+1}$'s are graded elementary $\cstar$-algebras of parity $1$. Moreover, these algebras verify $\Cl_p\hat{\otimes}\Cl_q \cong \Cl_{p+q}$ as graded $\cstar$-algebras (see for instance~\cite[\S.14.5]{Blackadar:K-theory},~\cite{Atiyah-Bott-Schapiro:Clifford}).	
\end{ex}

For the sake of simplicity, we assume in what follows that $\cH$ is a complex separable infinite-dimensional Hilbert space. Then, by choosing an isomorphism $\cH\cong \cH\oplus \cH$, we have a complex graded Hilbert space $\hat{\cH}:=\cH\oplus \cH=(\cH\oplus\cH)^{0}\oplus (\cH\oplus\cH)^{1}$, where the grading is given by $(x,y)\mto (y,x)$. We thus obtain a complex graded elementary $\cstar$-algebra $\what{\cK}_{ev}:=\cK(\hat{\cH})$ of parity $0$ (here "\emph{ev}" stands for \emph{even}) whose grading automorphism is the unitary $\begin{pmatrix}0&1\\1&0\end{pmatrix}$. We also get a graded elementary $\cstar$-algebra $\what{\cK}_{odd}:=\cK(\cH)\oplus \cK(\cH)$ with the standard odd grading. The next subsections are aimed at describing the Real structures of $\what{\cK}_{ev}$ and $\what{\cK}_{odd}$.


\subsection{Real structures on $\what{\cK}_{ev}$}

\begin{df}~\label{df:structures_H}
A {\it Real structure} (resp. {\it quaternionic structure}) on $\what{\cH}$ is a homogeneous anti-unitary $J:\what{\cH}\To \what{\cH}$ such that $J^2=1$ (resp. such that $J^2=-1$). \\
Real structures on $\what{\cH}$ will be denoted as $J_{\RR}$,  or as $J_{i,\RR}, i=0,1$ if we need to emphasise the degree $i$ of $J_{\RR}$. Similarly, quaternionic structures will be denoted as $J_{\HH}$, or $J_{i,\HH}, i=0,1$.  	
\end{df}

Given a Real structure $J_\RR:\what{\cH}\To \what{\cH}$, its $(+1)$-eigenspace $\what{\cH}_{J_{\RR}}:=\{x\in \what{\cH} \mid J_\RR(x)=x\}$ (that we will also denote by $\what{\cH}_\RR$ if there is no risk of confusion) is a real graded separable infinite-dimensional Hilbert space such that $\what{\cH}\cong \what{\cH}_{J_\RR}\otimes_{\RR}\CC$. Furthermore, there exists an orthonormal basis $\{e_n\}_{n\in \NN}$ of $\what{\cH}$, unique up to conjugation with homogeneous elements in the orthogonal group $O(\what{\cH}_{J_\RR})$, such that $J_\RR$ is given by $J_\RR(x):=\sum_n\bar{x}_ne_n$ for all $x=\sum_n x_ne_n\in \what{\cH}$. Writing $J_\RR$ in this form, we get the following straightforward lemma. 

\begin{lem}
Let $J_\RR$ be as above. Define $\sigma_\RR: \what{\cK}_{ev}\To \what{\cK}_{ev}$ by $\sigma_{\RR}(T):=J_\RR TJ_\RR$. Then $\sigma_\RR$ is a Real structure on $\what{\cK}_{ev}$ such that $(\what{\cK}_{ev})_{\sigma_\RR}\cong \cK_\RR(\what{\cH}_\RR)$ as real graded $\cstar$-algebras.	
\end{lem}

Now suppose $J_\HH:\what{\cH}\To \what{\cH}$ is a quaternionic structure. Define the degree $0$ operator $I:\what{\cH}\To \what{\cH}$ by $Ix:=ix$. Then $I^2=-1$, and $IJ=-JI$. Thus, we can define the operator $K:=IJ:\what{\cH}\To \what{\cH}$ which has the same degree as $J$ and which is such that $K^2=-1=IJK$. It turns out that there exists a graded action of the quaternions $\HH$ on $\what{\cH}$ given by $(i,x)\mto ix, (j,x)\mto jx:=Jx$, and $(k,x)\mto kx:=Kx=IJx$, where $\{1,i,j,k\}$ is the usual basis of the division ring $\HH$. Let $\what{\cH}_{J_\HH}$ (or just $\cH_\HH$ if there is no risk of confusion) be the quaternionic graded Hilbert space, where the $\HH$-valued inner product is given by $\<x,y\>_\HH:=\<x,y\>+\<x,Jy\>j$ if $\<\cdot,\cdot\>$ denotes the complex scalar product of $\what{\cH}$.  

\begin{lem}
Let $J_\HH$ be as above. Define $\sigma_\HH:\what{\cK}_{ev}\To \what{\cK}_{ev}$ by $\sigma_\HH(T):=-J_\HH TJ_\HH$. Then $\sigma_\HH$ is a Real structure on $\what{\cK}_{ev}$ such that $(\what{\cK}_{ev})_{\sigma_\HH}$ is isomorphic, under a graded isomorphism, to the graded real $\cstar$-algebra $\cK_\HH(\what{\cH}_\HH)$ of the compact $\HH$-linear operators on the graded quaternionic Hilbert space $\what{\cH}_\HH$.  	
\end{lem}
 
\begin{proof}
 The only thing we need to show is the graded isomorphism. Suppose that $T\in (\what{\cK}_{ev})_{\sigma_\HH}$. Then, $TJ_\HH =J_\HH T$, so that $T$ extends uniquely to a compact $\HH$-linear operator $\widetilde{T}:\what{\cH}_\HH \To \what{\cH}_\HH$ through the formula $\widetilde{T}(jx):=J_\HH(Tx)$ for $x\in \what{\cH}$. This provides a homomorphism of real graded $\cstar$-algebras $(\what{\cK}_{ev})_{\sigma_\HH}\To \cK_\HH(\what{\cH}_\HH), T\mto \widetilde{T}$. Conversely, any $\widetilde{T}\in \cK_\HH(\what{\cH}_\HH)$ induces a unique $T\in \cL(\what{\cH})$ such that $Tx=\widetilde{T}x$ for all $x\in \what{\cH}$. Then $T\in \what{\cK}_{ev}$. Moreover, one has $(TJ_\HH) x=T(J_\HH x)=\widetilde{T}(jx)=j\widetilde{T}x=(J_\HH T)x$; hence, $TJ_\HH =J_\HH T$, and then $\sigma_\HH(T)=T$. We then get a homomorphism of real graded $\cstar$-algebras $\cK_\HH(\what{\cH}_\HH)\To (\what{\cK}_{ev})_{\sigma_\HH}, \widetilde{T}\mto T$. It is easy to check that these two homomorphisms are inverse of each other. 	
\end{proof}

The following result classifies all the Real structures on $\what{\cK}_{ev}$.
 
\begin{pro}~\label{pro:R-structures_K}
Suppose that $\sigma$ is a Real structure on $\what{\cK}_{ev}$. Then, $\sigma$ is either of the form $\sigma_\RR$, or of the form $\sigma_\HH$.	
\end{pro}

\begin{proof}
Choose an orthonormal basis $\{e_n\}$ of $\what{\cH}$, and for $T\in \what{\cK}_{ev}$, define $\overline{T}\in \what{\cK}_{ev}$ by $\overline{T}(x):=\overline{T(\bar{x})}$, where if $x=\sum_n x_ne_n$, we set $\bar{x}:=\sum_n\bar{x}_ne_n$. Then $\overline{T}=vTv$, where $v:\what{\cH}\To \what{\cH}$ is the anti-unitary defined by the complex conjugation with respect to the basis $\{e_n\}$. Moreover, $v^2=1$. Now, define $\bar{\sigma}\in \text{Aut}^{(0)}(\what{\cK}_{ev})$ by $\bar{\sigma}(T):=\sigma(\overline{T})$. Then, there exists a homogeneous unitary $u\in \what{\U}(\what{\cH})$ such that $\bar{\sigma}=\text{Ad}_u$. Whence, $\sigma(T)=\bar{\sigma}(\overline{T})=uvTvu^{-1}=JTJ^{-1}$, where $J:=uv$. Observe that $J$ is a homogeneous anti-unitary since $v$ is. Furthermore, for all $T\in \what{\cK}_{ev}$, we have $T=\sigma^2(T)=J^2T(J^{-1})^2$; therefore $J^2=\pm 1$.
\end{proof}

\begin{df}
We say that a Rg elementary $\cstar$-algebra $(A,\sigma)$ of parity $0$ is (of type) $[0;\ve,\eta]$, where $\ve=0,1$, $\eta=\pm$, if its Real structure is induced by an anti-unitary $J$ of degree $\ve$ such that $J^2=\eta1$. 	
\end{df}

\begin{rem}
It follows from Proposition~\ref{pro:R-structures_K} that there are four types of Rg elementary $\cstar$-algebras of parity $0$: $[0;0,+], [0;0,-],[0;1,+]$, and $[0;1,-]$. 
\end{rem}

\begin{rem}
Regarding $\cK(\cH)$ as of parity $0$ (with the trivial grading of $\cH$), we get that every Real structure on $\cK(\cH)$ is the conjugation with an anti-unitary $J:\cH\To \cH$ such that $J^2=\pm1$. Thus, a Real graded $\cstar$-algebra of the form $\cK(\cH)$ is either a $[0;0,+]$ or $[0;0,-]$.   	
\end{rem}

\begin{ex}[Real structures on $\Cl_2$]~\label{ex:Real_Cl_2}
Consider the second Clifford algebra $\Cl_2=\cK(\CC\oplus \CC)=M_2(\CC)$, equipped with the standard even grading. There is a canonical Real structure $J_\RR$ of degree $0$ on the graded Hilbert space $\CC\oplus \CC$ given by the complex conjugation, and a canonical quaternionic structure of degree $0$ $J_{0,\HH}=iJ_{0,\RR}$, which induce the same Real structure $cl_{0,2}$ on $\Cl_2$ such that $(\Cl_2)_{\sigma_\RR}\cong (\Cl_2)_{cl_{0,2}}\cong M_2(\RR) \cong Cl_{0,2}$. In other words, $\Cl_2$ is the complexification of the second real Clifford algebra $Cl_{0,2}$ (see~\cite{Atiyah-Bott-Schapiro:Clifford} for more details on the real Clifford algebras $Cl_{p,q}$). 
However, $\Cl_2$ is also the complexification of the quaternions $\HH$ as follows. Define the quaternionic structure $J_{1,\HH}:\CC\oplus \CC\To \CC\oplus \CC$ of degree $1$ by $(x,y)\mto (\bar{y},-\bar{x})$. The graded quaternionic Hilbert space obtained is $\HH$; the Real structure induced by $J_{1,\HH}$ is denoted by $cl_{2,0}$. Observe that $(\Cl_2)_{cl_{2,0}}=\cK_\HH(\HH)=\HH\cong Cl_{2,0}$. Moreover, this Real structure is equivalent to the one induced by the anti-unitary $J_{1,\RR}(x,y):=(\bar{y},\bar{x})$. These two Real structures will play a central role in the classification of elementary Rg $\cstar$-algebras in Subsection~\ref{subsect:BrR(pt)}.
 \end{ex}
 
 
\subsection{Real structures on $\what{\cK}_{odd}$}

In this subsection we describe the Real structures on $\what{\cK}_{odd}$. We start by some observations which will be useful. Suppose we are given a trivially graded $\cstar$-algebra $A$. Then, any Real structure $\sigma$ on $A$ defines two different Real structures $\sigma\oplus \sigma$ and $\sigma\oplus(-\sigma)$ on the graded $\cstar$-algebra $A\oplus A$ (with the standard odd grading), respectively given by $(a,b)\mto (\sigma(a),\sigma(b))$ and $(a,b)\mto (\sigma(a),-\sigma(b))$. Notice that the latter Real structure is equivalent to $(a,b)\mto (\sigma(b),\sigma(a))$. Furthermore, if we denote $A_\RR:=A_\sigma$, then on the one hand, we get that $(A\oplus A)_{\sigma\oplus\sigma}$ is the graded real $\cstar$-algebra $A_\RR\oplus A_\RR$ with the standard odd grading, and on the other hand,  $(A\oplus A)_{\sigma\oplus(-\sigma)}=A_\RR\oplus iA_\RR$ is isomorphic to the graded graded $\cstar$-algebra $A_{real}$ which is the underlying $\RR$-algebra of $A$. It is easy to see that the grading of $A_{real}$ is given by $A_{real}^0=A_\RR$ and $A_{real}^1=iA_\RR$. Conversely, we have the following.

\begin{pro}
Let $A$ be a complex $\cstar$-algebra, and let $A\oplus A$ be equipped with the standard odd grading $(a,b)\mto (b,a)$. Suppose $\tau$ is a Real structure on $A\oplus A$. Then, $\tau$ is either of the form $(a,b)\mto (\sigma(a),\sigma(b))$ or $(a,b)\mto (\sigma(b),\sigma(a))$, where $\sigma:A\To A$ is a Real structure on the ungraded $\cstar$-algebra $A$.	
\end{pro}

\begin{proof}
Since $\tau$ is of degree $0$, it can be written in the form $\tau=\begin{pmatrix}\tau^+&0\\0&\tau^-\end{pmatrix}$ with respect to the decomposition $A\oplus A=(A\oplus A)^0\oplus (A\oplus A)^1$, where $\tau^+:(A\oplus A)^0\To (A\oplus A)^0$ is a Real structure on the $\cstar$-subalgebra $(A\oplus A)^0$ of $A\oplus A$, and $\tau^-:(A\oplus A)^1\To (A\oplus A)^1$ is an anti-linear isomorphism of vector space. For all $(a,a)\in (A\oplus A)^0$, $\tau^+(a,a)\in (A\oplus A)^0$, so that it is in the form $(\sigma(a),\sigma(a))$. If $(a_i,a_i)\To (a,a)\in (A\oplus A)^0$, then $(\sigma(a_i),\sigma(a_i))=\tau^+(a_i,a_i)\To \tau^+(a,a)=(\sigma(a),\sigma(a))$, and then $\sigma(a_i)\To \sigma(a)$ in $A$. Furthermore, it is straightforward that $\sigma(ab)=\sigma(a)\sigma(b)$, $\sigma(\lambda a)=\bar{\lambda}\sigma(a)$ for all $\lambda\in \CC, a\in A$, and that $\sigma^2=\id$, so that $\sigma$ is a Real structure on $A$. Now, for all $(b,-b)\in (A\oplus A)^1$, $(b,-b)\cdot (b,-b)=(b^2,b^2)\in (A\oplus A)^0$; thus, $$(\tau^-(b,-b))^2=\tau(b^2,b^2)=\tau^+(b^2,b^2)=(\sigma(b)^2,\sigma(b)^2).$$ Hence, since this is true for all $b\in A$, we obtain $\tau^-(b,-b)=(\pm \sigma(b),\mp \sigma(b))$. If $\tau^-(b,-b)=(\sigma(b),-\sigma(b))$, then $\tau$ is given by $\tau(a,b)=(\sigma(a),\sigma(b))$, for all $(a,b)\in a\oplus a$, and if $\tau^-(b,-b)=(-\sigma(b),\sigma(b))$, then  for all $(a,b)\in A\oplus A$, $\tau(a,b)=(\sigma(b),\sigma(a))$. 	
\end{proof}

\begin{df}
A Real structure $\tau$ on $A\oplus A$ is called {\it even} if it is of the form $(a,b)\mto (\sigma(a),\sigma(b))$, it is {\it odd} if it is of the form $(a,b)\mto (\sigma(b),\sigma(a))$, where $\sigma$ is a Real structure on the ungraded $\cstar$-algebra $A$. 	
\end{df}

\begin{pro}
Assume $\tau:A\oplus A\To A\oplus A$ is a Real structure. Then 
\begin{itemize}
	\item[(a)] $(A\oplus A,\tau)\cong (A\hat{\otimes}\Cl_1,\sigma\hat{\otimes}cl_{0,1})$, if $\tau$ is even, and
	\item[(b)] $(A\oplus A,\tau)\cong (A\hat{\otimes}\Cl_1,\sigma\hat{\otimes}cl_{1,0})$, if $\tau$ is odd.	
	\end{itemize}	
\end{pro}

\begin{proof}
As graded complex $\cstar$-algebras, $A\oplus A\cong A\hat{\otimes}\Cl_1\cong A\otimes \Cl_1$ (cf.~\cite[Corollary 14.5.3]{Blackadar:K-theory}). If $\tau$ is even, then as graded real $\cstar$-algebras, $(A\oplus A)_\RR \cong A_\RR\oplus A_\RR \cong (A_\RR\hat{\otimes}Cl_{0,1})=(A\hat{\otimes}\Cl_1)_{\sigma\hat{\otimes}cl_{0,1}}$, where $A_\RR:=A_\sigma$, and $(A\oplus A)_\RR:=(A\oplus A)_\tau$; this establishes (a). If $\tau$ is odd, then $(A\oplus A)_\RR\cong A_{real}\cong A_\RR\hat{\otimes}\CC\cong A_\RR\hat{\otimes}Cl_{1,0}\cong(A\hat{\otimes}\Cl_1)_{\sigma\hat{\otimes}cl_{1,0}}$, which establishes (b).	
\end{proof}

\begin{cor}~\label{cor:K_odd_Cl}
Suppose $\sigma$ is a Real structure on $\what{\cK}_{odd}$. Then, there exists an anti-unitary $J:\cH\To \cH$ with $J^2=\pm1$, such that either $(\what{\cK}_{odd},\sigma)\cong (\cK(\cH)\hat{\otimes}\Cl_1,\text{Ad}_J\hat{\otimes}cl_{0,1})$, or $(\what{\cK}_{odd},\sigma)\cong (\cK(\cH)\hat{\otimes}\Cl_1,\text{Ad}_J\hat{\otimes}cl_{1,0})$. 
\end{cor}

\begin{df}
We say that a Rg elementary $\cstar$-algebra $(\what{\cK}_{odd},\sigma)$ of parity $1$ is (of type) $[1;\ve,\eta]$, if the Real structure is of parity $\ve$ ({\it i.e}., $\ve$ is $0$ if $\sigma$ is even, and $1$ if $\sigma$ is odd), and if the anti-unitary $J$ of Corollary~\ref{cor:K_odd_Cl} is such that $J^2=\eta1$, where $\eta=\pm$.	
\end{df}

It then follows that there are four of such types: $[1;0,+],[1;0,-],[1;1,+]$, and $[1;1,-]$.

\begin{ex}
$(\Cl_1,cl_{0,1})$ and $(\Cl_1,cl_{1,0})$ are of types $[1;0,+]$ and $[1;1,+]$, respectively.
\end{ex}

\medskip 


\subsection{The classification table}~\label{subsect:BrR(pt)}
We start this subsection with the following lemma.

\begin{lem}~\label{lem:product_even}
Let $\hat{\cH}_1$ and $\hat{\cH}_2$ be two graded complex Hilbert spaces, and let $J_i, \ i=1,2$ be an anti-unitary of degree $\ve_i$ on $\hat{\cH}_i$ such that $J_i^2=\pm1$. Denote by $g_i, i=1,2$ the grading automorphism of $\cK(\hat{\cH}_i)$. Then, there is an isomorphism of Real graded (elementary) $\cstar$-algebras \[(\cK(\hat{\cH}_1)\hat{\otimes}\cK(\hat{\cH}_2),\text{Ad}_{J_1}\hat{\otimes}\text{Ad}_{J_2})\cong (\cK(\hat{\cH}_1\hat{\otimes}\hat{\cH}_2),\text{Ad}_J),\]
where $J:=J_1g_1^{\ve_2}\hat{\otimes}J_2g_2^{\ve_2}$. 	
\end{lem}

\begin{proof}
The isomorphism of graded $\cstar$-algebras $\cK(\hat{\cH}_1)\hat{\otimes}\cK(\hat{\cH}_2)\To \cK(\hat{\cH}_1\hat{\otimes}\hat{\cH}_2)$ is given on homogeneous tensors by \[(T_1\hat{\otimes}T_2)(x_1\hat{\otimes}x_2)=(-1)^{|T_2|\cdot |x_1|}T_1(x_1)\hat{\otimes}T_2(x_2).\]
Moreover, a simple calculation shows that this is actually a Real isomorphism, when $\cK(\hat{\cH}_1\hat{\otimes}\hat{\cH}_2)$ is equipped with the Real structure $\text{Ad}_J$; indeed,  
\begin{align*}
\text{Ad}_J(T_1\hat{\otimes}T_2) & =\left(J_1g_1^{\ve_2}\hat{\otimes}J_2g_2^{\ve_1}\right)\left(T_1\hat{\otimes}T_2\right)\left(J_1g_1^{\ve_2}\hat{\otimes}J_2g_2^{\ve_1}\right)^\ast \\
 & = (-1)^{\ve_1\ve_2+\ve_2|T_1|}\left(J_1g_1^{\ve_2}T_1\hat{\otimes}J_2g_2^{\ve_1}T_2\right)\left((g_1^\ast)^{\ve_2}J_1^\ast \hat{\otimes}(g_2^\ast)^{\ve_1}J_2^\ast\right) \\
  & = (-1)^{\ve_2|T_1|+\ve_1|T_2|}\left((J_1g_1^{\ve_2}T_1(g_1^\ast)^{\ve_2}J_1^\ast)\hat{\otimes}(J_2g_2^{\ve_1}T_2(g_2^\ast)^{\ve_1}J_2^\ast)\right)\\
  & = \text{Ad}_{J_1}(T_1)\hat{\otimes}\text{Ad}_{J_2}(T_2).
\end{align*}
\end{proof}

A particular case of this lemma is the following.

\begin{cor}~\label{cor:product_K-Cl2}
Le $J$ be an anti-unitary on the ungraded Hilbert space $\cH$ such that $J^2=\eta1$, where as usual $\eta=\pm$. Let $cl_{0,2}$ and $cl_{2,0}$ be the Real structures of $\Cl_2$ defined in Example~\ref{ex:Real_Cl_2}. Then,
\begin{itemize} 
\item $[0;0,\eta]\cong (\cK(\cH),\text{Ad}_J)\hat{\otimes}(\Cl_2,cl_{0,2})$, where $J:\cH\To \cH$ is such that $J^2=\eta1$, and
\item $[0;1,\eta]\cong (\cK(\cH),\text{Ad}_J)\hat{\otimes}(\Cl_2,cl_{2,0})$, where $J:\cH\To \cH$ is such that $J^2=-\eta1$.
\end{itemize}	
\end{cor}

The next theorem can be viewed as a generalisation of C.T. Wall's result~\cite[Theorem 3]{Wall:Brauer} to the infinite dimensional case.

\begin{thm}~\label{thm:products_BrR_pt}
The type of the Real graded tensor product of two Real graded elementary $\cstar$-algebras $(A,\sigma_A)$ and $(B,\sigma_B)$ depends only on those of $(A,\sigma_A)$ and $(B,\sigma_B)$. Moreover, we have the formulas 
\begin{align}
  	[0;\ve_1,\eta_1]\hat{\otimes} [0;\ve_2,\eta_2] & = [0;\ve_1+\ve_2, (-)^{\ve_1\ve_2}\eta_1\eta_2] ~\label{eq:1:product_even}\\
  	[0;\ve_1,\eta_1]\hat{\otimes}[1;\ve_2,\eta_2] &  = [1;\ve_1+\ve_2,(-)^{\ve_1+\ve_1\ve_2}\eta_1\eta_2] ~\label{eq:2:product_even-odd} \\
  	[1;\ve_1,\eta_1]\hat{\otimes}[1;\ve_2,\eta_2] & = [0;1+\ve_1+\ve_2,(-)^{\ve_1\ve_2}\eta_1\eta_2] ~\label{eq:3:product_odd}, 
\end{align}  
where the sum of degrees is mod $2$.	
\end{thm}

\begin{proof}
The formula~\eqref{eq:1:product_even} is nothing more than Lemma~\ref{lem:product_even}. Indeed, we have seen that the Real structure on $\cK(\hat{\cH}_1\hat{\otimes}\hat{\cH}_2)$ is defined by the anti-unitary $J=J_1g_1^{\ve_2}\hat{\otimes}J_2g_2^{\ve_1}$. The degree of $J$ is then $\ve=\ve_1+\ve_2$, and $J^2=(-1)^{\ve_1\ve_2}J_1^2\hat{\otimes}J_2^2=(-1)^{\ve_1\ve_2}\eta_1\eta_2 1\hat{\otimes}1$.\\
Also, combining Corollary~\ref{cor:K_odd_Cl}, Corollary~\ref{cor:product_K-Cl2}, we get ~\eqref{eq:2:product_even-odd}, by considering the isomorphism of Rg $\cstar$-algebras \[(\cK(\hat{\cH}_1),\text{Ad}_{j_1})\hat{\otimes}(\cK(\cH_2)\hat{\otimes}\Cl_1,\text{Ad}_{J_2}\hat{\otimes}\tau_1)\cong (\cK(\hat{\cH}_1\hat{\otimes}\cH_2)\hat{\otimes}\Cl_1,\text{Ad}_J\hat{\otimes}\tau_1),\]
where $J=J_1\hat{\otimes}J_2$, and $\tau_1$ is either $cl_{0,1}$ or $cl_{1,0}$. \\
Finally, the equality~\eqref{eq:3:product_odd} follows from Corollary~\ref{cor:product_K-Cl2} and the following isomorphisms of Rg $\cstar$-algebras, which can be established by merely using the properties of the real Clifford algebras~\cite{Atiyah-Bott-Schapiro:Clifford}:
\begin{align*}
	(\Cl_1,cl_{0,1})\hat{\otimes}(\Cl_1,cl_{0,1}) & \cong (\Cl_2,cl_{0,2})	\\
	(\Cl_1,cl_{0,1})\hat{\otimes}(\Cl_1,cl_{1,0}) & \cong (\Cl_2,cl_{0,2})\\
	(\Cl_1,cl_{1,0})\hat{\otimes}(\Cl_1,cl_{1,0}) & \cong (\Cl_2,cl_{2,0}).
	\end{align*}	
\end{proof}

We summarize all the preceding discussions by the following result.

\begin{dfpro}
Denote by $\wK_0$, the Rg elementary $\cstar$-algebra $(\wK_{ev},\text{Ad}_{J_\RR})$, where $J_\RR$ is the anti-unitary of degree $0$ on $\hat{\cH}$ defined by $(x,y)\mto (\bar{x},\bar{y})$ ("${}^-$" is the complex conjugation with respect to an arbitrary orthonormal basis of $\what{\cH}$). Then $\wK_0$ is of type $[0;0,+]$.

Say that two Rg elementary $\cstar$-algebras $A$ and $B$ are \emph{stably isomorphic} if $A\hat{\otimes}\wK_0\cong B\hat{\otimes}\wK_0$, as Rg $\cstar$-algebras.

Stable isomorphism classes of Rg elementary $\cstar$-algebras form an abelian group of order $8$ under Rg tensor products, denoted by $\wRBr(\ast)$, and called \emph{the Rg Brauer group of the point}. The zero element of $\wRBr(\ast)$ is the element $\wK_0$.

 Furthermore, elements of $\wRBr(\ast)$ are, up to stable isomorphisms, classified by the following $8$-periodic table

\renewcommand{\arraystretch}{1} 
\setlength{\tabcolsep}{0.7cm} 
\begin{table}[!h]
\centering

\begin{tabular}{|cc|} \hline
 Parity $0$  & Parity $1$ \\ \hline
  $\what{\cK}_0:=[0;0,+]$    & $\what{\cK}_1:=[1;0,+]$  \\ 
  $\what{\cK}_2:=[0;1,+]$    & $\what{\cK}_3:=[1;1,-]$  \\ 
  $\what{\cK}_4:=[0;0,-]$    & $\what{\cK}_5:=[1;0,-]$   \\ 
  $\what{\cK}_6:=[0;1,-]$    & $\what{\cK}_7:=[1;1,+]$   \\ \hline

\end{tabular}
\caption{Classification of Rg elementary $\cstar$-algebras\label{classification}}
\end{table} 	

\end{dfpro}

\begin{rem}
Under the notations of Table~\ref{classification}, we set for all $n\in \NN^\ast$:
\[\what{\cK}_n:= \underbrace{\what{\cK}_1\hat{\otimes}\cdots\hat{\otimes}\what{\cK}_1}_{n-times}.\]
Then $\what{\cK}_p\hat{\otimes}\what{\cK}_q\cong \what{\cK}_{p+q}$, and from Theorem~\ref{thm:products_BrR_pt}, $\what{\cK}_n\cong \what{\cK}_{n+8}$ for all $n\in \NN$. Now, define $\what{\cK}_{-n}$ as the inverse of $\what{\cK}_n$ in $\wRBr(\ast)$. Then $\what{\cK}_{-n}= \what{\cK}_{8-n}$	
\end{rem}

\begin{ex}~\label{ex:type-of-Clifford}(Cf.~\cite{Schroder:KR-theory}).
One can determine the Real structures of the graded Clifford $\cstar$-algebras $\Cl_n$ (recall Example~\ref{ex:Clifford}), for $n\in \NN^\ast$, in the following way: decompose $n$ into a sum $p+q$, and consider the Real space $\RR^{p,q}\otimes_\RR\CC$, with the obvious involution; this latter induces a Real structure $cl_{p,q}$ on the graded $\cstar$-algebra $\Cl_n=Cl(\RR^{p,q}\otimes_\RR\CC)$, such that the Real part is isomorphic to the graded real Clifford algebra $Cl_{p,q}$. For this reason, we denote the thus obtained Real graded $\cstar$-algebra by $\Cl_{p,q}$. Indeed, for every decomposition $n=p+q$, it is not hard to check that $\Cl_{p,q}$ is a Rg elementary $\cstar$-algebra of type $q-p \mod 8$ (see for instance~\cite{Donovan-Karoubi}).	
\end{ex}

%
%



\begin{bibdiv}\begin{biblist}

\bib{Atiyah:K_Reality}{article}{
  author={Atiyah, M. F.},
  title={$K$\nobreakdash -theory and reality},
  journal={Quart. J. Math. Oxford Ser. (2)},
  volume={17},
  date={1966},
  pages={367--386},
  issn={0033-5606},
  review={\MRref {0206940}{34\,\#6756}},
}

\bib{Atiyah-Bott-Schapiro:Clifford}{article}{
  author={Atiyah, M. F.},
  author={Bott, R.},
  author={Shapiro, A.},
  title={Clifford modules},
  journal={Topology},
  volume={3},
  date={1964},
  number={suppl. 1},
  pages={3--38},
  issn={0040-9383},
  review={\MR {0167985 (29 \#5250)}},
}

\bib{Atiyah-Segal:Twisted_K}{article}{
  author={Atiyah, M.},
  author={Segal, G.},
  title={Twisted $K$-theory},
  journal={Ukr. Mat. Visn.},
  volume={1},
  date={2004},
  number={3},
  pages={287--330},
  issn={1810-3200},
  translation={ journal={Ukr. Math. Bull.}, volume={1}, date={2004}, number={3}, pages={291--334}, issn={1812-3309}, },
  review={\MR {2172633 (2006m:55017)}},
}

\bib{Blackadar:K-theory}{book}{
  author={Blackadar, Bruce},
  title={\(K\)\nobreakdash -theory for operator algebras},
  series={Mathematical Sciences Research Institute Publications},
  volume={5},
  edition={2},
  publisher={Cambridge University Press},
  place={Cambridge},
  date={1998},
  pages={xx+300},
  isbn={0-521-63532-2},
  review={\MRref {1656031}{99g:46104}},
}

\bib{Connes:anti-isom}{article}{
  author={Connes, A.},
  title={A factor not anti-isomorphic to itself},
  journal={Bull. London Math. Soc.},
  volume={7},
  date={1975},
  pages={171--174},
  issn={0024-6093},
  review={\MR {0435864 (55 \#8815)}},
}

\bib{Crocker-Kumjian-Raeburn-Williams:Brauer}{article}{
  author={Crocker, D.},
  author={Kumjian, A.},
  author={Raeburn, I.},
  author={Williams, D. P.},
  title={An equivariant Brauer group and actions of groups on $C^*$-algebras},
  journal={J. Funct. Anal.},
  volume={146},
  date={1997},
  number={1},
  pages={151--184},
  issn={0022-1236},
  review={\MR {1446378 (98j:46076)}},
  doi={10.1006/jfan.1996.3010},
}

\bib{Dixmier-Douady:Champs}{article}{
  author={Dixmier, J.},
  author={Douady, A.},
  title={Champs continus d'espaces hilbertiens et de $C^{\ast } $-alg\`ebres},
  language={French},
  journal={Bull. Soc. Math. France},
  volume={91},
  date={1963},
  pages={227--284},
  issn={0037-9484},
  review={\MR {0163182 (29 \#485)}},
}

\bib{Donovan-Karoubi}{article}{
  author={Donovan, P.},
  author={Karoubi, M.},
  title={Graded Brauer groups and $K$-theory with local coefficients},
  journal={Inst. Hautes \'Etudes Sci. Publ. Math.},
  number={38},
  date={1970},
  pages={5--25},
  issn={0073-8301},
  review={\MR {0282363 (43 \#8075)}},
}

\bib{Doran-Fell:Representations}{book}{
  author={Fell, J. M. G.},
  author={Doran, R. S.},
  title={Representations of $^*$\nobreakdash -algebras, locally compact groups, and Banach $^*$\nobreakdash -algebraic bundles. Vol. 1},
  series={Pure and Applied Mathematics},
  volume={125},
  note={Basic representation theory of groups and algebras},
  publisher={Academic Press Inc.},
  place={Boston, MA},
  date={1988},
  pages={xviii+746},
  isbn={0-12-252721-6},
  review={\MRref {936628}{90c:46001}},
}

\bib{Freed-Hopkins:Twisted_K2}{article}{
  author={Freed, D. S.},
  author={Hopkins, M. J.},
  author={Teleman, C.},
  title={Loop groups and twisted $K$-theory II},
  journal={J. Amer. Math. Soc.},
  volume={26},
  date={2013},
  number={3},
  pages={595--644},
  issn={0894-0347},
  review={\MR {3037783}},
  doi={10.1090/S0894-0347-2013-00761-4},
}

\bib{Giordano:Antiautomorphismes}{article}{
  author={Giordano, T.},
  title={Antiautomorphismes involutifs des facteurs de von Neumann injectifs. I},
  language={French},
  journal={J. Operator Theory},
  volume={10},
  date={1983},
  number={2},
  pages={251--287},
  issn={0379-4024},
  review={\MR {728909 (85h:46087)}},
}

\bib{Hilsum-Skandalis:Morphismes}{article}{
  author={Hilsum, M.},
  author={Skandalis, G.},
  title={Morphismes \(K\)\nobreakdash -orient\'es d'espaces de feuilles et fonctorialit\'e en th\'eorie de Kasparov \textup (d'apr\`es une conjecture d'A. Connes\textup )},
  language={French, with English summary},
  journal={Ann. Sci. \'Ecole Norm. Sup. (4)},
  volume={20},
  date={1987},
  number={3},
  pages={325--390},
  issn={0012-9593},
  review={\MRref {925720}{90a:58169}},
}

\bib{Husemoller:Fibre}{book}{
  author={Husem\"oller, D.},
  title={Fibre bundles},
  series={Graduate Texts in Mathematics},
  volume={20},
  edition={3},
  publisher={Springer-Verlag},
  place={New York},
  date={1994},
  pages={xx+353},
  isbn={0-387-94087-1},
  review={\MR {1249482 (94k:55001)}},
}

\bib{Jones:anti-isom}{article}{
  author={Jones, V. F. R.},
  title={A ${\rm II}_{1}$ factor anti-isomorphic to itself but without involutory antiautomorphisms},
  journal={Math. Scand.},
  volume={46},
  date={1980},
  number={1},
  pages={103--117},
  issn={0025-5521},
  review={\MR {585235 (82a:46075)}},
}

\bib{Karoubi:Twisted}{article}{
  author={Karoubi, M.},
  title={Twisted $K$-theory---old and new},
  conference={ title={$K$-theory and noncommutative geometry}, },
  book={ series={EMS Ser. Congr. Rep.}, publisher={Eur. Math. Soc., Z\"urich}, },
  date={2008},
  pages={117--149},
  review={\MR {2513335 (2010h:19010)}},
  doi={10.4171/060-1/5},
}

\bib{Kasparov:Operator_K}{article}{
  author={Kasparov, G. G.},
  title={The operator \(K\)\nobreakdash -functor and extensions of \(C^*\)\nobreakdash -algebras},
  language={Russian},
  journal={Izv. Akad. Nauk SSSR Ser. Mat.},
  volume={44},
  date={1980},
  number={3},
  pages={571--636, 719},
  issn={0373-2436},
  translation={ language={English}, journal={Math. USSR-Izv.}, volume={16}, date={1981}, number={3}, pages={513--572 (1981)}, },
  review={\MRref {582160}{81m:58075}},
}

\bib{Kumjian-Muhly-Renault-Williams:Brauer}{article}{
  author={Kumjian, A.},
  author={Muhly, P. S.},
  author={Renault, J. N.},
  author={Williams, D. P.},
  title={The Brauer group of a locally compact groupoid},
  journal={Amer. J. Math.},
  volume={120},
  date={1998},
  number={5},
  pages={901--954},
  issn={0002-9327},
  review={\MR {1646047 (2000b:46122)}},
}

\bib{Li:Real_algebras}{book}{
  author={Li, B.},
  title={Real operator algebras},
  publisher={World Scientific Publishing Co. Inc.},
  place={River Edge, NJ},
  date={2003},
  pages={xiv+241},
  isbn={981-238-380-8},
  review={\MR {1995682 (2004k:46100)}},
  doi={10.1142/9789812795182},
}

\bib{Moutuou:Twistings}{article}{
  author={Moutuou, E.M.},
  title={Twistings of $KR$ for Real groupoids},
  status={eprint},
  note={\arxiv {1110.6836}},
  date={2011},
}

\bib{Moutuou:Real.Cohomology}{article}{
  author={Moutuou, E. M.},
  title={On groupoids with involutions and their cohomology},
  journal={New York J. Math.},
  volume={19},
  year={2013},
  pages={729--792},
}

\bib{Moutuou:Thesis}{thesis}{
  author={Moutuou, E.M.},
  title={Twisted groupoid $KR$--Theory},
  type={Ph.D. thesis},
  institution={Universit\'e de Lorraine - Metz, and Universit\"at Paderborn},
  eprint={http://www.theses.fr/2012LORR0042},
  date={2012},
}

\bib{Parker:Brauer}{article}{
  author={Parker, E. M.},
  title={The Brauer group of graded continuous trace $C^*$-algebras},
  journal={Trans. Amer. Math. Soc.},
  volume={308},
  date={1988},
  number={1},
  pages={115--132},
  issn={0002-9947},
  review={\MR {946434 (89h:46080)}},
  doi={10.2307/2000953},
}

\bib{Phillips:anti-isom}{article}{
  author={Phillips, N. C.},
  title={A simple separable $C^*$-algebra not isomorphic to its opposite algebra},
  journal={Proc. Amer. Math. Soc.},
  volume={132},
  date={2004},
  number={10},
  pages={2997--3005 (electronic)},
  issn={0002-9939},
  review={\MR {2063121 (2005b:46127)}},
  doi={10.1090/S0002-9939-04-07330-7},
}

\bib{Phillips-Viola:anti-isom}{article}{
  author={Phillips, N. C.},
  author={Viola, M. G.},
  title={A simple separable exact ${\rm C}^*$-algebra not anti-isomorphic to itself},
  journal={Math. Ann.},
  volume={355},
  date={2013},
  number={2},
  pages={783--799},
  issn={0025-5831},
  review={\MR {3010147}},
  doi={10.1007/s00208-011-0755-z},
}

\bib{Raeburn-Williams:Morita_equivalence}{book}{
  author={Raeburn, I.},
  author={Williams, D. P.},
  title={Morita equivalence and continuous-trace $C^*$-algebras},
  series={Mathematical Surveys and Monographs},
  volume={60},
  publisher={American Mathematical Society},
  place={Providence, RI},
  date={1998},
  pages={xiv+327},
  isbn={0-8218-0860-5},
  review={\MR {1634408 (2000c:46108)}},
}

\bib{Rosenberg:Continuous-trace}{article}{
  author={Rosenberg, J.},
  title={Continuous-trace algebras from the bundle theoretic point of view},
  journal={J. Austral. Math. Soc. Ser. A},
  volume={47},
  date={1989},
  number={3},
  pages={368--381},
  issn={0263-6115},
  review={\MR {1018964 (91d:46090)}},
}

\bib{Saltman:Azumaya}{article}{
  author={Saltman, David J.},
  title={Azumaya algebras with involution},
  journal={J. Algebra},
  volume={52},
  date={1978},
  number={2},
  pages={526--539},
  issn={0021-8693},
  review={\MR {495234 (80a:16013)}},
  doi={10.1016/0021-8693(78)90253-3},
}

\bib{Schroder:KR-theory}{book}{
  author={Schr{\"o}der, H.},
  title={$K$-theory for real $C^*$-algebras and applications},
  series={Pitman Research Notes in Mathematics Series},
  volume={290},
  publisher={Longman Scientific \& Technical},
  place={Harlow},
  date={1993},
  pages={xiv+162},
  isbn={0-582-21929-9},
  review={\MR {1267059 (95f:19006)}},
}

\bib{Tu:Twisted_Poincare}{article}{
  author={Tu, J.-L.},
  title={Twisted $K$\nobreakdash -theory and Poincar\'e duality},
  journal={Trans. Amer. Math. Soc.},
  volume={361},
  date={2009},
  number={3},
  pages={1269--1278},
  issn={0002-9947},
  review={\MRref {2457398}{}},
}

\bib{Tu-Xu-Laurent-Gengoux:Twisted_K}{article}{
  author={Tu, J.-L.},
  author={Xu, P.},
  author={Laurent-Gengoux, C.},
  title={Twisted $K$\nobreakdash -theory of differentiable stacks},
  language={English, with English and French summaries},
  journal={Ann. Sci. \'Ecole Norm. Sup. (4)},
  volume={37},
  date={2004},
  number={6},
  pages={841--910},
  issn={0012-9593},
  review={\MRref {2119241}{2005k:58037}},
}

\bib{Wall:Brauer}{article}{
  author={Wall, C. T. C.},
  title={Graded Brauer groups},
  journal={J. Reine Angew. Math.},
  volume={213},
  date={1963/1964},
  pages={187--199},
  issn={0075-4102},
  review={\MR {0167498 (29 \#4771)}},
}

 \end{biblist}\end{bibdiv}
\end{document}